\documentclass[11pt]{amsart}
\usepackage[utf8]{inputenc}
\usepackage{enumerate}
\usepackage{amsmath}
\usepackage{amsfonts}
\usepackage{mathrsfs}
\usepackage{galois}
\usepackage[toc,page]{appendix}
\usepackage{xcolor}
\usepackage{import} 
\usepackage[colorlinks, linkcolor = blue, anchorcolor = blue, citecolor = blue]{hyperref}
\usepackage{eso-pic}
\usepackage{amssymb}
\usepackage{soul}
\usepackage{graphicx}
\usepackage{mathtools}
\usepackage{lipsum}
\usepackage{adjustbox}

\newcommand{\cB}{\mathcal{B}}
\newcommand{\cC}{\mathcal{C}}

\newcommand{\cF}{\mathcal{F}}
\newcommand{\cG}{\mathcal{G}}
\newcommand{\cH}{\mathcal{H}}

\newcommand{\cM}{\mathcal{M}}

\newcommand{\cP}{\mathcal{P}}

\newcommand{\cS}{\mathcal{S}}

\newcommand{\cU}{\mathcal{U}}
\newcommand{\cV}{\mathcal{V}}
\newcommand{\cW}{\mathcal{W}}

\newcommand{\cZ}{\mathcal{Z}}

\newcommand{\bB}{\mathbf{B}}
\newcommand{\bC}{\mathbf{C}}

\newcommand{\bF}{\mathbf{F}}

\newcommand{\bI}{\mathbf{I}}

\newcommand{\bL}{\mathbf{L}}
\newcommand{\bM}{\mathbf{M}}

\newcommand{\RP}{{\mathbb{RP}}}
\newcommand{\R}{\mathbb R}

\newcommand{\Z}{\mathbb Z}
\newcommand{\N}{\mathbb N}
\newcommand{\Zc}{\mathcal Z}

\newcommand{\Fc}{\mathcal F}
\newcommand{\Pc}{\mathcal P}

\newcommand{\Vc}{\mathcal V}
\newcommand{\Lb}{\mathbf L}
\newcommand{\Fb}{\mathbf F}
\newcommand{\Mb}{\mathbf M}
\newcommand{\id}{\mathrm {id}}

\newcommand{\area}{\mathrm{area}}
\newcommand{\vol}{\mathrm{vol}}

\newcommand{\dmn}{\mathrm{dmn}}
\renewcommand{\index}{\mathrm{index}}

\renewcommand{\tilde}{\widetilde}

\newcommand{\spt}{\operatorname{spt}}

\newcommand{\x}{\times}

\newcommand\norm[1]{\lVert#1\rVert}
\newcommand\HC{\mathcal{H}}

\numberwithin{equation}{section}
\newtheorem{thm}{Theorem}[section]
\newtheorem{cor}[thm]{Corollary}
\newtheorem{prop}[thm]{Proposition}
\newtheorem{lem}[thm]{Lemma}

\newtheorem{claim}[thm]{Claim}

\theoremstyle{definition}
\newtheorem{defn}[thm]{Definition}

\newtheorem{rmk}[thm]{Remark}
\newtheorem{nota}[thm]{Notation}

\newtheorem*{question*}{Question}

\title[Strong multiplicity one theorems]{Strong multiplicity one theorems and homological min-max theory}
\author{Adrian Chun-Pong Chu}\author{Yangyang Li}

\address{The University of Chicago, Department of Mathematics, Eckhart Hall,
5734 S University Ave,
Chicago, IL, 60637}
\email{acpc@uchicago.edu}
\address{The University of Chicago, Department of Mathematics, Eckhart Hall,
5734 S University Ave,
Chicago, IL, 60637}
\email{yangyangli@uchicago.edu}

\usepackage[style=alphabetic, backend=biber, sorting=nyt, url=false ]{biblatex}
\begin{document}

\maketitle

\begin{abstract}
It was asked by Marques-Neves which min-max $p$-widths of the unit $3$-sphere lie strictly between $2\pi^2$ and $8\pi$. We show that the 10th to the 13th widths do. More generally, we prove stronger versions of  X. Zhou's multiplicity one theorem.

\end{abstract}

\setcounter{tocdepth}{1}
\tableofcontents

\section{Introduction}

    For every closed smooth Riemannian manifold $(M,g)$ of dimension at least $2$, there is an associated non-decreasing sequence of positive real numbers, $\{\omega_p(M,g)\}_{p=1,2, \dots}$, which are called the min-max $p${\it-widths} of $M$. The heuristic definition can be described as follows.
    
    Using geometric measure theory, we consider the flat cycle space $\Zc$, which consists of ``all boundaryless geometric objects in $M$ of codimension $1$'', with coefficients in $\Z_2$. F. Almgren \cite{Alm62} showed that this space $\Zc$ is weakly homotopy equivalent to $\RP^\infty$, and consequently, its cohomology ring in $\Z_2$-coefficients is $\Z_2[\bar \lambda]$, which is generated by an order-$2$ element $\bar \lambda$. Now, a continuous map $\Phi$ from some finite simplicial complex into $\Zc$ is called a {\it $p$-sweepout} if the pullback  $\Phi^*(\bar \lambda^p)$ is non-zero. The $p$-width of $M$ is then defined by
    \[
        \omega_p(M,g)\coloneqq\inf_{\substack{p\textrm{-sweepout}\\\Phi:X\to\Zc}}\;\sup_{x\in X}\area(\Phi(x))\,.
    \]
    Note that sometimes the $g$ in $\omega_p(M,g)$ is omitted. Remarkably, the $p$-widths also follow a Weyl law \cite{LMN18}, similar to the spectrum of the Laplace-Beltrami operator on $M$.

    For the unit $3$-sphere $S^3$, it is well-known that the first four widths are equal to $4\pi$.
    Building upon the resolution of the Willmore conjecture by Marques-Neves \cite{MN14}, C. Nurser \cite{Nur16} showed that 
    \[
        \omega_5(S^3)=\omega_6(S^3)=\omega_7(S^3)=2\pi^2,\;\;2\pi^2<\omega_9(S^3)<8\pi, \;\;\omega_{13}(S^3)\leq 8\pi\,.
    \]
    Then,  Marques-Neves posed the following question \cite[\S 9]{MN17}:
        Which $p$-widths of $S^3$ lie {\it strictly} between $2\pi^2$ and $8\pi$? 
    
    Recently, F. Marques \cite{Mar23} proved that $\omega_8(S^3)$ is $2\pi^2$. In this paper, we prove that the $10$th to the $13$th widths of $S^3$ lie strictly between $2\pi^2$ and $8\pi$, by establishing the following theorem.
    \begin{thm}\label{thm_13_width}
        $\omega_{13}(S^3) < 8\pi$ for the unit $3$-sphere $S^3$.
    \end{thm}
    Note that it is still open whether the $14$-width of $S^3$ is also strictly less than $8\pi$.

    Let us explain our motivation for pursuing  the improvement from $\omega_{13}(S^3)\leq 8\pi$ to Theorem \ref{thm_13_width}. In recent years, the $p$-widths have played a crucial role in constructing minimal hypersurfaces using the Almgren-Pitts min-max theory \cite{MN17,IMN18,Zho20,MN21,YangyangLi20_improved_morse_index}. In particular, they are crucial in Song's proof \cite{Son23} of Yau's conjecture regarding the existence of infinitely many immersed closed minimal surfaces in 3-manifolds \cite[p.689]{Yau82}. In min-max theory, one subtle feature is that the minimal hypersurfaces obtained may have multiplicities. In a closed Riemannian manifold $(M^{n+1},g)$ with $3\leq n+1\leq 7$, the min-max theory yields a collection $\{\Sigma_1, \dots,\Sigma_N\}$ of disjoint, closed, smooth, embedded, minimal hypersurfaces accompanied by a set $\{m_1, \dots, m_N\}$ of positive integers. They constitute a varifold
    \begin{equation}\label{eq:emb_min_cycle}
        m_1|\Sigma_1|+ \dots +m_N|\Sigma_N|
    \end{equation}
    with a mass of $\omega_p(M)$.
    Note that any varifold of the form (\ref{eq:emb_min_cycle}) is called an {\it embedded minimal cycle.}
    
    Regarding the result $\omega_{13}(S^3)\leq 8\pi$, C. Nurser constructed an explicit $13$-sweepout such that the supremum of the area is attained by multiplicity-two equatorial $2$-spheres. In our proof of Theorem \ref{thm_13_width}, we show that such a sweepout cannot be optimal, thereby leading to a strict inequality of the $13$-width. 

    For other explicit computations of widths, it is worth noting that for the unit round $2$-sphere,  Aiex showed that the first three widths are $2\pi$ and the fourth to the eighth are $4\pi$ \cite{Aie19}, and then Chodosh-Mantoulidis showed that each $p$-width is given by $2\pi\lfloor\sqrt{p}\rfloor$ \cite{CM23}.
    Readers interested can refer to \cite{Luna19WidthRealProjective, BatistaLima22MinMaxWidth, Chu22,Don22,BatistaLima23Lens,Zhu23}. 
    
    Furthermore, our techniques lead to more general multiplicity one theorems, explained in the following. When the ambient manifold has a bumpy metric or positive Ricci curvature, X. Zhou \cite{Zho20} proved the existence of a multiplicity-one, two-sided, minimal hypersurface with the area given by the $p$-width (though this does not hold for general metrics, as shown by Wang-Zhou \cite{wangZhou2022higherMulti}). In our paper, we strengthen Zhou's multiplicity one theorem as follows.

    Let $X$ be a {\it pure} finite simplicial $k$-complex (which means $X$ is a finite union of $k$-simplexes) with empty boundary. 
    For any  $\Fb$-continuous map $\Phi: X \to\Zc_n(M;\Z_2)$, let $\cH(\Phi)$ be the set of all maps $\Phi':X' \to\Zc_n(M;\Z_2)$, where $X'$ is any pure finite simplicial $k$-cycle, such that there exists some $\Fb$-continuous map  $\Psi:W \to\Zc_n(M;\Z_2)$, where $W$ is a pure finite simplicial $(k+1)$-complex, with 
    $$\partial W=X\sqcup X', \quad \Psi|_{X}=\Phi,\quad \Psi|_{X'}=\Phi'.$$ Heuristically, $\cH(\Phi)$ is the set of all maps {\it homologous} to $\Phi$. Then, we define the {\it width} of $\cH(\Phi)$ by
    \[
        \bL(\cH(\Phi))\coloneqq \inf_{\Phi'\in\cH(\Phi)} \sup_{x \in \dmn(\Phi')} \bM(\Phi'(x))>0\,.
    \] 

    \begin{thm}[Strong multiplicity one theorem I]\label{thm:strong_multi_one_I}
       Consider a closed Riemannian manifold $(M^{n+1},g)$, where $3\leq n+1\leq 7$, equipped with a bumpy metric or a metric of positive Ricci curvature. Let $\Phi: X \to\Zc_n(M; \Z_2)$ be an $\Fb$-continuous map,  where $X$ is a pure finite simplicial cycle, such that $\bL(\cH(\Phi))>0$. Then for every minimizing sequence of the width $\bL(\cH(\Phi))$, its critical set contains some varifold induced by a multiplicity one, smooth, embedded, minimal hypersurface.
    \end{thm}

    Precise definitions of the terminologies will be provided in \S \ref{thm_restrictive_min_max}. Moreover, Theorem \ref{thm_13_width} would be a direct consequence of Theorem \ref{thm:strong_multi_one_I}. We will also prove another version of the strong multiplicity one theorem:  Essentially, the difference between Theorem \ref{thm:strong_multi_one_I} and \ref{thm:strong_multi_one_II} is  $\forall\;\exists$ versus $\exists\;\forall$. 

    \begin{thm}[Strong multiplicity one theorem II]\label{thm:strong_multi_one_II}
       Consider a closed Riemannian manifold $(M^{n+1},g)$, where $3\leq n+1\leq 7$, equipped with a bumpy metric or a metric of positive Ricci curvature. Let $\Phi:X \to\Zc_n(M; \Z_2)$ be an $\Fb$-continuous map, where  $X$ is a pure simplicial cycle, such that $\bL(\cH(\Phi))>0$. Then there exists a pulled-tight minimizing sequence $(\Phi_i)_i$ for the width $\bL(\cH(\Phi))$ such that every embedded minimal cycle in the critical set $\bC((\Phi_i)_i)$ is induced by some multiplicity one, smooth, embedded, minimal hypersurface.
    \end{thm}
    
    \begin{rmk} In fact, as we will see in the proof, in both  Theorem \ref{thm:strong_multi_one_I} and \ref{thm:strong_multi_one_II}, the minimal hypersurfaces mentioned are actually the boundary of some smooth domain.
    \end{rmk}

    It is worth highlighting that very recently, Wang-Zhou \cite{WangZhou23FourMinimalSpheres} established a multiplicity one theorem in the Simon-Smith min-max setting, which yielded four embedded, minimal $2$-spheres in every $S^3$ with a bumpy metric or positive Ricci curvature. This result builds
    on the their work \cite{WangZhou2025improved} on the regularity theory for multiple membranes problem, and also Sarnataro-Stryker's work \cite{SarnataroStryker2023Optimal} on the regularity for minimizers of the prescribed mean curvature functional.



    \subsection{Main ideas}\label{sect:mainIdeas}
    
        In this paper, we employ various variants of the min-max theory. Instead of doing min-max over the class of maps homotopic to a given sweepout, we will, for example, consider the class of maps {\it homologous} to a given sweepout, or restrict the class of maps by imposing an upper bound on mass (a technique previously developed by the second author in \cite{YangyangLi20_improved_morse_index}). These  min-max theorems are detailed in \S \ref{sect:min-max}. 
        Let us outline their roles in proving Theorem \ref{thm_13_width} and \ref{thm:strong_multi_one_II} here. Note that, while Theorem \ref{thm_13_width} will be a direct consequence of Theorem \ref{thm:strong_multi_one_I}, for the sake of clarity in presentation, we would just sketch the proof of Theorem \ref{thm_13_width}.
    
        \subsubsection{Theorem \ref{thm_13_width}} \label{subsubsect:Thm1}
            Let us sketch the proof of $\omega_{13}(S^3)<8\pi$. By C. Nurser \cite{Nur16}, there exists a 13-sweepout $\Phi_{0}$  such that $\Mb\circ\Phi_0\leq 8\pi$ and its critical set consists of multiplicity two equatorial $2$-spheres. Suppose by contradiction that $\omega_{13}(S^3)=8\pi$.

            Let us take a sequence of bumpy metrics $g_1,g_2,...$ that tend to the round metric $\bar g$. Take a sequence of positive numbers $\delta_1,\delta_2,...\to 0$. We will run for each $i$ a {\it restrictive homological min-max with a mass upper bound $8\pi+\delta_i$} under the metric $g_i$. More precisely, for each $i$, we will consider the class $\HC_i$ of all maps that are homologous to $\Phi_0$ through some ``cobordism" (in the space of cycles) whose mass is bounded from above by $8\pi+\delta_i$, and then do min-max in this class. The min-max width $L_i$ will tend to $8\pi$.

            By X. Zhou's multiplicity one theorem, for each $g_i$, the width $L_i$ corresponds to some min-max minimal hypersurface $\Sigma_i$ with multiplicity one. And by Marques-Neves \cite{MN21}, we can assume that $\Sigma_i$ is the only minimal hypersurface in $(M,g_i)$ with area $L_i$, even if multiplicity is allowed. Since $\bar g$ has positive Ricci curvature, by Sharp's compactness theorem we can assume $\Sigma_i$ tends to some minimal hypersurface $\Sigma$ smoothly, with multiplicity one.
 
            Now, let $\cC^1$ (resp. $\cC^{>1}$) be the set of $\bar g$-minimal hypersurfaces with multiplicity one (resp. greater than one) and area $8\pi$. For some sufficiently large $i$, we can  choose an ``optimal" 13-sweepout $\Phi_i$ in $\HC_i$, and an ``optimal"  cobordism $\Psi_i$ between $\Phi_0$ and $\Phi_i$, so that by \cite[Theorem 4.7]{MN21} we would have, heuristically:
            \begin{itemize}
                \item If $\Phi_i(x)$ has large area (i.e. has area greater than $L_i-\epsilon$ for some $\epsilon>0$), then $\Phi_i(x)$ is close to $\Sigma_i$, and thus to $\Sigma\in \cC^1$.
                \item If $\Psi_i(x)$ has large area, then $\Psi_i(x)$ is close to $\cC^1\cup\cC^{>1}$. 
            \end{itemize}

            We will now derive a contradiction by constructing some $\Xi\in\HC_i$ such that $\Mb_{g_i}\circ \Xi<L_i$. For simplicity, {\it let us suppose for now there is a Morse index upper bound for every element $\cC^1$ and $\cC^{>1}$}. Then, since $\bar g$ has a positive Ricci curvature, the closure of $\cC^1$ and  $\cC^{>1}$ are separated from each other by Sharp's compactness. Thus,  viewing $\Phi_i, \Psi_i, \cC^1$, and $\cC^{>1}$ all as sets of currents by abuse of notation, we can choose a subset $A\subset \Psi_i$ that is away from $\cC^{>1}$ and contains all elements of $\Psi_i$ which are close to $\cC^1$: See Figure \ref{fig:introPic}. Now,   we remove from $\Phi_i$ the part $\Phi_i\cap A$, and glue back in a cap $\partial A\backslash \Phi_i $. The new sweepout is our $\Xi$: See Figure \ref{fig:introPic}. Now, $\Xi$ is away from  $\cC^1$ and $\cC^{>1}$. Thus, by the two bullet points in the last paragraph, $\Mb_{g_i}\circ \Xi<L_i$.  Contradiction arises. 

            \begin{figure}[h]
                \centering
                \makebox[\textwidth][c]{\includegraphics[width=5.5in]{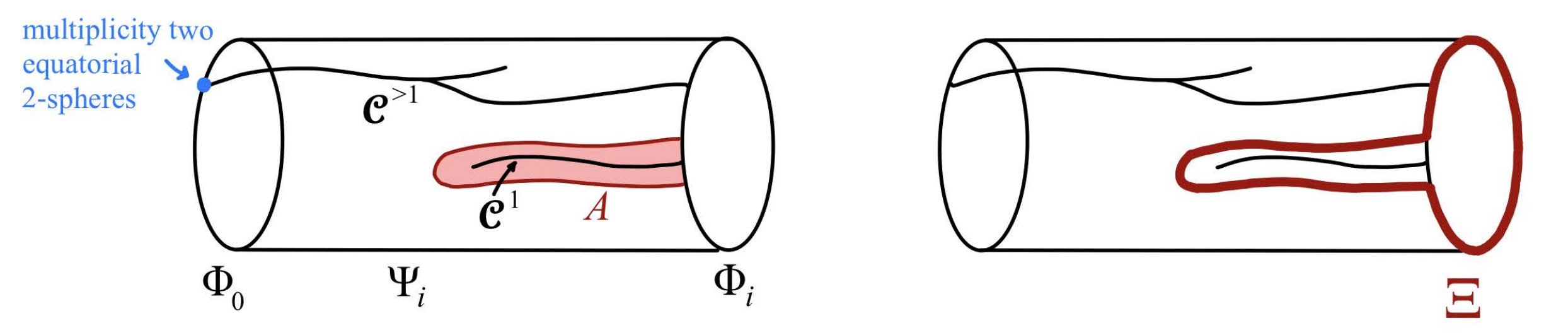}}
                \caption{}
                \label{fig:introPic}
            \end{figure}

            In reality, {\it we do not have an upper bound on Morse index} for elements of $\cC^1$ and  $\cC^{>1}$. The sets $\cC^1$ and  $\cC^{>1}$ will be defined in a different way, using Pitts' notion of almost-minimizing, and the fact that they are separated would be deduced by examining properties of annular replacements.

        \subsubsection{Theorem \ref{thm:strong_multi_one_II}} \label{subsubsect:Thm3}
            Let $\cC^1$ (resp. $\cC^{>1}$) be the set of $\bar g$-minimal hypersurfaces with multiplicity one (resp. greater than one) and area $\omega_p(M)$. As above, for simplicity, we assume these two sets are separated.
            In the spirit of \cite[Theorem 4.7]{MN21}, choose an ``optimal" $p$-sweepout $\Phi$ for $\omega_p(M)$ such that if $x$ is such that $\Phi(x)$ has high area, then $\Phi(x)$ is close to $\cC^1\cup\cC^{>1}$. Let $A$ be the set of  $x$ such that $\Phi(x)$ is in fact close to $\cC^{>1}$. We will do a relative, homological min-max process, by considering {\it the set $\cH$ of maps  that is homological to $\Phi|_{A}$ relative to $\Phi|_{\partial A}$.} 

            To prove Theorem \ref{thm:strong_multi_one_II}, it suffices to show that the min-max width for the relative, homological min-max class $\cH$ is less than $\omega_p(M)$, because then we can ``lower" $\Phi|_A$, and thus $\Phi$, away from $\cC^{>1}$. Suppose the otherwise, so that $\Phi|_A$ is an ``optimal" sweepout for $\cH$ such that elements of high area are close to $\cC^{>1}$. Then we are in a situation analogous to the proof of Theorem \ref{thm_13_width}, where the 13-sweepout $\Phi_0$  is an optimal sweepout whose critical set lie in $\cC^{>1}$. Thus, we can argue as in the proof of Theorem \ref{thm_13_width} to get a contradiction.

    \subsection{Organization}

        In \S \ref{sect:prelim}, we include some preliminary materials. In \S \ref{sect:mr_am} we introduce the notion of $(m,r)$-almost minimizing varifold. Results in these two sections are rather technical: They are mostly adaptation of existing min-max content to the restrictive setting (i.e. with a mass upper bound) or the setting where the ambient metric is perturbed. In \S\ref{sect:deformation}, we prove  results related to deformation of sweepouts.

        In \S \ref{sect:min-max}, we prove a restrictive (homotopic) min-max theorem, and a restrictive homological min-max theory.
        The    key new ideas of this paper (as outlines in \S \ref{sect:mainIdeas}) will be begin in \S \ref{sect:proof_main_thm_I}, where we prove Theorem \ref{thm:strong_multi_one_I}. In   \S \ref{sect:proof_main_thm}, we prove Theorem \ref{thm_13_width} using Theorem \ref{thm:strong_multi_one_I}.  In \S\ref{sect:proof_main_thm_II}, we prove  Theorem \ref{thm:strong_multi_one_II}. In \S \ref{sect:technical}, we prove some technical results used in  \S \ref{sect:proof_main_thm_I}.

    \subsection*{Acknowledgement}
        We would like to thank Andr\'e Neves for the helpful discussions. The second author was partially supported by the AMS-Simons travel grant.

\part{Preparation work}

\section{Preliminaries}\label{sect:prelim}

    Throughout this paper, unless specified otherwise, the ambient Riemannian manifolds $(M^{n+1}, g)$ we consider will always be smooth and closed, with $3\leq n+1\leq 7$.
    
    \subsection{Notations} \label{sect:nota}
        \begin{itemize}
            \item $\bI_k(M;\Z_2)$: the set of integral $k$-dimensional currents in $M$  with $\Z_2$-coefficients.
            \item $\Zc_{k}( M;\Z_2)\subset \bI_k(M;\Z_2)$: the subset that consists of elements $T$ such that $T=\partial Q$ for some $Q\in\bI_{k+1}( M;\Z_2)$ (such $T$ are also called {\it flat $k$-cycles}). 
            \item $\Zc_{k}( M;\nu;\Z_2)$ with $\nu=\cF,\bF,\bM$: the set $\Zc_{k}( M;\Z_2)$ equipped with the three common topologies given respectively by the  {\it flat} metric $\Fc$,  the {\it $\Fb$-metric}, and the  {\it mass} $\Mb$ (see, for example, the survey \cite{MN20}). 
            \item $\mathcal{V}_n(M)$ or $\cV(M)$: the closure, in the varifold weak topology, of the space of $n$-dimensional rectifiable varifolds in $M$.
            \item $\norm{V}$: the Radon measure  induced on $M$ by  $V\in \mathcal{V}_n(M)$.
            \item For any $a$, the varifold topology on $\{V\in \Vc_n(M):\norm{V}(M)\leq a\}$ can be induced by an {\it $\Fb$-metric} defined by Pitts in \cite[p.66]{Pit81}.
            \item $|T| \in \mathcal{V}_n(M)$: the varifold induced by a current $T\in \Zc_{k}( M;\Z_2)$, or a submanifold $T$.
            \item In the same spirit, given a map $\Phi$ into $\Zc_n(M;\Z_2)$, the associated map into $ \mathcal{V}_n(M)$ is denoted $|\Phi|$.
            \item $\spt(\cdot)$: the support of a current or a varifold.
            \item $\bB^\nu_\epsilon(\cdot)$:  the open $\epsilon$-neighborhood  of an element or  a subset in $\Zc_n(M;\nu;\Z_2)$.
            \item $\bB^{\Fb}_\epsilon(\cdot)$: the open $\epsilon$-neighborhood  of an element or  a subset of $\cV_n(M)$ under the $\Fb$-metric.
            \item $I(1,j)$: the cubical complex on $I\coloneqq[0,1]$ whose 1-cells and 0-cells are respectively 
            \[
                [0,1/3^j],[1/3^j,2/3^j],\dots,[1-1/3^j,1]\;\;\textrm{ and }\;\; [0],[1/3^j],[2/3^j],\dots,[1].
            \]
            \item $I(m,j)$: the  cubical complex structure
            \[I(m,j)=I(1,j)\otimes\dots\otimes I(1,j)\;\;(m\textrm{ times})\] 
            on $I^m$.
            \item $X_q$: the set of $q$-cells of $X$.
            \item $X(q)$ for a cubical subcomplex of $I(m,j)$: the subcomplex of $I(m, j + q)$ with support $X$.
            \item $\mathbf{f}(\Phi)$: the {\it fineness} of a map $\Phi:X_0\to\Zc_n(M;\Z_2)$,
            \[
                \sup\{\Mb(\Phi(x)-\Phi(y)):x,y\textrm{ belong to some common cell}\}.
            \]
            \item $\mathbf{n}(i,j) : I(m,i)_0 \to I(m,j)_0$ for $j \leq i$: the map such that $\mathbf{n}(i,j)(x)$ is the closest vertex in $I(m,j)_0$ to $x$.
            \item $\Gamma^\infty(M)$: the set of smooth Riemannian metrics on $M$.
            \item $B_g(p,r)$: the open $r$-neighborhood of a point $p$ in metric $g$.
            \item  Suppose $W$ is a finite simplicial $(k+1)$-complex, and we have already defined two simplicial $k$-complexes $W^\alpha$ and $W^\omega$ such that 
            $\partial W=W^\alpha + W^\omega$ (with $\Z_2$ coefficients). Then given any map $\Psi: W \to \mathcal{Z}_n(M;  \mathbb{Z}_2)$, we  denote $\Psi^\alpha \coloneqq \Psi|_{W^\alpha }$ and $\Psi^\omega\coloneqq\Psi|_{W^\omega}$.
        \end{itemize}
    
        Throughout this paper, for the sake of simplicity, we would consider a complex and its underlying space as identical. Moreover, all simplicial or cubical complexes are assumed to have $\Z_2$-coefficients.

    \subsection{$p$-width}\label{subsect:widths}
        By the Almgren isomorphism theorem \cite{Alm62} (see also \cite[\S 2.5]{LMN18}),  when equipped with the flat topology, $\Zc_{n }( M;\Z_2)$ is weakly homotopic equivalent to $\R\mathbb P^\infty$. Thus we can denote its cohomology ring by $\Z_2[\bar\lambda]$. 
        
        \begin{defn}
            Let $\Pc_p$ be the set of all $\Fb$-continuous maps $\Phi:X\to \Zc_n(M;\Z_2)$, where $X$ is a finite simplicial complex, such that $\Phi^*(\bar\lambda^p)\ne 0$. Elements of $\Pc_p$ are called {\it $p$-sweepouts.}
        \end{defn}

        \begin{rmk}
            Note that every finite cubical complex is homeomorphic to a finite simplicial complex and vice versa (see \cite[\S 4]{BP02}). So when $X$ is a finite cubical complex in above, the notion of $p$-sweepout still makes sense. 
        \end{rmk}
        \begin{defn}
            Denoting by $\dmn(\Phi)$ the domain of $\Phi$, the {\it $p$-width} of $(M,g)$ is defined by 
            \[
                \omega_p(M,g)\coloneqq\inf_{\Phi\in\Pc_p}\sup_{x\in \dmn(\Phi)}\Mb(\Phi(x)).
            \]
        \end{defn}
        We may write $\omega_p(M)$ for $\omega_p(M,g)$ if no confusion is caused.
        \begin{defn}
            A sequence $(\Phi_i)_i$ in $\cP_p$ is called a  {\it minimizing sequence} for $\cP_p$, or the $p$-width $\omega_p(M)$,  if \[\limsup_{i\to\infty} \max_{x\in \dmn(\Phi_i)} \Mb(\Phi_i(x))=\omega_p(M).\] 
        \end{defn}

        For a minimizing sequence $(\Phi_i)_i$, we define its \textit{critical set} by
        \[
            \bC_g((\Phi_i)_i)\coloneqq\{V=\lim_j|\Phi_{i_j}(x_j)|:\{i_j\}_j\subset\mathbb N, x_j\in \dmn(\Phi_{i_j}), \| V\|_g(M)=\omega_p(M)\}\,.
        \]
        We will often omit the subscript $_g$ if no confusion is caused. This will also be the case for other variants of min-max theory in \S \ref{sect:min-max}.
        Now, a sequence is called \textit{pulled-tight} if every varifold in $\bC((\Phi_i)_i)$ is stationary.

        \begin{rmk}\label{def_equiv_width}
            There is an equivalent definition of $p$-widths from \cite[Remark 5.7]{MN21}: First, an $\Fc$-continuous map $\Phi:X\to\Zc_n(M;\Z_2)$ is said to have {\it no concentration of mass} if 
            \[
                \lim_{r\to 0}\sup_{x\in X,p\in M}\norm{\Phi(x)}(B_r(p))=0\,.
            \]
            Then, when defining the $p$-width, instead of using the collection $\Pc_p$, we use the collection of all $\Fc$-continuous maps $\Phi$ with no concentration of mass such that $\Phi^*(\bar\lambda^p)\ne 0$.
        \end{rmk}

    \subsection{Interpolations}
        In this subsection, we collect some interpolation results in the literature. Let $(M, g)$ be a closed manifold and $m$ be a positive integer.
    
        \begin{prop}[{\cite[Theorem~3.7]{MN17}}]\label{prop:almgren_ext}
            There exist positive constants $C_{\ref{prop:almgren_ext}} = C_{\ref{prop:almgren_ext}}(M, g, m)$ and $\delta_{\ref{prop:almgren_ext}} = \delta_{\ref{prop:almgren_ext}}(M, g, m)$ with the following property:
            
            If $X$ is a cubical subcomplex of $I(m, l)$ for some $l \in \mathbb{N}^+$ and 
            \[
                \phi: X_0 \to \mathcal{Z}_n(M; \Z_2)
            \]
            has $\mathbf{f}(\phi) < \delta_{\ref{prop:almgren_ext}}$, then there exists a map (called the \textit{Almgren extension})
            \[
                \Phi: X \to \mathcal{Z}_n(M; \Mb_g; \Z_2)
            \]
            continuous in the mass norm and satisfying
            \begin{enumerate}
                \item $\Phi(x) = \phi(x)$ for all $x \in X_0$;
                \item if $\alpha$ is some $j$-cell in $X_j$, then $\Phi$ restricted to $\alpha$ depends only on the values of $\phi$ assumed on the vertices of $\alpha$;
                \item $\sup\{\mathbf{M}(\Phi(x) - \Phi(y)) : x, y \text{ lie in a common cell of }X\} \leq C_{\ref{prop:almgren_ext}} \mathbf{f}(\phi)$.
            \end{enumerate}
        \end{prop}
    
        \begin{prop}[{\cite[Proposition~3.2]{MN21}}]\label{prop:homotopy_of_M_close}
            There exist positive constants $C_{\ref{prop:homotopy_of_M_close}} = C_{\ref{prop:homotopy_of_M_close}}(M, g, m)$ and $\delta_{\ref{prop:homotopy_of_M_close}} = \delta_{\ref{prop:homotopy_of_M_close}}(M, g, m)$ with the following property:
            
            If $X$ is a cubical subcomplex of $I(m, l)$ for some $l \in \mathbb{N}^+$ and two continuous maps
            \[
                \Phi_0, \Phi_1: X \to \mathcal{Z}_n(M; \Mb_g; \Z_2)
            \]
            satisfy
            \[
                \sup_{x \in X} \Mb_g(\Phi_0(x) - \Phi_1(x)) < \delta_{\ref{prop:homotopy_of_M_close}}\,,
            \]
            then there exists a homotopy 
            \[
                H: [0, 1] \times X \to \mathcal{Z}_n(M; \Mb_g; \Z_2)
            \] with $H(0, \cdot) = \Phi_0$ and $H(1, \cdot) = \Phi_1$ and such that
            \[
                \sup_{(t, x) \in [0, 1] \times X} \Mb_g(H(t,x) - \Phi_0(x)) \leq C_{\ref{prop:homotopy_of_M_close}} \sup_{x \in X} \Mb_g(\Phi_0(x) - \Phi_1(x))\,.
            \]
            In particular, for all $(t, x) \in [0, 1] \times X$,
            \[
                \Mb_g(H(t,x)) \leq \Mb_g(\Phi_0(x)) + C_{\ref{prop:homotopy_of_M_close}} \sup_{x \in X} \Mb_g(\Phi_0(x) - \Phi_1(x))\,.
            \]
        \end{prop}

    Moreover, we will need the following improved version of \cite[Proposition~3.8]{MN17}. Namely, in the original proposition, the interpolation map is continuous in the flat topology, but below we show that the interpolation map can made continuous in the mass norm.
    
        \begin{prop}[Improved version of {\cite[Proposition~3.8]{MN17}}]\label{prop:cts_Hotp_of_disc_Hotp}
            There exist  positive constants $\eta_{\ref{prop:cts_Hotp_of_disc_Hotp}} = \eta_{\ref{prop:cts_Hotp_of_disc_Hotp}}(M, g, m)$ and $C_{\ref{prop:cts_Hotp_of_disc_Hotp}} = C_{\ref{prop:cts_Hotp_of_disc_Hotp}}(M, g, m)$ with the following property:
    
            Suppose that $X$ is a cubical subcomplex in $I(m, q)$, and $\phi_0: X(l_0)_0 \to \Zc_n(M; \Z_2)$ is $(X, \Mb_g)$-homotopic to $\phi_1: X(l_1)_0 \to \Zc_n(M; \Z_1)$ through a discrete homotopy map 
            \[
                h: I(1, q + l)_0 \times X(l)_0 \to \Zc_n(M; \Z_2)
            \]
            with fineness $\mathbf{f}(h) < \eta_{\ref{prop:cts_Hotp_of_disc_Hotp}}$ and such that if $i = 0, 1$ and $x \in X(l)_0$, then
            \[
                h([i], x) = \phi_i(\mathbf{n}(q + l, q + l_i)(x))\,.
            \]
            Then the Almgren extensions
            \[
                \Phi_0, \Phi_1: X \to \Zc_n(M; \Mb_g; \Z_2)
            \]
            of $\phi_0$, $\phi_1$, respectively, are homotopic through a $\Mb_g$-continuous homotopy map
            \[
                H: [0, 1] \times X \to \Zc_n(M; \Mb_g; \Z_2)
            \]
            with $H(0, \cdot) = \Phi_0$ and $H(1, \cdot) = \Phi_1$. Furthermore, for all $(t, x) \in [0, 1] \times X$, there exists $(t_0, x_0) \in I(1, q + l)_0 \times X(l)_0$ such that $x$ and $x_0$ are in the same cell of $X(l)$, and 
            \[
                \Mb_g(H(t,x)) \leq \Mb_g(h(t_0, x_0)) + C_{\ref{prop:cts_Hotp_of_disc_Hotp}} \mathbf{f}(h)\,.
            \]
        \end{prop}
        \begin{proof}
            Set $\eta_{\ref{prop:cts_Hotp_of_disc_Hotp}} = \min\left(\frac{\delta_{\ref{prop:homotopy_of_M_close}}}{2C_{\ref{prop:almgren_ext}}}, \delta_{\ref{prop:almgren_ext}}\right)$ and $C_{\ref{prop:cts_Hotp_of_disc_Hotp}} = C_{\ref{prop:almgren_ext}} + C_{\ref{prop:homotopy_of_M_close}} \cdot 2C_{\ref{prop:almgren_ext}}$.
    
            For $i = 0, 1$, let $\phi'_i: X(l)_0 \to \Zc_n(M; \Z_2)$ be given by $\phi'_i(x) = h([i], x)$. Since 
            \[
                \mathbf{f}(\phi_i), \mathbf{f}(\phi'_i) \leq \mathbf{f}(h) < \eta_{\ref{prop:cts_Hotp_of_disc_Hotp}}\,,
            \]
            by Proposition \ref{prop:almgren_ext}, for $i = 0, 1$, let $\Phi_i, \Phi'_i: X \to \Zc_n(M; \Mb_g; \Z_2)$ be the Almgren extensions of $\phi_i, \phi'_i$, respectively, and it follows that 
            \[
                \Mb_g(\Phi_i(x) - \Phi'_i(x)) \leq 2C_{\ref{prop:almgren_ext}} \mathbf{f}(h) < 2C_{\ref{prop:almgren_ext}} \eta_{\ref{prop:cts_Hotp_of_disc_Hotp}} \leq \delta_{\ref{prop:homotopy_of_M_close}}\,.
            \]
            Hence, for $i = 0, 1$, by Proposition \ref{prop:homotopy_of_M_close}, there exists a $\Mb_g$-continuous 
            \[
                H'_i: [0, 1] \times X \to \Zc_n(M; \Mb_g; \Z_2)\,,
            \]
            with $H'_i(0, \cdot) = \Phi_i$ and $H'_i(1, \cdot) = \Phi'_i$, and such that for all $(t, x) \in (0, 1) \times X$,
            \[
            \begin{aligned}
                \Mb_g(H'_i(t,x)) &\leq \Mb_g(\Phi'_i(x)) + C_{\ref{prop:homotopy_of_M_close}} \cdot 2C_{\ref{prop:almgren_ext}} \mathbf{f}(h)\\
                    &\leq \Mb_g(\phi'_i(x_0)) + C_{\ref{prop:almgren_ext}} \mathbf{f}(\phi'_i) + C_{\ref{prop:homotopy_of_M_close}} \cdot 2C_{\ref{prop:almgren_ext}} \mathbf{f}(h)\\
                    &\leq \Mb_g(h([i], x_0)) + (C_{\ref{prop:almgren_ext}}+ C_{\ref{prop:homotopy_of_M_close}} \cdot 2C_{\ref{prop:almgren_ext}}) \mathbf{f}(h)\\
                    &\leq \Mb_g(h([i], x_0)) + C_{\ref{prop:cts_Hotp_of_disc_Hotp}} \mathbf{f}(h)\,
            \end{aligned}
            \]
            as long as $x_0 \in X(l)_0$ and $x$ are in the same cell of $X(l)$.
    
            By Proposition \ref{prop:almgren_ext} again, the Almgren extension
            \[
                H': [0, 1] \times X \to \Zc_n(M;\Mb_g;\Z_2)
            \]
            of $h$ is a homotopy between $\Phi'_0$ and $\Phi'_1$ such that for all $(t, x) \in (0, 1) \times X$,
            \[
                \Mb_g(H'(t,x)) \leq \Mb_g(h(t_0, x_0)) + C_{\ref{prop:almgren_ext}} \mathbf{f}(h) \leq \Mb_g(h(t_0, x_0)) + C_{\ref{prop:cts_Hotp_of_disc_Hotp}} \mathbf{f}(h)\,,
            \]
            as long as $(t_0, x_0) \in I(1, q + l)_0 \times X(l)_0$ and $(t, x)$ are in the same cell of $I(1, q + l) \times X(l)$.
    
            Finally, concatenating $H'_0, H'$ and $H'_1(1 - t, x)$, we obtain a homotopy
            \[
                H: [0, 1] \times X \to \Zc_n(M;\Mb_g;\Z_2)
            \]
            between $\Phi_0$ and $\Phi_1$. By the previous estimates, for all $(t, x) \in [0, 1] \times X$,
            \[
                \Mb_g(H(t,x)) \leq \Mb_g h(t_0, x_0) + C_{\ref{prop:cts_Hotp_of_disc_Hotp}} \mathbf{f}(h)\,.
            \]
            for some $(t_0, x_0) \in I(1, q + l)_0 \times X(l)_0$ where $x$ and $x_0$ are in the same cell of $X(l)$.
        \end{proof}
    
        \begin{prop}[Simplicial variant of {\cite[Proposition~3.7]{MN21}}]\label{prop:M_cts_approx_F_cts}
            Let $X$ be a finite simplicial complex, and $\Phi: X \to \cZ_n(M; \Fb_g; \Z_2)$ be a continuous map. Then for every $\varepsilon > 0$ there exists an $\Mb_g$-continuous map 
            \[
                \Phi': X \to \cZ_n(M; \Mb_g; \Z_2)
            \] and an $\bF_g$-continuous homotopy $H: [0, 1] \times X \to \cZ_n(M; \bF_g; \Z_2)$ with $H(0, \cdot) = \Phi$ and $H(1, \cdot) = \Phi'$, and such that
            \[
            \begin{aligned}
                \sup_{(t, x) \in (0, 1) \times X} \bF_g(H(t, x), \Phi(x)) &< \varepsilon\,,\\
                \sup_{(t, x) \in (0, 1) \times X} \bM_g(H(t, x)) &< \sup_{x \in X} \bM_g(\Phi(x)) + \varepsilon\,.
            \end{aligned}
            \]
        \end{prop}
        \begin{proof}
            By \cite[Chapter~4]{BP02}, the finite simplicial complex $X$ is homeomorphic to a cubical subcomplex of some $I^N$. It follows from that {\cite[Proposition~3.7]{MN21}} for every $\varepsilon > 0$, we have a desired homotopy map $H$ with
            \[
                \sup_{(t, x) \in (0, 1) \times X} \bF_g(H(t, x), \Phi(x)) < \varepsilon\,.
            \]
            Furthermore, by the definition of $\mathbf{F}_g$, for every $\varepsilon' > 0$, we can select even smaller $\varepsilon$ ensuring that the aforementioned inequality leads to
            \[
                \sup_{(t, x) \in (0, 1) \times X} |\bM_g(H(t, x)) - \bM_g(\Phi(x))| < \varepsilon'\,.
            \]
            This concludes the proof.
        \end{proof}

\section{\texorpdfstring{$(m, r)_{g}$}{(m,r)g}-almost minimizing varifolds}\label{sect:mr_am}

    Let $M^{n+1}$ ($3 \leq n+1 \leq 7$) be a closed smooth manifold, and $\Gamma^\infty(M)$ be the set of all the Riemannian metrics on $M$. Let us first recall the definitions of almost minimizing varifolds in \cite{Pit81, MN21}.

    \begin{defn}
        For each pair of positive numbers $\varepsilon, \delta$, an open subset $U \subset (M^{n+1}, g)$, and $T \in \Zc_n(M; \Z_2)$, an \textit{$(\varepsilon, \delta)$-deformation} of $T$ in $U$ is a finite sequence $(T_i)^q_{i = 0}$ in $\Zc_n(M; \Z_2)$ with
        \begin{enumerate}
            \item $T_0 = T$ and $\spt(T - T_i) \subset U$ for all $i = 1, \cdots, q$;
            \item $\Mb_g(T_i - T_{i - 1}) \leq \delta$ for all $i = 1, \cdots, q$;
            \item $\Mb_g(T_i) \leq \Mb_g(T) + \delta$ for all $i = 1, \cdots, q$;
            \item $\Mb_g(T_q) < \Mb_g(T) -\varepsilon$.
        \end{enumerate}

        We define $\mathfrak{a}_g(U; \varepsilon, \delta)$ to be the set of all flat cycles $T \in \Zc_n(M; \Z_2)$ that do not admit $(\varepsilon, \delta)$-deformations in $U$.
    \end{defn}

    \begin{defn}
        For an open set $U \subset (M^{n+1}, g)$, a varifold $V \in \Vc_n(M)$ is \textit{almost minimizing} if for every $\varepsilon > 0$, we can find $\delta > 0$ and
        \[
            T \in \mathfrak{a}_g(U; \varepsilon, \delta)
        \]
        with $\Fb_{g}(V, |T|) < \varepsilon$.
    \end{defn}

    In the following, for $m \in \N^+$, we set $I_m \coloneqq 3^{m3^m}$. We now define a quantitative almost minimizing condition inspired by Pitts' combinatorial arguments.

    \begin{defn}\label{defn_property_m}
        Let $m \in \N^+$ and $r \in \R^+$. A varifold $V$ in $(M,g)$ is \textit{$(m, r)_g$-almost minimizing} if the following holds. 
        For any point $p\in M$ and any $I_m$ concentric annuli $\{\mathrm{An}_g(p,r_i-s_i,r_i+s_i)\}^{I_m}_{i = 1}$, where $\{r_i\}$ and $\{s_i\}$ satisfy
          \begin{align*}
              r_i-2s_i&> 2(r_{i+1}+2s_{i+1}),\;i=1,\dots,I_m-1,\\
              r_{I_m}-2s_{I_m}&>0,\\
              r_1 + s_1 &< r,
          \end{align*}
        $V$ is almost minimizing \cite[Definition 3.1]{Pit81} in at least one of the annuli.
    \end{defn}
    \begin{rmk}
        The definition presented here closely resembles the Property $(m)$ in \cite{Yangyang_Li19} with the added requirement of a radius bound assumption.
    \end{rmk}
    
    \begin{thm}[{\cite[Theorem~4]{SS81}}]\label{thm:Schoen_Simon_reg}
        In $(M^{n+1}, g) (3 \leq n+1 \leq 7)$, if $V$ is $(m, r)_g$-almost minimizing and stationary, then $V$ is a stationary integral varifold whose support $\spt(V)$ is a smooth, embedded, closed, minimal hypersurface. 
    \end{thm}
    
    When the metric $g$ is obvious from the context, we might omit the subscript $_g$ for simplicity. 
    
    In the subsequent subsections, we establish a set of technical lemmas concerning the two essential concepts of the previous definitions: annular replacements and the almost minimizing property.

    \subsection{Annular replacement}
     
        First, we prove a lemma that bounds the $\mathbf{F}$-distance between two varifolds that differ in small annuli, where one of the varifolds is sufficiently close to a stationary varifold.
        
        \begin{lem}\label{lem:close_of_small_rep}
            Let $D, L \in \R^+$, $m \in \N^+$, and $K \subset \Gamma^\infty(M)$ be a compact set of $C^\infty$ Riemannian metrics on $M$.
            There exists $\eta_{\ref{lem:close_of_small_rep}} = \eta_{\ref{lem:close_of_small_rep}}(M, K, D, L, m) > 0$ for which the following hold.
            
            If $g, \bar g \in K$ and $V_0, V_1, V_2 \in \Vc_n(M)$ satisfy
            \begin{itemize}
                \item $\|g - \bar g\|_{C^\infty, \bar g} \leq \eta_{\ref{lem:close_of_small_rep}}$,
                \item $\|V_0\|_g(M), \|V_1\|_g(M), \|V_2\|_g(M) \leq 2L$,
                \item $\Fb_g(V_0, V_1) \leq \eta_{\ref{lem:close_of_small_rep}}$, 
                \item $V_0$ stationary in $(M, \bar g)$,
                \item $V_1 = V_2$ on $M \setminus (\overline{B}_{g}(p_1, 2\eta_{\ref{lem:close_of_small_rep}}) \cup \cdots \cup \overline{B}_{g}(p_t, 2\eta_{\ref{lem:close_of_small_rep}}))$ for some collection $\{p_1, \cdots, p_t\} \subset M, t \leq 3^{2m}$,
                \item $ \|V_1\|_g(M) - \eta_{\ref{lem:close_of_small_rep}} \leq \|V_2\|_g(M) \leq \|V_1\|_g(M) + \eta_{\ref{lem:close_of_small_rep}}$,
            \end{itemize}
            then $\Fb_g(V_1, V_2) < D / 2$.
        \end{lem}
        \begin{proof}
            The proof is essentially the same as that of \cite[Lemma~4.5]{MN21}.
    
            If this is false for any $\eta_i = \frac{1}{i}$ $(i = 1, 2, 3, \cdots)$, we obtain the existence of $g_i, \bar{g}_i \in K$, $V^i_0, V^i_1, V^i_2 \in \Vc_n(M)$ satisfying all the conditions in the lemma with $(\eta_i, g_i, \bar{g}_i, V^i_0, V^i_1, V^i_2)$ in place of $(\eta, g, \bar g, V_0, V_1, V_2)$ but
            \[
                \Fb_g(V^i_1, V^i_2) \geq D / 2\,.
            \]
    
            Up to a subsequence, by taking a limit, we obtain $g', \bar g' \in K$, $V'_0, V'_1, V'_2 \in \Vc_n(M)$ and $\{p_1, \cdots, p_{t'}\} \subset M$, $t' \leq 3^{2m}$, such that 
            \begin{enumerate}[(i)]
                \item $g' = \bar g'$;
                \item $V'_1 = V'_0$ is stationary;
                \item $\|V'_0\|(M) = \|V'_2\|(M)$ and $V'_0 = V'_2$ on $M \setminus (\overline{B}_{g'}(p_1, r) \cup \cdots \cup \overline{B}_{g'}(p_{t'}, r))$ for all $r > 0$;
                \item $\Fb_g(V'_0, V'_2) \geq D / 2$.
            \end{enumerate}
    
            It follows from the monotonicity formula for stationary varifolds that there exists $C > 0$ such that
            \[
                \|V'_0\|_{g'}(B_{g'}(q, r)) \leq Cr^n
            \]
            for all $r$ sufficiently small. Hence, (iii) implies that $V'_0 = V'_2$, which contradicts (iv).
        \end{proof}
    
        \begin{lem}\label{lem:balls_not_cover_SV}
            For any $m \in \mathbb N^+$ and $g \in \Gamma^\infty(M)$, there exists a positive constant $\eta_{\ref{lem:balls_not_cover_SV}} = \eta_{\ref{lem:balls_not_cover_SV}}(M, g, m)$ with the following property:
    
            If $V \in \Vc_n(M)$ is a stationary varifold on $(M, g)$, $\{p_i\}^t_{i = 1} \subset M$ with $t \leq 4\cdot 3^{2m}$ is a collection of points, and $\{r_i\}^t_{i = 1} \subset (0, 4\eta_{\ref{lem:balls_not_cover_SV}})$, then
            \[
                \spt(V) \setminus \bigcup^{t}_{i = 1} \overline{B}_g(p_i, r_i) \neq \emptyset\,.
            \]
        \end{lem}
        \begin{proof}
            Using the monotonicity formula, we can find positive constants $R > s > 0$ which only depend on $M, g$ and $m$, such that for any $p \in M$ and for any stationary varifold $V$,
            \[
                \|V\|_g(B_g(p,R)) \geq 8 \cdot 3^{2m} \|V\|_g(B_g(p,s))\,.
            \]
            We can set $\eta_{\ref{lem:balls_not_cover_SV}} \coloneqq s / 4$. 
            
            Now, assume for the sake of contradiction that there exists a stationary varifold $V$ such that
            \[
                \spt(V) \subset \bigcup^{t}_{i = 1} \overline{B}_g(p_i, r_i)\,.
            \]
            Then, there must exist a ball $\overline{B}_g(p_{i_0}, r_{i_0})$ satisfying
            \[
                \|V\|_g(\overline{B}_g(p_{i_0}, r_{i_0})) > \frac{\|V\|_g(M)}{4 \cdot 3^{2m}}\,.
            \]
            Since $r_{i_0} < 4 \eta_{\ref{lem:balls_not_cover_SV}} \leq s$, it follows that
            \[
                \|V\|_g(\overline{B}_g(p_{i_0}, R)) \geq (8 \cdot 3^{2m}) \|V\|_g(\overline{B}_g(p_{i_0}, r_{i_0}) ) > \|V\|_g(M)\,.
            \]
            This leads to a contradiction.
        \end{proof}
        
        \begin{lem}\label{lem:cycles_remove_balls}
            For any $m \in \mathbb N^+$ and $g \in \Gamma^\infty(M)$, there exists a positive constant $\eta_{\ref{lem:cycles_remove_balls}} = \eta_{\ref{lem:cycles_remove_balls}}(M, g, m)$ with the following property:
            
            For some $k \in \{0, 1, \cdots, n\}$, let $T, S \in \mathbf{I}_k(M; \Z_2)$ be two flat chains with $\partial T = \partial S = 0$, $\{p_i\}^t_{i = 1} \subset M$ with $t \leq 4\cdot 3^{2m}$ be a collection of points, and $\{r_i\}^t_{i = 1} \subset (0, 4\eta_{\ref{lem:cycles_remove_balls}})$. If
            $T = S$ on $M \setminus \bigcup^{t}_{i = 1} \overline{B}_g(p_i, r_i)$, then 
            \[
                T \in \Zc_k(M; \Z_2) \iff S \in \Zc_k(M; \Z_2)\,.
            \]
        \end{lem}
        \begin{rmk}
            In other words, if two flat chains with zero boundary coincide except on a suitable collection of balls, and one of them is a flat cycle, i.e., the boundary of a flat chain, then so is the other.
        \end{rmk}
        \begin{proof}
            Since $(M, g)$ is closed, we can find a positive constant $\eta_1 = \eta_1(M, g)$ such that for every $p \in M$ and  $s \in (0, 4\eta_1)$, $\overline{B}_g(p, s)$ is diffeomorphic to a ball.
    
            \begin{claim}\label{claim:cycles_remove_balls_1}
                If every $r_i < 4\eta_1$ and all the balls $\{\overline{B}_g(p_i, r_i)\}_i$ are disjoint from each other, then
                \[
                    T \in \Zc_n(M; \Z_2) \iff S \in \Zc_n(M; \Z_2)\,.
                \]
            \end{claim}
            \begin{proof}
                If $T \in \Zc_n(M; \Z_2)$, then there exists $R \in \mathbf{I}_{k+1}(M; \Z_2)$ such that $T = \partial R$.
    
                Since $T = S$ on $M \setminus \bigcup^{t}_{i = 1} \overline{B}_g(p_i, r_i)$, then
                \[
                    \spt(T - S) \subset \bigcup^{t}_{i = 1} \overline{B}_g(p_i, r_i)\,.
                \]
    
                In particular, since $\partial(T - S) = 0$ and the closed balls are disjoint from each other, for each $i$, $P_i = (T - S) \llcorner \overline{B}_g(p_i, r_i)$ also satisfies
                \[
                    \partial P_i = 0\,.
                \]
                As $\overline{B}_g(p_i, r_i)$ is contractible, we can find $R_i \in \mathbf{I}_k(M; \Z_2)$ with $\partial R_i = P_i$.
    
                Therefore, for $R' \coloneqq R + \sum_i R_i$, we have
                \[
                    \partial R' = \partial R + \sum_i \partial R_i = T + (T - S) = S\,,
                \]
                i.e., $S \in \Zc_k(M; \Z_2)$.
    
                By symmetry, if $S \in \Zc_k(M; \Z_2)$, then $T \in \Zc_k(M; \Z_2)$ as well. This completes the argument.
            \end{proof}
    
             In general, we have the following covering results.
    
            \begin{claim}\label{claim:cycles_remove_balls_2}
                There exists $\eta_2 = \eta_2(M, g, m) > 0$ such that if $r_i \in (0, 4 \eta_2)$, then any collection of balls $\{\overline{B}_g(p_i, r_i)\}^t_{i = 1}$ with $t \leq 4 \cdot 3^{2m}$, can be covered by $4 \cdot 3^{2m}$ many pairwise disjoint closed balls of radius at most $\eta_1$, $\{\overline{B}_g(p'_i, r'_i)\}^{t'}_{i = 1}$.
            \end{claim}
            \begin{proof}
                If false, we obtain a sequence of balls $\{\overline{B}_g(p^j_i, r^j_i)\}^{t_j}_{i = 1}$ such that
                \begin{enumerate}[(i)]
                    \item $t_j \leq 4\cdot 3^{2m}$;
                    \item $\lim_j \sup_i r^j_i = 0$;
                    \item $\{\overline{B}_g(p^j_i, r^j_i)\}^{t_j}_{i = 1}$ cannot be covered by $4 \cdot 3^{2m}$ disjoint closed balls of radius at most $\eta^1$.
                \end{enumerate}
                Up to a subsequence, $\{\overline{B}_g(p^j_i, r^j_i)\}^{t_j}_{i = 1}$ converges to a set of points $\{p'_i\}^{t'}_{i = 1}$ in the Hausdorff sense and $t' \leq 4\cdot 3^{2m}$. Obviously, we can choose $4 \cdot 3^{2m}$ pairwise disjoint closed balls of radius at most $\eta_1$  whose interiors cover $\{p'_i\}^{t'}_{i = 1}$. Hence, these balls also cover $\{\overline{B}_g(p^j_i, r^j_i)\}^{t_j}_{i = 1}$ for sufficiently large $j$. This contradicts our assumption (iii).
            \end{proof}
    
            Now, let $\eta_{\ref{lem:cycles_remove_balls}} = \min(\eta_1, \eta_2)$ defined above, and then $\{\overline{B}_g(p_i, r_i)\}^t_{i = 1}$ can be covered by pairwise disjoint balls of radii no greater than $\eta_1$,
            \[
                \{\overline{B}_g(p'_i, r'_i)\}^{t'}_{i = 1}\,.
            \]
            Since $T = S$ on $M \setminus \bigcup^{t'}_{i = 1} \overline{B}_g(p'_i, r'_i)$, it follows from Claim \ref{claim:cycles_remove_balls_1} that 
            \[
                T \in \Zc_k(M; \Z_2) \iff S \in \Zc_k(M; \Z_2)\,.
            \]
        \end{proof}

        \begin{defn}
            In $(M, g)$, an \emph{embedded minimal cycle} is a stationary integral varifold supported on a smooth, closed, embedded minimal hypersurface. We denote by $\mathcal{W}^L_g \subset \Vc_n(M)$ the set of all the embedded minimal cycles with total measure $L$.
        \end{defn}

        Next, we prove a varifold version of Lemma~\ref{lem:cycles_remove_balls}.
    
        \begin{lem}\label{lem:cycle_of_small_rep}
            Let $m \in \mathbb N^+$, $L \in \R^+$ and $g \in \Gamma^\infty(M)$ be a $C^\infty$ Riemannian metrics on $M$. Let $\cG \subset \mathcal{W}^L_g$ be a compact subset such that for every $W \in \cG$, there exists $T \in \Zc_n(M; \Z_2)$ such that $W = |T|$. Let $r \coloneqq \min(\eta_{\ref{lem:balls_not_cover_SV}}(M, g, m), \eta_{\ref{lem:cycles_remove_balls}}(M, g, m))$ from Lemmas \ref{lem:balls_not_cover_SV} and \ref{lem:cycles_remove_balls}. Then there exists $\eta_{\ref{lem:cycle_of_small_rep}} = \eta_{\ref{lem:cycle_of_small_rep}}(M, g, m, \cG) > 0$ with the following property.
    
            If $V_1, V_2 \in \mathcal{V}_n(M)$ satisfy
            \begin{itemize}
                \item $V_1 \in \mathbf{B}^{\Fb_g}_{\eta_{\ref{lem:cycle_of_small_rep}}}(\cG)$;
                \item $V_2 \in \mathcal{W}^L_g$;
                \item $V_2 = V_1$ on $M \setminus (\overline{B}_{g}(p_1, 2r) \cup \cdots \cup \overline{B}_{g}(p_{t}, 2r))$ for some collection\\ $\{p_1, \cdots, p_{t}\} \subset M, t \leq 3^{2m}$.
            \end{itemize}
            Then there exists $T_2 \in \Zc_n(M; \Z_2)$ such that $V_2 = |T_2|$.
        \end{lem}
        \begin{rmk}
            As in Lemma \ref{lem:cycles_remove_balls}, the key point here is that $V_2$ is the varifold associated to a flat chain that is a flat cycle, which, more precisely, is the reduced boundary of a Caccioppoli set.
        \end{rmk}
        \begin{proof}
            The lemma follows from the following claim.
        
            \begin{claim}\label{claim:cycle_of_small_rep_0}
                For every $W \in \cG$, there exists $\eta_1 = \eta_1(M, g, r, W) > 0$ such that for any $V \in \mathcal{W}^L_g$, any $s \in [2r, 5r/2]$, and any collection $\{p_1, \cdots, p_{t}\} \subset M, t \leq 3^{2m}$, if
                \begin{equation}\label{eqn:cycle_of_small_rep_1}
                    \Fb_{g, M \setminus (\overline{B}_{g}(p_1, s) \cup \cdots \cup \overline{B}_{g}(p_{t}, s))}(W, V) < \eta_1\,,
                \end{equation}
                then there exists $T' \in \cZ_n(M; \Z_2)$ such that $V = |T'|$. In particular, since $r$ depends on $M$, $g$, and $m$, so $\eta_1 = \eta_1(M, g, m, W)$.
            \end{claim}
            \begin{proof}
                Let $T \in \Zc_n(M; \Z_2)$ such that $W = |T|$. By Allard's $\varepsilon$-regularity theorem \cite{All72}, there exists an $\eta_1(M, g, r, W) > 0$ independent of the choice of balls, such that for any $V$ satisfying \eqref{eqn:cycle_of_small_rep_1}, we have 
                \[
                    \spt(V) \cap M \setminus (\overline{B}_{g}(p_1, 3r) \cup \cdots \cup \overline{B}_{g}(p_{t}, 3r))
                \] is a subset of a (multiplicity-one) minimal graph over $W$. 
                
                Hence, as $r \leq \eta_{\ref{lem:balls_not_cover_SV}}(M, g, m)$, by Lemma \ref{lem:balls_not_cover_SV}, $V$ has multiplicity one and thus,
                \[
                    \mathcal{W}^L_g \ni V = |T'|\,,
                \]
                for some current $T' \in \mathbf{I}_n(M; \Z_2)$ with $\partial T' = 0$.
    
                Furthermore, the graphical property implies the existence of $R \in \mathbf{I}_{n+1}(M; \Z_2)$ such that
                \[
                    \spt (\partial R - T' - T) \subset (\overline{B}_{g}(p_1, 7r/2) \cup \cdots \cup \overline{B}_{g}(p_{t}, 7r/2))\,.
                \]
                i.e., $\partial R + T = T'$ in $M \setminus (\overline{B}_{g}(p_1, 7r/2) \cup \cdots \cup \overline{B}_{g}(p_{t}, 7r/2))$. It follows from Lemma \ref{lem:cycles_remove_balls} and $r \leq \eta_{\ref{lem:cycles_remove_balls}}(M, g, m)$ that 
                \[
                    T' \in \Zc_n(M; \Z_2)\,.
                \]
            \end{proof}
            
            Indeed, there exists $s \in [2r, 5r/2]$ such that $\spt(W)$ intersects every $\partial B_{g}(p_i, s)$ transversally. Hence, there exist $M_s = M \setminus (\overline{B}_{g}(p_1, s) \cup \cdots \cup \overline{B}_{g}(p_{t}, s))$ and  $\eta'_W > 0$ such that for every $V' \in \Vc_n(M)$,
            \begin{equation}\label{eqn:lem:cycle_of_small_rep_1}
                \Fb_{g}(W, V') < \eta'_W \implies \Fb_{g, M_s}(W, V') < \eta_1(M, g, r, W)\,,
            \end{equation}
            Since $\cG$ is compact, there exists a finite subset $\{W_i\}^N_{i = 1}$ such that
            \[
                \cG \subset \bigcup^N_{i = 1} \bB^{\Fb_g}_{\eta'_{W_i}}(W_i)\,,
            \]
            and furthermore, we can choose $\eta_{\ref{lem:cycle_of_small_rep}}$ such that
            \begin{equation}\label{eqn:lem:cycle_of_small_rep_2}
                \bB^{\Fb_g}_{\eta_{\ref{lem:cycle_of_small_rep}}}(\cG) \subset \bigcup^N_{i = 1} \bB^{\Fb_g}_{\eta'_{W_i}}(W_i)\,.
            \end{equation}
    
            Now if $V_1$ and $V_2$ satisfies all the conditions in the lemma, then
            \begin{align*}
                V_1 \in \bB^{\Fb_g}_{\eta_{\ref{lem:cycle_of_small_rep}}}(\cG) &\implies \exists\ W_i \in \cG,\ V_1 \in \bB^{\Fb_g}_{\eta'_{W_i}}(W_i)\quad \text{(by \eqref{eqn:lem:cycle_of_small_rep_2})}\\
                    &\implies \exists\ s \in [2r, 5r/2],\ \Fb_{g, M_s}(W_i, V_1) < \eta_{W_i}\quad \text{(by \eqref{eqn:lem:cycle_of_small_rep_1})}\\
                    &\implies \Fb_{g, M_s}(W_i, V_2) < \eta_{W_i} \quad \text{(by the third bullet point)}\\
                    &\implies \exists\ T_2 \in \Zc_n(M; \Z_2),\ V_2 = |T_2| \quad \text{(by Claim \ref{claim:cycle_of_small_rep_0})}\,.
            \end{align*}
        \end{proof}
        
        \begin{lem}\label{lem:far_of_diff_geom}
            Let $m \in \mathbb N^+$, $L \in \R^+$, $g \in \Gamma^\infty(M)$ be a $C^\infty$ Riemannian metrics on $M$, $r \coloneqq \min(\eta_{\ref{lem:balls_not_cover_SV}}(M, g, m), \eta_{\ref{lem:cycles_remove_balls}}(M, g, m))$ and $\cG, \cB \subset \mathcal{W}^L_g$ be compact subsets such that
            \begin{itemize}
                \item For every $V \in \cG$, there exists $T \in \Zc_n(M; \Z_2)$ such that $V = |T|$;
                \item For every $V \in \cB$, no $T \in \Zc_n(M; \Z_2)$ such that $V = |T|$.
            \end{itemize}
            Then there exists $\eta_{\ref{lem:far_of_diff_geom}} = \eta_{\ref{lem:far_of_diff_geom}}(M, g, m, L, \cG, \cB) > 0$ with the following property:
    
            If $g', g'' \in \Gamma^\infty(M)$, and $\{V_i\}^6_{i = 1} \cup \{W_i\}^6_{i = 1} \subset \Vc_n(M)$ satisfy
            \begin{itemize}
                \item $\|g' - g\|_{C^\infty, g} < \eta_{\ref{lem:far_of_diff_geom}}$, $\|g'' - g\|_{C^\infty, g} < \eta_{\ref{lem:far_of_diff_geom}}$;
                \item $\|V_i\|_{g}(M), \|W_i\|_{g}(M)< 2L$ for any $i$;
                \item $V_1 \in \bB^{\Fb_g}_{\eta_{\ref{lem:far_of_diff_geom}}}(\cG), W_1 \in \bB^{\Fb_g}_{\eta_{\ref{lem:far_of_diff_geom}}}(\cB)$;
                \item $V_2 = V_1$ on $M \setminus (\overline{B}_{g'}(p_1, 2r) \cup \cdots \cup \overline{B}_{g'}(p_{t}, 2r))$ for some collection\\ $\{p_1, \cdots, p_{t}\} \subset M, t \leq 3^{2m}$;
                \item $W_2 = W_1$ on $M \setminus (\overline{B}_{g'}(\bar{p}_1, 2r) \cup \cdots \cup \overline{B}_{g'}(\bar{p}_{\bar t}, 2r))$ for some collection\\ $\{\bar p_1, \cdots, \bar p_{\bar t}\} \subset M, \bar t \leq 3^{2m}$;
                \item $\Fb_{g'}(V_3, V_2) < \eta_{\ref{lem:far_of_diff_geom}}$;
                \item $\Fb_{g'}(W_3, W_2) < \eta_{\ref{lem:far_of_diff_geom}}$;
                \item $\Fb_{g'}(V_4, V_3) < \eta_{\ref{lem:far_of_diff_geom}}$;
                \item $\Fb_{g'}(W_4, W_3) < \eta_{\ref{lem:far_of_diff_geom}}$;
                \item $V_5 = V_4$ on $M \setminus (\overline{B}_{g''}(p'_1, 2r) \cup \cdots \cup \overline{B}_{g''}(p'_{t'}, 2r))$ for some collection\\ $\{p'_1, \cdots, p'_{t'}\} \subset M, t' \leq 3^{2m}$;
                \item $W_5 = W_4$ on $M \setminus (\overline{B}_{g''}(\bar{p}'_1, 2r) \cup \cdots \cup \overline{B}_{g''}(\bar{p}'_{\bar t'}, 2r))$ for some collection\\ $\{\bar p'_1, \cdots, \bar p'_{\bar t'}\} \subset M, \bar t' \leq 3^{2m}$;
                \item $\Fb_{g''}(V_6, V_5) < \eta_{\ref{lem:far_of_diff_geom}}$;
                \item $\Fb_{g''}(W_6, W_5) < \eta_{\ref{lem:far_of_diff_geom}}$;
            \end{itemize}
            Then $\bF_{g}(V_6, W_6) \geq \eta_{\ref{lem:far_of_diff_geom}}$.
    
            In particular, we have: Let $\cG'_{\ref{lem:far_of_diff_geom}} = \cG'_{\ref{lem:far_of_diff_geom}}(M, g, m, L, \cG, \cB)$ be the set of all $V_3$, $\cG''_{\ref{lem:far_of_diff_geom}} = \cG''_{\ref{lem:far_of_diff_geom}}(M, g, m, L, \cG, \cB)$ be the set of all $V_6$, $\cB'_{\ref{lem:far_of_diff_geom}} = \cB'_{\ref{lem:far_of_diff_geom}}(M, g, m, L, \cG, \cB)$ be the set of all $W_3$, and $\cB''_{\ref{lem:far_of_diff_geom}} = \cB''_{\ref{lem:far_of_diff_geom}}(M, g, m, L, \cG, \cB)$ be the set of all $W_6$, which satisfy the conditions above. Then,
            \begin{enumerate}
                \item $\cG'_{\ref{lem:far_of_diff_geom}}, \cG''_{\ref{lem:far_of_diff_geom}}, \cB'_{\ref{lem:far_of_diff_geom}}$ and $\cB''_{\ref{lem:far_of_diff_geom}}$ are open;
                \item $\cG \subset \cG'_{\ref{lem:far_of_diff_geom}} \subset \cG''_{\ref{lem:far_of_diff_geom}}$ and $\cB \subset \cB'_{\ref{lem:far_of_diff_geom}} \subset \cB''_{\ref{lem:far_of_diff_geom}}$;
                \item ${\bF_g}(\cG''_{\ref{lem:far_of_diff_geom}}, \cB''_{\ref{lem:far_of_diff_geom}}) \geq {\bF_g}(\cG'_{\ref{lem:far_of_diff_geom}}, \cB'_{\ref{lem:far_of_diff_geom}}) \geq \eta_{\ref{lem:far_of_diff_geom}}$.
            \end{enumerate}
        \end{lem}
        \begin{rmk}\label{rmk:desc_cB}
            \begin{enumerate}
                \item This lemma essentially shows that through a finite number of annular replacements and approximations, it is not possible for two elements from sets $\cG$ and $\cB$, respectively, to approach each other closely.
                \item For every $V \in \cB$, either one connected component of $\spt(V)$ has multiplicity, or there exists $T' \in \mathbf{I}_n(M; \Z_2) \setminus \Zc_n(M; \Z_2)$ such that $V = |T'|$ and $\partial T' = 0$.
            \end{enumerate}
        \end{rmk}
        \begin{proof}
            If this is false, by compactness, there exists $V_1 \in \cG$ and $W_1 \in \cB$ such that
            \[
                V_1 = W_1
            \]
            on $M \setminus (\overline{B}_{g}(p_1, 3r) \cup \cdots \cup \overline{B}_{g}(p_{t}, 3r))$ for some collection $\{p_1, \cdots, p_{t}\} \subset M, t \leq 4\cdot 3^{2m}$;
    
            Suppose for the sake of contradiction that one connected component $\tilde \Sigma$ of $\spt(W_1)$ has multiplicity at least two, then by Lemma \ref{lem:balls_not_cover_SV} and $r \leq \eta_{\ref{lem:balls_not_cover_SV}}(M, g, m)$, 
            \[
                \tilde \Sigma' \coloneqq \tilde \Sigma \setminus \bigcup^{t}_{i = 1} \bar{B}_g(p_i, 3r) \neq \emptyset\,.
            \]
            Therefore, $V_1\llcorner \tilde \Sigma' = W_1\llcorner \tilde \Sigma'$, which contradicts the multiplicity one property of $V_1$.
    
            Hence, both $V_1$ and $W_1$ have multiplicity one. In particular, there exists $T \in \mathcal Z_n(M; \Z_2)$  and $S \in \bI_n(M;\Z_2)$ such that
            \[
                V_1 = |T|, \quad W_1 = |S|\,.
            \]
            Since $T = S$ on $M \setminus (\overline{B}_{g}(p_1, 3r) \cup \cdots \cup \overline{B}_{g}(p_{t}, 3r))$, and $r \leq \eta_{\ref{lem:cycles_remove_balls}}(M, g, m)$, $S \in \cZ_n(M; \Z_2)$, which contradicts the definition of $\cB$.
        \end{proof}
    
    \subsection{Almost minimizing property}

        First, we show that admitting an $(\varepsilon, \delta)$-deformation is an open condition.
        
        \begin{lem}\label{lem:(e, d)-def_varying_metric}
            Let $U \subset (M, g)$ be an open subset and $L, \varepsilon > 0$ be positive numbers. There exists $\eta_{\ref{lem:(e, d)-def_varying_metric}} = \eta_{\ref{lem:(e, d)-def_varying_metric}}(M, g, L, U, \varepsilon) > 0$ for which the following hold.
            
            If $g' \in \Gamma^\infty(M)$ and $V, V' \in \mathcal{V}_n(M)$ satisfy
            \begin{itemize}
                \item $\|g' - g\|_{C^\infty, g} \leq \eta_{\ref{lem:(e, d)-def_varying_metric}}$,
                \item $\|V\|_g(M) \leq 2L$,
                \item $T \notin \mathfrak{a}_{g}(U; 2\varepsilon, \delta)$ for any $\delta > 0$ and $T \in \Zc_n(M; \Z_2)$ with $\Fb_{g}(V, |T|) < 2\varepsilon$,
                \item $\Fb_{g}(V, V') \leq \eta_{\ref{lem:(e, d)-def_varying_metric}}$,
            \end{itemize}
            then for any $\delta > 0$ and any $T' \in \Zc_n(M; \Z_2)$ with $\Fb_{g'}(V', |T'|) < \varepsilon$, 
            \[
                T' \notin \mathfrak{a}_{g'}(U; \varepsilon, \delta).
            \]
        \end{lem}
        \begin{proof}
            If this is false for any $\eta_i = \frac{1}{i} (i = 1, 2, 3, \cdots)$, we obtain the existence of $g'_i, V_i, V'_i$ satisfying all the conditions in the lemma with $(\eta_i, g'_i, V_i, V'_i)$ in place of $(\eta, g', V, V')$ but
            \[
                T'_i \in \mathfrak{a}_{g'_i}(U; \varepsilon, \delta_i)
            \]
            for some $\delta_i > 0$ and $T'_i \in \Zc_n(M; \Z_2)$ with $\Fb_{g'_i}(V'_i, |T'_i|) < \varepsilon$.
    
            For sufficiently large $i$, $\Fb_{g}(V_i, |T'_i|) < 2\varepsilon$, so $T'_i \notin \mathfrak{a}_{g}(U; 2\varepsilon, \delta_i / 2)$, i.e., there exists a finite sequence $(T_j)^q_{j = 0}$ in $\Zc_n(M; \Z_2)$ with
            \begin{enumerate}[(i)]
                \item $T_0 = T'_i$ and $\spt(T - T_j) \subset U$ for all $j = 1, \cdots, q$;
                \item $\Mb_g(T_j - T_{j - 1}) \leq \delta_i/2$ for all $j = 1, \cdots, q$;
                \item $\Mb_g(T_j) \leq \Mb_g(T) + \delta_i/2$ for all $j = 1, \cdots, q$;
                \item $\Mb_g(T_q) < \Mb_g(T) - 2\varepsilon$.
            \end{enumerate}
            Hence, if $i$ is sufficiently large, the above conditions induce the similar ones with respect to $g'_i$,
            \begin{enumerate}[(i)]
                \item $T_0 = T'_i$ and $\spt(T - T_j) \subset U$ for all $j = 1, \cdots, q$;
                \item $\Mb_{g'_i}(T_j - T_{j - 1}) \leq \delta_i$ for all $j = 1, \cdots, q$;
                \item $\Mb_{g'_i}(T_j) \leq \Mb_{g'_i}(T) + \delta_i$ for all $j = 1, \cdots, q$;
                \item $\Mb_{g'_i}(T_q) < \Mb_{g'_i}(T) - \varepsilon$;
            \end{enumerate}
            which contradicts $T'_i \in \mathfrak{a}_{g'_i}(U; \varepsilon, \delta_i)$.
        \end{proof}
        
        \begin{lem}\label{lem:not_am_varying_metric}
            Let $m \in \N^+$, $r, d, L \in \R^+$, $\mathcal{SV}^L_g$ be the space of all stationary varifolds in $(M^{n+1}, g)\ (3\leq n+1 \leq 7)$ with total measure $L$, and 
            \[
                \mathcal{W}^L_{g, m, r} \subset \mathcal{SV}^L_g
            \]
            be the subset of all the $(m, r)_g$-almost minimizing, stationary integral varifolds whose support is a smooth, closed minimal hypersurface with total measure $L$. Clearly, $\mathcal{W}^L_{g, m, r} \subset \mathcal{W}^L_g$. Then there exist positive constants $\bar \varepsilon_{\ref{lem:not_am_varying_metric}} = \bar \varepsilon_{\ref{lem:not_am_varying_metric}}(M, g, m, r, d, L)$, $\bar s_{\ref{lem:not_am_varying_metric}} = \bar s_{\ref{lem:not_am_varying_metric}}(M, g, m, r, d, L)$ and $\eta_{\ref{lem:not_am_varying_metric}} = \eta_{\ref{lem:not_am_varying_metric}}(M, g, m, r, d, L)$ for which the following hold.
            
            If $g' \in \Gamma^\infty(M)$ and $V \in \mathcal{V}_n(M)$ satisfy
            \begin{itemize}
                \item $\|g - g'\|_{C^\infty, g} < \eta_{\ref{lem:not_am_varying_metric}}$,
                \item $V \in \bB^{\Fb_g}_{\eta_{\ref{lem:not_am_varying_metric}}}(\mathcal{SV}^L_g)$, and
                \item $V \notin \bB^{\Fb_g}_d(\mathcal{W}^L_{g, m, r})$,
            \end{itemize}
            then there exists $p\in M$ and $I_m$ concentric annuli
            \[
                \{\mathrm{An}_{g',i}(V)\}^{I_m}_{i = 1} \coloneqq\{\mathrm{An}_{g'}(p,r_i-s_i,r_i+s_i)\}^{I_m}_{i = 1}
            \] 
            such that
            \begin{enumerate}
                \item $\{r_i\}$ and $\{s_i\}$ satisfy
                    \begin{align*}
                        r_i-2s_i&> 2(r_{i+1}+2s_{i+1}),\;i=1,\dots,I_m-1,\\
                        r_{I_m}-2s_{I_m}&>0,\\
                        r_1 + s_1 &< r,\\
                        \min_i\{s_i\} &> \bar s_{\ref{lem:not_am_varying_metric}};
                    \end{align*}
                \item For any  $i \in \{1, \cdots, I_m\}$, $\delta > 0$ and $T \in \Zc_n(M; \Z_2)$, if $\Fb_{g'}(V, |T|) < \bar \varepsilon_{\ref{lem:not_am_varying_metric}}$, then $T \notin \mathfrak{a}_{g'}(\mathrm{An}_{g',i}(V), \delta, \bar\varepsilon)$.
            \end{enumerate}
        \end{lem}
        \begin{proof}
            Let $\mathcal{K} \coloneqq \mathcal{SV}^L_g \setminus \bB^{\Fb_g}_d(\mathcal{W}^L_g)$.
        
            For every $V \in \mathcal{K}$, there exist $\varepsilon_V > 0$, $p_V \in M$, and $I_m$ concentric annuli $\{\tilde{\mathrm{An}}_{g, i}(V)\}^{I_m}_{i = 1} \equiv \{\mathrm{An}_g(p_V,r_{V, i}-s_{V, i},r_{V, i}+s_{V, i})\}^{I_m}_{i = 1}$ such that 
            \begin{enumerate}[(i)]
                \item $\{r_{V, i}\}$ and $\{s_{V, i}\}$ satisfy
                    \begin{align*}
                        r_{V,i}-2s_{V,i}&> 2(r_{V, i+1}+2s_{V, i+1}),\;i=1,\dots,I_m-1\,,\\
                        r_{V, I_m}-2s_{V, I_m}&>0\,,\\
                        r_{V, 1} + s_{V, 1} &< r\,;
                    \end{align*}
                \item For any $i \in \{1, \cdots, I_m\}$, $\delta > 0$ and $T \in \Zc_n(M; \Z_2)$, if $\Fb_{g}(V, |T|) <  \varepsilon_V$, then 
                \[
                    T \notin \mathfrak{a}_{g}(\tilde{\mathrm{An}}_{g, i}(V), \delta, \varepsilon_V)\,.
                \]
            \end{enumerate}
            Otherwise, $V$ is both $(m, r)_g$-almost minimizing and stationary, so by Theorem \ref{thm:Schoen_Simon_reg}, $V \in \mathcal{W}^L_g$, which contradicts $V \not \in \bB^{\Fb_g}_d(\mathcal{W}^L_g)$.
    
            Since $\mathcal{K}$ is compact, there exists a finite subset $\{V_j\}^N_{j = 1} \subset \mathcal{K}$ such that
            \[
                \mathcal{K} \subset \bigcup^N_{j = 1} \bB^{\Fb_g}_{\varepsilon_{V_j}/2}(V_j)\,.
            \]
            In other words, for every $V \in \mathcal{K}$, there exists some $j$ such that $V \in \bB^{\Fb_g}_{\varepsilon_{V_j}/2}(V_j)$, so for any $T \in \Zc_n(M;\Z_2)$ with $\Fb_{g}(V, |T|) <  \varepsilon_{V_j}/2$, any $i \in \{1, \cdots, I_m\}$ and any $\delta > 0$,
            \[
                T \notin \mathfrak{a}_{g}(\tilde{\mathrm{An}}_{g, i}(V_j), \delta, \varepsilon_{V_j}/2)\,.
            \]
    
            Moreover, there exists a positive constant $\eta'$ such that for any $g' \in \Gamma^\infty(M)$ with $\|g' - g\|_{C^\infty, g} < \eta'$, and any $j \in \{1, \cdots, N\}$, there exist concentric annuli $\{\tilde{\mathrm{An}}_{g', i}(V_j)\}^{I_m}_{i = 1} \equiv \{\mathrm{An}_{g'}(p_{V_j},r_{g', V_j, i}-s_{g', V_j, i},r_{g', V_j, i}+s_{g', V_j, i})\}^{I_m}_{i = 1}$ in $(M, g')$ such that
            \begin{enumerate}[(i)]
                \item The radii  $\{r_{g', V_j, i}\}^{I_m}_{i = 1}$ and $\{s_{g', V_j, i}\}^{I_m}_{i = 1}$ satisfy
                    \begin{align*}
                        r_{g', V_j, i}-2s_{g', V_j, i}&> 2(r_{g', V_j, i+1}+2s_{g', V_j, i+1}),\;i=1,\dots,I_m-1\,,\\
                        r_{g', V_j, I_m}-2s_{g', V_j, I_m}&>0\,,\\
                        r_{g', V_j, 1} + s_{g', V_j, 1} &< r\,;
                    \end{align*}
                \item $\min_i (s_{g', V_j, i}) > \min_j \min_i \frac{s_{V_j, i}}{2}$;
                \item $\tilde{\mathrm{An}}_{g', i}(V_j) \supset \tilde{\mathrm{An}}_{g, i}(V_j)$;
            \end{enumerate}
    
            Now, we set 
            \begin{align*}
                \bar\varepsilon_{\ref{lem:not_am_varying_metric}} &\coloneqq \min_j \frac{\varepsilon_{V_j}}{10}\,,\\
                \bar s_{\ref{lem:not_am_varying_metric}} &\coloneqq \min_j \min_i \frac{s_{V_j, i}}{2}\,,\\
                \eta_{\ref{lem:not_am_varying_metric}} &\coloneqq \min\{d, \min_j \min_i \eta_{\ref{lem:(e, d)-def_varying_metric}}(M, g, L, \tilde{\mathrm{An}}_{g, i}(V_j), \bar\varepsilon_{\ref{lem:not_am_varying_metric}})\}\,.
            \end{align*}
            
            To see that these constants fulfill our requirements, let $g' \in \Gamma^\infty(M)$ and $W \in \Vc_n(M)$ with $\|g - g'\|_{C^\infty, g} < \eta$, $W \in \bB^{\Fb_g}_\eta(\mathcal{SV}^L_g)$, and $W \notin \bB^{\Fb_g}_d(\mathcal{W}^L_g)$. Therefore, there exists $V \in \mathcal{K}$ and $V_j$ such that 
            \begin{align*}
                \Fb_g(W, V) &< \eta_{\ref{lem:not_am_varying_metric}}\,,\\
                V &\in \bB^{\Fb_g}_{\varepsilon_{V_j}/2}(V_j)\,.
            \end{align*}
            Since $V$ satisfies that for any $T \in \Zc_n(M;\Z_2)$ with $\Fb_{g}(V, |T|) < 2\bar\varepsilon_{\ref{lem:not_am_varying_metric}}$, any $i \in \{1, \cdots, I_m\}$ and any $\delta > 0$,
            \[
                T \notin \mathfrak{a}_{g}(\tilde{\mathrm{An}}_{g, i}(V_j), \delta, 2\bar\varepsilon_{\ref{lem:not_am_varying_metric}})\,,
            \]
            it follows from Lemma \ref{lem:(e, d)-def_varying_metric} that for any $T' \in \Zc_n(M; \Z_2)$ with $\Fb_{g'}(W, |T'|) < \bar\varepsilon_{\ref{lem:not_am_varying_metric}}$, any $i \in \{1, \cdots, I_m\}$ and any $\delta > 0$ that
            \[
                T' \notin \mathfrak{a}_{g'}(\tilde{\mathrm{An}}_{g, i}(V_j), \delta, \bar\varepsilon_{\ref{lem:not_am_varying_metric}})\,,
            \]
            and thus,
            \[
                T' \notin \mathfrak{a}_{g'}(\tilde{\mathrm{An}}_{g', i}(V_j), \delta, \bar\varepsilon_{\ref{lem:not_am_varying_metric}})\,,
            \]
            as $\tilde{\mathrm{An}}_{g', i}(V_j) \supset \tilde{\mathrm{An}}_{g, i}(V_j)$. Letting
            \[
                \mathrm{An}_{g', i}(W) \coloneqq \tilde{\mathrm{An}}_{g', i}(V_j)\,,
            \]
            for each $i$, the conclusions (1) and (2) follow immediately.
        \end{proof}
        
        \begin{lem}\label{lem:am_cpt_varying_metric}
            Let $m \in \N^+$, $r \in \R^+$, $(M^{n+1}, g) (3\leq n+1 \leq 7)$ be a closed Riemannian manifold, and $(g_i)^\infty_{i = 1}$ be a sequence of metrics with $g_i\to g$ in $C^\infty$. For each $i \in N^+$, let $\Sigma_i$ be an $(m, r)_{g_i}$-almost minimizing minimal hypersurface in $(M,g_i)$. Then $\Sigma_i$ subsequentially converges graphically in $C^\infty$, away from at most finitely many points, to some $(m, r)_{g}$-almost minimizing minimal hypersurface $\Sigma$ in $(M,g)$.
    
            Furthermore, suppose that $g$ is a metric with positive Ricci curvature or a bumpy metric. If $(\Sigma_i)^\infty_{i = 1}$ is multiplicity-one, then $\Sigma$ is also multiplicity-one; If $(\Sigma_i)^\infty_{i = 1}$ is two-sided, then $\Sigma$ is also two-sided. 
        \end{lem}
        \begin{proof}
            By Allard's compactness theorem \cite{All72}, $\Sigma_i$ subsequentially converges, in the varifold sense, to a stationary integral varifold $V$ in $(M, g)$. Without loss of generality, by relabelling, we may assume that $\Sigma_i$ converges to $V$ in the varifold sense.
        
            To see that $V$ is $(m, r)_g$ almost-minimizing, we take any $p \in M$ and  $I_m$ concentric annuli $\{\mathrm{An}_g(p,r_j-s_j,r_j+s_j)\}$ where $\{r_j\}$ and $\{s_j\}$ satisfy
            \begin{align*}
                r_j-2s_j&>2(r_{j+1}+2s_{j+1}),\;j=1,\dots,I_m-1,\\
                r_{I_m}-2s_{I_m}&>0,\\
                r_1 + s_ 1 &< r\,.
            \end{align*}
            Since each $\Sigma_i$ is almost minimizing in at least one of these finitely many annuli, there exists a $j_0 \in \{1, \cdots, I_m\}$ and a subsequence $(\Sigma_{i_k})^\infty_{k = 1}$ such that every $\Sigma_{i_k}$ is almost minimizing in $\mathrm{An}_g(p,r_{j_0}-s_{j_0},r_{j_0}+s_{j_0})$. Consequentially, their limit $V$ is also almost minimizing in $\mathrm{An}_g(p,r_{j_0}-s_{j_0},r_{j_0}+s_{j_0})$ by Lemma \ref{lem:(e, d)-def_varying_metric}, so $V$ is $(m, r)_g$-almost minimizing and its support is a smooth embedded minimal hypersurface, denoted by $\Sigma$, i.e., $V = m|\Sigma|$ for some $m \in \mathbb{N}^+$.
        
            If $m = 1$, then by Allard's regularity theorem, $\Sigma_i$ subsequentially converges in $C^\infty$. Therefore, it suffices to show that $m \geq 2$ is impossible provided that $g$ is a metric with positive Ricci curvature or a bumpy metric.
        
            Note that the almost minimizing property implies stability. By Schoen-Simon's regularity theory, the convergence above in $\mathrm{An}_g(p,r_{j_0}-s_{j_0},r_{j_0}+s_{j_0})$ is locally smooth and graphical.
        
            Let $\mathcal{S}$ be the set of open subsets of $M$ such that for each $A \in \mathcal{S}$, there exists a subsequence $(\Sigma_{i_k})^\infty_{k = 1}$ converges locally smoothly and graphically in $A$. Clearly, every non-empty totally ordered subset $\mathcal{T}$ of $\mathcal{S}$ has an upper bound (simply by taking the union of all the sets in $\mathcal{T}$), so by Zorn's lemma, $\mathcal{S}$ has at least one maximal element. Let us denote one of these maximal elements by $R$.
        
            Now, we shall show that $M \setminus R$ has at most finitely many points. Suppose not and there exists a sequence of distinct points $(p_l)^\infty_{l = 1}$ in $M \setminus R$ which converges to some $p \in M$. Let $(\Sigma_{i_k})^\infty_{k = 1}$ be the subsequence that converges locally smoothly and graphically in $R$. We can choose $I_m$ concentric annuli $\{\mathrm{An}_g(p, r_j -s _j, r_j + s_j)\}$ satisfying the relations as above and further such that each annuli contains at least one $p_l$. By the $(m, r)_g$-almost minimizing property at $p$, $(\Sigma_{i_k})^\infty_{k = 1}$ has a subsequence $(\Sigma_{i'_k})^\infty_{k = 1}$, each of which is also almost minimizing in a concentric annulus, say, $\mathrm{An}_g(p, r_{j_0} - s_{j_0}, r_{j_0} + s_{j_0})$ containing $p_{l_0}$. Therefore, $(\Sigma_{i'_k})^\infty_{k = 1}$ converges locally smoothly and graphically in $R \cup \mathrm{An}_g(p, r_{j_0} - s_{j_0}, r_{j_0} + s_{j_0}) \supsetneq R$, contradicting the maximality of $R$.
        
            It follows from \cite[Theorem 5]{AmbrozioCarlottoSharp2018comparing} that when $m \geq 2$, $\Sigma$ (or its double cover $\tilde \Sigma$ if one-sided) is stable and has nullity $1$, which is impossible since $g$ is either bumpy or of positive Ricci curvature.
        \end{proof}
    
        In \cite{Yangyang_Li19}, it was shown that every $p$-width $\omega(M, g)$ can be realized by a varifold with Property $(2p+1)_g$. An important ingredient of its proof is the following lemma, which was established in Proposition 3.2 therein. For the sake of completeness, we present the proof here.
    
        \begin{lem}\label{lem:simp_to_cub}
            If $X$ is a $k$-dimensional finite simplicial complex, there exists a cubical subcomplex $Y$ of $I(2k+1, l)$ for some $l \in \mathbb{N}^+$ for which the following hold.
            \begin{enumerate}
                \item If we regard $Y$ as a simplicial complex, $X$ can be viewed as a subcomplex of some refinement of $Y$. In particular, there exists an embedding
                \[
                    \iota: X \to Y;
                \]
                \item There exists a retraction map
                \[
                    r: Y \to X\,.
                \]
            \end{enumerate}
        \end{lem}
        \begin{proof}
            Note that the underlying set of a finite simplicial complex is also a compact polyhedron (See \cite[1.8]{RS82}). Applying the general position theorem for maps \cite[Theorem~5.4]{RS82} with $M = I^{2k+1}$ (endowed with the Euclidean metric), $P = X$, $P_0 = \emptyset$, $\varepsilon = 1/10$ and the closed map $f: X \to I^{2k+1}$ defined by 
            \[
                f(p) \equiv c(I^{2k+1})
            \]
            where $c(I^{2k+1})$ is the center point of $I^{2k+1}$, we obtain a piecewise-linear embedding
            \[
                f' : X \to I^{2k+1}\,.
            \]
            $f'$ is an embedding since it is nondegenrate and $\dim(S(f')) \leq 2k - (2k + 1) < 0$.
    
            By \cite[p.33]{RS82}, there exists a regular neighborhood $U$, and by \cite[Corollary~3.30]{RS82}, $f'(X)$ is a deformation retract of $Z$, i.e., there exists a retraction map
            \[
                \tilde r: U \to f'(X)\,.
            \]
            Since $d = \operatorname{dist}(f'(X), \partial U) > 0$, we can find a large integer $l = l(k, d) \in \mathbb{N}^+$ such that for every (closed) $(2k+1)$-cell $\alpha$ of $I(2k+1, l)$, if $|\alpha| \cap f'(X) \neq \emptyset$, then $|\alpha| \subset U$, and we set $Y$ to be the union of all such $|\alpha|$. It is obvious that
            \[
                f'(X) \subset Y \subset U\,.
            \]
            
            Therefore, $f'$ and $\tilde r$ induce the embedding $\iota: X \to Y$ and the retraction $r \coloneqq  f'^{-1} \circ \tilde r \vert_{Y} : Y \to X$.
        \end{proof}

\section{Deformations}\label{sect:deformation}

    In this subsection, we adapt some crucial technical deformation constructions from Pitts \cite{Pit81} and Marques-Neves \cite{MN21} to our setting. Later, we will apply these constructions to improve sweepouts.  

    First, we have the following lemma regarding pull-tight with a mass upper bound.
    \begin{lem}[Pull-tight]\label{lem:pull-tight}
        Given $c > 0$, we define 
        \[
        \begin{aligned}
            \mathcal{V}^{\leq c} &\coloneqq \{V \in \mathcal{V}_n(M):\|V\|_g(M) \leq c\}\,,\\
            \mathcal{SV}^{\leq c} &\coloneqq \{V \in \mathcal{V}^{\leq c} :V \text{ is stationary in }(M, g)\}\,,\\
            \cZ^{\leq c} &\coloneqq \cZ_n(M; \bF_g; \Z_2) \cap \{T :|T| \in \mathcal{V}^{\leq c}\}\,.
        \end{aligned}
        \]
        Then there exist continuous maps,
        \[
        \begin{aligned}
            \bar F^\textsc{PT}:& [0, 1] \times \mathcal{V}^{\leq c} \to \mathcal{V}^{\leq c}\,,\\
            F^\textsc{PT}:& [0, 1] \times \cZ^{\leq c} \to \cZ^{\leq c}\,,\\
        \end{aligned}
        \]
        such that 
        \begin{enumerate}
            \item For all $V \in \mathcal{V}^{\leq c}$, $\bar F^\textsc{PT}(0, V) = V$;
            \item For all $t \in [0, 1]$, $\bar F^\textsc{PT}(t, V) = V$ if $V \in \mathcal{SV}^{\leq c}$;
            \item For all $t \in (0, 1]$, $\|\bar F^\textsc{PT}(t, V)\|_g(M) < \|V\|_g(M)$ if $V \notin \mathcal{SV}^{\leq c}$;
            \item Furthermore, for each $t \in [0, 1]$ and each $S \in \cZ^{\leq c}$,
            \[
                |F^\textsc{PT}(t, S)| = \bar F^\textsc{PT}(t, |S|)\,.
            \]
        \end{enumerate}
    \end{lem}
    \begin{proof}
        The proof is essentially the same as that of \cite[Theorem~4.3]{Pit81} (see also Sect. 15 of \cite{MN14}).
    \end{proof}

    \begin{cor}[Pulled-tight sequence]\label{cor:PT-seq}
        Given a constant $c > 0$ and a sequence of finite simplicial complices $(X_i)^\infty_{i = 1}$, let $(\Phi_i: X_i \to \Zc_n(M; \Fb_g; \Z_2))^\infty_{i = 1}$ be a sequence of $\Fb_g$-continuous maps such that
        \[
        \begin{aligned}
            L = \limsup_{i} \sup_{x \in X_i} \Mb_g \circ \Phi_i(x) &> 0\,,\\
            \sup_i \sup_{x \in X_i} \Mb_g \circ \Phi_i(x) &\leq c\,.
        \end{aligned}
        \]
        Then there exists a sequence $(\Phi'_i: X_i \to \Zc_n(M; \Fb_g; \Z_2))^\infty_{i = 1}$ such that
        \begin{enumerate}
            \item For each $i \in \mathbb{N}^+$, there exists a homotopy map 
            \[
                H^{\textsc{PT}}_i: [0, 1] \times X_i \to \Zc_n(M; \bF_g; \Z_2)
            \]
            with $H^{\textsc{PT}}_i(0, \cdot) = \Phi_i(\cdot), H^{\textsc{PT}}_i(1, \cdot) = \Phi'_i(\cdot)$ satisfying that for each $(t, x) \in [0, 1] \times X_i$,
            \[
                \Mb_g (H^{\textsc{PT}}_i(t, x)) \leq \Mb_g(\Phi_i(x))\,;
            \]
            \item The set of varifolds
            \[
                \cV^{L}((\Phi'_i)_i)\coloneqq \{V=\lim_j|\Phi'_{i_j}(x_j)|:\N^+ \ni i_j \nearrow \infty, x_j\in X_{i_j},\|V\|_g(M)=L\}
            \]
            is a subset of 
            \[
                \cV^{L}((\Phi_i)_i)\coloneqq \{V=\lim_j|\Phi_{i_j}(x_j)|:\N^+ \ni i_j \nearrow \infty, x_j\in X_{i_j},\|V\|_g(M)=L\};
            \]
            
            \item $\cV^{L}((\Phi'_i)_i) \subset \cS\cV^{\leq c}$, i.e., $\cV^{L}((\Phi'_i)_i)$ only contains stationary varifolds.
        \end{enumerate}
    \end{cor}
    \begin{proof}
        By Lemma \ref{lem:pull-tight}, we can define $H^\textsc{PT}_i$ by
        \[
            H^\textsc{PT}_i(t, x) \coloneqq F^\textsc{PT}(t, \Phi_i(x))\,. 
        \]
        The properties of $F^\textsc{PT}$ and $\bar{F}^\textsc{PT}$ immediately yield all the stated conclusions.
    \end{proof}
    
    The following is a continuous version of \cite[Theorem 4.6]{MN21}.
    
    \begin{lem}[$(\varepsilon, \delta)$-deformation]\label{lem:(e, d)-deformation}
        Let $R, \bar\varepsilon, \eta, s> 0$ be constants such that $\bar{\varepsilon} < 2R$, and $\mathcal{W} \subset \mathcal{V}_n(M)$.
        Let $X$ be a pure finite simplicial $k$-complex, $\Phi: X \to \Zc_n(M; \Mb_g; \Z_2)$ be a continuous map and $L = \sup_{x \in X} \mathbf{M}_g(\Phi(x))$ such that if $x \in X$ satisfies
        \[
            \Mb_g(\Phi(x)) \geq L - \bar{\varepsilon}, \quad \mathbf{F}_g(|\Phi(x)|, \mathcal{W}) \geq R\,,
        \]
        then $\Phi(x)$ satisfies \emph{annular $(\bar \varepsilon, \delta)$-deformation conditions}, i.e., there exist $p(x) \in M$ and $I_{2k+1}$ positive numbers 
        \[
            r_1(x), \cdots, r_{I_{2k+1}}(x), s_1(x), \cdots, s_{I_{2k+1}}(x)
        \]
        satisfying
        \[
        \begin{aligned}
            s_i(x) &\geq s, \quad i = 1, \cdots, I_{2k + 1} - 1\\
            r_i(x) - 2s_i(x) &> 2(r_{i+1}(x) + 2s_{i+1}(x)), \quad i = 1, \cdots, I_{2k+1} - 1\\
            r_1(x) + 2s_1(x) &< \eta,\\
            r_{I_{2k+1}}(x) - 2s_{I_{2k+1}}(x) &> 0,\\
        \end{aligned}
        \]
        such that $\Phi(x)$ admits an $(\bar{\varepsilon}, \delta)$-deformation in each annulus
        \[
            \mathrm{An}_g(p(x), r_i(x) - s_i(x), r_i(x) + s_i(x)) \cap M,
        \]
        $i = 1, \cdots, I_{2k+1}$, for every $\delta > 0$.

        Then for any $\bar \delta > 0$, there exists a continuous map
        \[
            \Phi^*: X \to \Zc_n(M; \Mb_g; \Z_2)
        \] for which the following hold.
        \begin{enumerate}
            \item There exists a homotopy map 
            \[
                H^{\textsc{DEF}}: [0, 1] \times X \to \mathcal{Z}_n(M; \Mb_g; \Z_2)
            \]
            with $H^{\textsc{DEF}}(0, \cdot) = \Phi$ and $H^{\textsc{DEF}}(1, \cdot) = \Phi^*$ satisfying that for each $(t, x) \in [0, 1] \times X$,
            \[
                \bM_g(H^{\textsc{DEF}}(t,x)) < \bM_g(\Phi(x)) + \bar \delta\,;
            \]
            \item For any $x \in X$ and $t \in [0, 1]$, there exist $\hat x = \hat x(x) \in X$ and $T_{t, x} \in \cZ_n(M; \mathbb{Z}_2)$ such that
            \[
                \Mb_g(H^\textsc{DEF}(t, x), T_{t, x}) < \bar \delta\,,
            \]
            \[
                \Mb_g(\Phi(\hat x), \Phi(x)) < \bar \delta\,,
            \]
            \[
                \Mb_g(\Phi(\hat x)) > \Mb_g(H^\textsc{DEF}(t, x)) - \bar \delta\,,
            \]
            and
            \[
                T_{t, x}\llcorner(M \setminus(\overline{B}_g(p_1, \eta) \cup \cdots \cup (\overline{B}_g(p_m, \eta))) = \Phi(\hat x)\llcorner(M \setminus(\overline{B}_g(p_1, \eta) \cup \cdots \cup (\overline{B}_g(p_m, \eta)))
            \]
            for some collection $\{p_1, \cdots, p_m\} \subset M$, $m \leq 3^{2k+1}$;
            \item If $\mathbf{M}_g(\Phi^*(x)) \geq L - \bar{\varepsilon}/10$, then
            \[
                \mathbf{F}_g(|\Phi(\hat x)|, \mathcal{W}) \leq 2R\,,
            \]
            where $\hat x$ is the same as that in (2).
        \end{enumerate}
    \end{lem}
    \begin{proof}
        Fix $\bar \delta > 0$.
    
        Since $X$ is a $k$-dimensional finite simplicial complex, by Lemma~\ref{lem:simp_to_cub}, there exists a cubical subcomplex $Y$ of $I(2k+1, l)$ for some $l \in \mathbb{N}^+$ and a retraction map
        \[
            r: Y \to X\,.
        \]
        Define $\Psi: Y \to \mathcal{Z}_n(M;\Mb_g; \Z_2)$ by $\Psi = \Phi \circ r$. It is easy to verify that $\Psi$ also satisfies all the assumptions for $\Phi$ in the lemma.

        For each $q \in \mathbb{N}^+$, we define $\psi_q: Y(q)_0 \to \mathcal{Z}_n(M; \Z_2)$, by
        \[
            \psi_q = \Psi\vert_{Y(q)_0}\,.
        \]
        Since $\Psi$ is continuous in the $\Mb_g$-topology, the fineness $\mathbf{f}(\psi_q) \to 0$ as $q \to \infty$. In the following, we shall subsequently choose $q$ larger and larger so as to apply interpolation propositions from the previous section.

        First, we choose $N_1 \in \mathbb{N}^+$, such that for all $q \geq N_1$,
        \[
            \mathbf{f}(\psi_q) < \min(\delta_{\ref{prop:almgren_ext}}, \min(\delta_{\ref{prop:homotopy_of_M_close}}, \bar \delta / (5 C_{\ref{prop:homotopy_of_M_close}} ))/(2C_{\ref{prop:almgren_ext}}))
        \]
        and for all $x, y \in Y$, if $x$ and $y$ lie in a common cell of $Y(q)$,
        \[
            \Mb_g(\Psi(x), \Psi(y)) < \min(\delta_{\ref{prop:homotopy_of_M_close}}, \bar\delta / (5 C_{\ref{prop:homotopy_of_M_close}}))/2\,.
        \]
        Then by Proposition \ref{prop:almgren_ext}, $\psi_q$ has the Almgren extension $\Psi_q$ and
        \[
            \sup_{y \in Y} \bM_g(\Psi(y) - \Psi_q(y)) < \min(\delta_{\ref{prop:homotopy_of_M_close}}, \bar\delta / (5 C_{\ref{prop:homotopy_of_M_close}})\,.
        \]
        It follows from Proposition \ref{prop:homotopy_of_M_close}, there exists a homotopy map 
        \[
            H^{(1)}_q:[0, 1] \times Y \to \Zc_n(M; \Mb_g; \Z_2)
        \]
        with $H^{(1)}_q(0, \cdot) = \Psi$ and $H^{(1)}_q(1, \cdot) = \Psi_q$, and for all $t \in [0, 1]$ and $y \in Y$,
        \[
        \begin{aligned}
            \Mb_g(H^{(1)}_q(t, y)) &\leq \Mb_g(\Psi(y)) + C_{\ref{prop:homotopy_of_M_close}} \sup_{y \in Y} \bM_g(\Psi(y) - \Psi_q(y))\\
                &< \Mb_g(\Psi(y)) + \bar\delta/5\,.
        \end{aligned}
        \]
        
        Secondly, we choose $N_2 \in \mathbb{N}^+$ with $N_2 > N_1$, such that for all $q \geq N_2$, the following condition from \cite[Theorem~4.6]{MN21} holds,
        \[
            (2k+1) \mathbf{f}(\psi_q)(1 + 4(3^{2k+1} - 1)s^{-1}) < \min\{\frac{\bar\varepsilon}{3^{2(2k+1)}8}, \gamma_{\text{iso}}\}\,,
        \]
        Consequently, \cite[Theorem~4.6]{MN21} implies that there exists $C = C(k, s) > 0$, an integer $q' > q$, and a map
        \[
            \psi^*_q: Y(q')_0 \to \mathcal{Z}_n(M; \mathbb{Z}_2)
        \]
        such that
        \begin{enumerate}[(i)]
            \item $\psi^*_q$ is $(Y, \mathbf{M}_g)$-homotopic to $\psi_q$, through a discrete homotopy
            \[
                h_q: I(1, l + q')_0 \times Y(q')_0 \to \mathcal{Z}_n(M; \Z_2)
            \]
            with fineness $\mathbf{f}(h_q) \leq C \mathbf{f}(\psi_q)$;
        \end{enumerate}
        and for all $(t, y) \in I(1, l + q')_0 \times Y(q')_0$, if $\hat y = \mathbf{n}(l + q', l + q)(y)$, then
        \begin{enumerate}[(i)]
            \setcounter{enumi}{1}
            \item $h_q(t, y)(M \setminus(\overline{B}_\eta(p_1) \cup \cdots \cup (\overline{B}_\eta(p_m))) = \psi_q(\hat y)\llcorner(M \setminus(\overline{B}_\eta(p_1) \cup \cdots \cup (\overline{B}_\eta(p_m)))$
            for some collection $\{p_1, \cdots, p_m\} \subset M$, $m \leq 3^{2k+1}$;
            \item $\Mb_g(h_q(t, y)) \leq \Mb_g(\psi(\hat y)) + 2 \cdot 3^{2(2k+1)}(2k+1)(1 + 4(3^{2k+1} - 1)s^{-1})\mathbf{f}(\psi_q)$;
            \item if $\Mb_g(\psi^*_q(y)) \geq L - \bar\varepsilon / 5$, then $\Fb_g(|\psi_q(\hat y)|, \mathcal{W}) \leq 2R$.
        \end{enumerate}

        Thirdly, we choose $N_3 \in \mathbb{N}^+$ with $N_3 > N_2$, such that for all $q \geq N_3$, 
        \[
            \mathbf{f}(h_q) \leq C\mathbf{f}(\psi_q) < \min(\eta_{\ref{prop:cts_Hotp_of_disc_Hotp}}, \bar \delta / (5C_{\ref{prop:cts_Hotp_of_disc_Hotp}}))\,,
        \]
        \[
            2 \cdot 3^{2(2k+1)}(2k+1)(1 + 4(3^{2k+1} - 1)s^{-1})\mathbf{f}(\psi_q) < \bar \delta / 5\,,
        \]
        and for all $x$ and $y$ which lie in the same cell of $Y(q)$,
        \[
            \Mb_g(\Psi(x) - \Psi(y)) < \bar \delta /5\,.
        \]
        Applying Proposition \ref{prop:cts_Hotp_of_disc_Hotp} to (i) above, we obtain a $\Mb_g$-continuous homotopy map
        \[
            H^{(2)}_q:[0,1] \times Y \to \Zc_n(M; \Mb_g; \Z_2)
        \]
        with $H^{(2)}_q(0, \cdot) = \Psi_q$ the Almgren extension of $\psi_q$, and $H^{(2)}_q(1, \cdot) = \Psi^*_q$ the Almgren extension of $\psi^*_q$. Furthermore, for all $t \in [0, 1]$, $y \in Y$, there exists $(t_0, y_0) \in I(1, l + q')_0 \times Y(q')_0$ such that
        $y$ and $y_0$ are in the same cell of $Y(q')$, and
        \[
        \begin{aligned}
            \Mb_g(H^{(2)}_q(t, y)) &\leq \Mb_g(h_q(t_0, y_0)) + C_{\ref{prop:cts_Hotp_of_disc_Hotp}} \mathbf{f}(h_q) \\
                &\leq \Mb_g(\psi_q(\hat y_0)) + 2 \cdot 3^{2(2k+1)}(2k+1)(1 + 4(3^{2k+1} - 1)s^{-1})\mathbf{f}(\psi_q) + C_{\ref{prop:cts_Hotp_of_disc_Hotp}} \mathbf{f}(h_q)\\
                &< \Mb_g(\psi_q(\hat y_0)) + 2\bar\delta / 5\\
                &< \Mb_g(\Psi(y)) + 3\bar\delta / 5
        \end{aligned}
        \]
        where we use (iii) in the second line and the fact that $y$ and $\hat y_0$ are in the same cell of $Y(q)$ in the last line.
        
        Now, concatenating $H^{(1)}_q$ and $H^{(2)}_q$, we obtain a homotopy
        \[
            H_q: [0, 1] \times Y \to \Zc_n(M; \Mb_g; \Z_2)
        \]
        between $\Psi$ and $\Psi^*_q$, and for all $(t, y) \in [0, 1] \times Y$,
        \begin{equation}\label{eqn:DEF_delta}
            \Mb_g(H_q(t,y)) < \bM(\Psi(y)) + \bar \delta\,.
        \end{equation}

        Finally, we choose $N_4 \in \mathbb{N}^+$ with $N_4 > N_3$ such that for all $q \geq N_4$,
        \[
            \mathbf{f}(h_q) \leq C \mathbf{f}(\psi_q) < \min(\bar \varepsilon, \bar\delta) / (10 C_{\ref{prop:almgren_ext}})\,,
        \]
        and for all $x$ and $y$ which lie in the same cell of $Y(q)$,
        \[
            \Fb_g(|\Psi(x)|, |\Psi(y)|) < R\,.
        \]
        
        Hence, we can fix a $q \geq N_4$, and set
        \[
            H^\textsc{DEF} \coloneqq H_q\vert_{[0, 1] \times X}, \quad \Phi^* \coloneqq \Psi^*_q\vert_X\,.
        \]
        The statement (1) follows immediately from \eqref{eqn:DEF_delta}. 
        
        To see (2), for any $x \in X$ and $t \in [0, 1]$, choose $y_0 \in Y(q')_0$ such that $x$ and $y_0$ lie in the same cell of $Y(q')$. If $H^\textsc{DEF}(x, t) = H^{(1)}_q(x, t_1)$, then we set $t_0 \coloneqq 0$; Otherwise $H^\textsc{DEF}(x, t) = H^{(2)}_q(x, t_2)$, and we choose $t_0 \in [0,1](l+q')_0$ such that $t_0$ and $t_2$ lie in the same cell of $[0,1](1+q')$. We let $T_{t, x} \coloneqq h_q(t_0, y_0)$. By Proposition \ref{prop:almgren_ext} and Proposition \ref{prop:homotopy_of_M_close},
        \begin{align*}
            \Mb_g(H^\textsc{DEF}(t, x), T_{t, x}) &= \Mb_g(H_q(t, x), h_q(t_0, y_0)) \\
                &< \max(C_{\ref{prop:homotopy_of_M_close}}\sup_{y \in Y}\Mb_g(\Psi(y) - \Psi_q(y)), C_{\ref{prop:almgren_ext}}\mathbf{f}(h_q)) \\
                &< \bar \delta / 5\,.
        \end{align*}
        It follows from (ii) above that
        \[
            h_q(t_0, y_0)(M \setminus(\overline{B}_\eta(p_1) \cup \cdots \cup (\overline{B}_\eta(p_m))) = \psi_q(\hat y_0)\llcorner(M \setminus(\overline{B}_\eta(p_1) \cup \cdots \cup (\overline{B}_\eta(p_m)))
        \]
        for some collection $\{p_1, \cdots, p_m\} \subset M$, $m \leq 3^{2k+1}$. Using the retraction map $r$, we let $\hat x = r(\hat y_0)$, then we obtain
        \[
            T_{t, x}\llcorner(M \setminus(\overline{B}_\eta(p_1) \cup \cdots \cup (\overline{B}_\eta(p_m))) = \Phi(\hat x)\llcorner(M \setminus(\overline{B}_\eta(p_1) \cup \cdots \cup (\overline{B}_\eta(p_m)))\,.
        \]
        In addition, by (iii),
        \begin{align*}
            \Mb_g(\Phi(\hat x)) &= \Mb_g(\Psi(\hat y_0)) \\
                &= \Mb_g(\psi_q(\hat y_0))\\
                &\geq \Mb_g(h_q(t_0, y_0)) - \bar \delta / 5\\
                &> \Mb_g(H^\textsc{DEF}(t, x)) - \bar \delta\,.
        \end{align*}
        Since $x$ and $y_0$ lie in the same cell of $Y(q')$
        \begin{align*}
            \Mb_g(\Phi(\hat x), \Phi(x)) = \Mb_g(\Psi(\hat y_0), \Psi(x)) < \bar \delta
        \end{align*}
        
        As for (3), if $\Mb_g(\Phi^*(x)) \geq L - \bar \varepsilon / 10$, we choose $y_0 \in Y(q')_0$ again such that $x$ and $y_0$ lie in the same cell of $Y(q')$. Since
        \[
        \begin{aligned}
            \Mb_g(\psi^*_q(y_0)) &\geq \Mb_g(\Psi^*_q(x)) - C_{\ref{prop:almgren_ext}} \mathbf{f}(h_q) \\
                &\geq \Mb_g(\Phi^*(x)) - \frac{\bar\varepsilon}{10}\\
                &\geq L - \frac{\bar\varepsilon}{5}\,,
        \end{aligned}
        \]
        by (iv) above, there exist $\hat y_0 \in Y(q)_0$, such that 
        \[
            \Fb_g(|\Phi(\hat x)|, \mathcal{W}) = \Fb_g(|\Psi(\hat y_0)|, \mathcal{W}) \leq 2R\,.
        \]
    \end{proof}
    
    \begin{cor}[$(\varepsilon,\delta)$-deformed sequence]\label{cor:(e, d)-deformation_seq}
        For any $c, D > 0$, there exists a positive constant $\eta_{\ref{cor:(e, d)-deformation_seq}} = \eta_{\ref{cor:(e, d)-deformation_seq}}(M, g, D, c) \in (0, D)$ with the following property.
        
        Let $\mathcal{W}\subset \mathcal{V}_n(M)$, $(X_i)^\infty_{i = 1}$ a sequence of $k$-dimensional finite simplicial complices and $(\Phi_i: X_i \to \Zc_n(M; \Fb_g; \Z_2))^\infty_{i = 1}$ be a sequence of $\Fb_g$-continuous maps such that
        \begin{itemize}
            \item $L = \limsup_{i} \sup_{x \in X_i} \Mb_g \circ \Phi_i(x) \in (0, c)$;
            \item $\Vc^L((\Phi_i)_i) \subset \bB^{\Fb_g}_{\eta_{\ref{cor:(e, d)-deformation_seq}}}(\mathcal{SV}^{L})$;
            \item No varifold in $\mathcal{V}^L((\Phi_i)_i) \setminus \bB^{\Fb_g}_{\eta_{\ref{cor:(e, d)-deformation_seq}}}(\mathcal{W})$ is $(2k+1, \eta_{\ref{cor:(e, d)-deformation_seq}})_g$-almost minimizing.
        \end{itemize}

        For any sequence $(\delta_i)^\infty_{i = 1} > 0$, there exists a sequence $(\Phi^*_i: X_i \to \Zc_n(M; \Fb_g; \Z_2))^\infty_{i = 1}$ such that
        \begin{enumerate}
            \item For each $i \in \N^+$, there exists a homotopy map
            \[
                H^\textsc{DEF}_i: [0, 1] \times X_i \to \Zc_n(M; \Fb_g; \Z_2)
            \]
            with $H^\textsc{DEF}_i(0, \cdot) = \Phi_i$, $H^\textsc{DEF}_i(1, \cdot) = \Phi^*_i$ satisfying that for each $(t, x) \in [0, 1] \times X_i$,
            \[
                \Mb_g(H^\textsc{DEF}_i(t,x)) < \Mb_g(\Phi_i(x)) + \delta_i\,;
            \]
            \item $\Vc^L((\Phi^*_i)_i) \subset \bB^{\bF_g}_{D}(\mathcal{W})$.
        \end{enumerate}
    \end{cor}
    \begin{rmk}
        Intuitively, given any sweepout and any subset of varifolds containing no almost minimizing varifold, we can perturb the sweepout away from the subset.
    \end{rmk}
    \begin{proof}
        By Proposition \ref{prop:M_cts_approx_F_cts}, for each $i \in \N^+$, there exists a $\Mb_g$-continuous sequence $(\Psi_i: X \to \Zc_n(M; \Mb_g; \Z_2)^\infty_{i = 1}$ such that $\Phi_i$ and $\Psi_i$ are homotopic through a $\Fb_g$-continuous homotopy $\tilde H_i$ which satisfies
        \[
            \Mb_g(\tilde H_i(x, t)) < \Mb_g(\Phi_i(x)) + \min(\delta_i, 1/i)
        \]
        for all $x$ and $t$. Clearly,
        \[
            L = \limsup_{i} \sup_{x \in X_i} \Mb_g \circ \Psi_i(x) \equiv \limsup_{i} L_i\,,
        \]
        and
        \[
            \Vc^L((\Phi_i)_i) = \Vc^L((\Psi_i)_i)\,.
        \]
    
        We define 
        \begin{align*}
            R \equiv \eta_{\ref{cor:(e, d)-deformation_seq}} &\coloneqq \min(D / 6, \eta_{\ref{lem:close_of_small_rep}}(M, \{g\}, D, c, 2k+1))\,,\\
            \bar\varepsilon_1 &\coloneqq \bar\varepsilon_{\ref{lem:not_am_varying_metric}}(M, g, 2k+1, R, R, c)\,,\\
            \eta &\coloneqq \eta_{\ref{lem:not_am_varying_metric}}(M, g, 2k+1, R, R, c)\,,\\
            s &\coloneqq \bar{s}_{\ref{lem:not_am_varying_metric}}\,,
        \end{align*}
        from Lemma \ref{lem:close_of_small_rep} and Lemma \ref{lem:not_am_varying_metric}.

        Since $\Vc^L((\Psi_i)_i)$ is compact and $\bB^{\Fb_g}_{\eta_{\ref{cor:(e, d)-deformation_seq}}}(\mathcal{SV}^{L}_n)$ is open, for sufficiently large $i$, there exists $\bar\varepsilon \in (0, \bar\varepsilon_1)$, such that if $\Mb_g(\Psi_i(x)) \geq L_i - \bar\varepsilon$, then 
        \[
            \Psi_i(x) \in \bB^{\Fb_g}_{\eta_{\ref{cor:(e, d)-deformation_seq}}}(\mathcal{SV}^L_n)\,.
        \]
        By Lemma \ref{lem:not_am_varying_metric} with $r = d = R$, $\Psi_i$ satisfies the annular $(\bar\varepsilon, \delta)$-deformation assumptions of Lemma \ref{lem:(e, d)-deformation} with $R, \bar\varepsilon, R, s$ and $\mathcal{W}$ defined above.

        Therefore, for each $i$, we can choose $\bar\delta \in (0, \min(\delta_i, \eta_{\ref{cor:(e, d)-deformation_seq}}, \bar\varepsilon / 2))$, we obtain a homotopy map
        \[
            \tilde H^\textsc{DEF}_i: [0, 1] \times X_i \to \Zc_n(M; \Fb_g; \Z_2)
        \]
        with $\tilde H^\textsc{DEF}_i(0, \cdot) = \Psi_i$, $\tilde H^\textsc{DEF}_i(1, \cdot) = \Psi^*_i$ satisfying that for each $(t, x) \in [0, 1] \times X_i$,
        \[
            \Mb_g(\tilde H^\textsc{DEF}_i(t,x)) < \Mb_g(\Psi_i(x)) + \bar \delta < \Mb_g(\Psi_i(x)) + \delta_i\,.
        \]
        Hence, concatenating $\tilde H_i$ and $\tilde H^\textsc{DEF}_i$ implies the conclusion (1).
        
        Moreover, if $\Mb_g(\Psi^*_i(x)) \geq L_i - \bar\varepsilon / 10$, by Lemma \ref{lem:(e, d)-deformation}, there exist $\hat x \in X_i$ and $T_{1, x} \in \Zc_n(M; \Z_2)$ with the following properties:
        \begin{enumerate}[(i)]
            \item $\Mb_g(\Psi^*_i(x), T_{1, x}) < \bar\delta < \eta_{\ref{cor:(e, d)-deformation_seq}}$ and $\Fb_g(\Psi^*_i(x), T_{1, x}) < \eta_{\ref{cor:(e, d)-deformation_seq}}$;
            \item $T_{1, x} = \Psi_i(\hat x)$ on $M \setminus \bigcup^m_{i = 1} \overline{B}_g(p_1, \eta)$, $m \leq 3^{2k + 1}$;
            \item $\Mb_g(\Psi_i(\hat x)) > \Mb_g(\Psi^*_i(x)) - \bar\delta \geq L_i - \bar\varepsilon$;
            \item $\bF_g(|\Psi_i(\hat x)|, \mathcal{W}) \leq 2R < D/3$\,.
        \end{enumerate}
        By (iii), $\Psi_i(\hat x) \in \bB^{\Fb_g}_{\eta_{\ref{cor:(e, d)-deformation_seq}}}(\mathcal{SV}^L_n)$. By (ii) and Lemma \ref{lem:close_of_small_rep}, 
        \[
            \bF_g(|\Psi_i(\hat x)|, |T_{1, x}|) < D/2\,.
        \]
        By definition of $\eta_{\ref{cor:(e, d)-deformation_seq}}$, we conclude that 
        \[
            |\Psi^*_i(x)| \subset \bB^{\bF_g}_{D}(\mathcal{W})\,.
        \]
        which is the conclusion (2).
    \end{proof}
    
\part{Main arguments}
\section{Min-max theorems}\label{sect:min-max}
    The concept of restrictive min-max theory was originally introduced by the second author in \cite[Section~2.3]{YangyangLi20_improved_morse_index} with the purpose of generating CMC hypersurfaces. In this work, we extend and apply this theory to our specific setting. At the end of this section, we will briefly explain the motivation of the mass restriction conditions and their role in proving the main theorems of the paper.

    In the following, let $X$ be a $k$-dimensional, finite, simplicial complex and $Z$ be a subcomplex of a refinement of $X$. Note that $Z$ can be the empty set.

\subsection{Restrictive homotopic min-max theory}
    \begin{defn}\label{def:restrictive_homotopy_class}
        In a closed Riemannian manifold $(M, g)$, given $\delta > 0$ and an $\bF_g$-continuous map $\Phi: X \to \Zc_n(M;\bF_g;\Z_2)$, we define the \textit{restrictive $(X, Z)$-homotopy class of $\Phi$ with an upper bound $\delta$}, denoted by $\Pi^\delta_g(\Phi)$, to be the set of $\bF_g$-continuous maps $\Psi: X \rightarrow \Zc_n(M;\bF_g;\Z_2)$ satisfying the following conditions:
            \begin{enumerate}
                \item Each $\Psi$ is homotopic to $\Phi$ in the $\bF_g$-topology, and
                \item The homotopy map $H: [0, 1] \times X \rightarrow \Zc_n(M;\bF_g;\Z_2)$ satisfies
                    \begin{equation}\label{eq:homotopyZ}
                        \sup_{t \in [0, 1], z\in Z} \bM_g(H(t,z)) < \sup_{z\in Z} \bM_g(\Phi(z)) + \delta\,,
                    \end{equation}
                    and
                    \begin{equation}
                        \sup_{t \in [0, 1], x \in X} \bM_g(H(t,x)) < \sup_{x \in 
                        X} \bM_g(\Phi(x)) + \delta\,.
                    \end{equation}
            \end{enumerate}
    \end{defn}

    \begin{defn}
        The {\it restrictive min-max width of $\Pi^\delta_g(\Phi)$} is defined as
        \[
          \mathbf{L}(\Pi^\delta_g(\Phi)) \coloneqq \inf_{\Psi\in \Pi^\delta_g(\Phi)} \sup_{x \in X} \bM_g\circ\Psi(x)\,.
        \]
    \end{defn}

    \begin{rmk}
        One might wonder why in (\ref{eq:homotopyZ}) we allow the homotopy to increase the mass by $\delta$ on $Z$, instead of requiring the map to be fixed on $Z$. The reason is that when doing interpolation, we need to allow the mass to increase slightly on $Z$ (c.f. (1.4) in X. Zhou's definition of relative homotopy class in \cite{Zho20}).
    \end{rmk}

    \begin{defn}
        A sequence of maps $(\Phi_i)^\infty_{i = 1}$ in $\Pi^\delta_g(\Phi)$ is called a \textit{minimizing sequence} for $\Pi^\delta_g(\Phi)$ if 
        \[
            \mathbf{L}(\Pi^\delta_g(\Phi)) = \limsup_{i \to \infty} \sup_{x \in X} \mathbf{M}_g\circ \Phi_i(x)\,.
        \]
        
        For a minimizing sequence $(\Phi_i)^\infty_{i = 1}$ for $\Pi^\delta_g(\Phi)$, we define its \textit{critical set} by 
        \[
            \bC((\Phi_i)_i)\coloneqq\{V=\lim_j|\Phi_{i_j}(x_j)|: \N^+ \ni i_j\nearrow \infty, x_j\in X,\| V\|_g(M)=\mathbf{L}(\Pi^\delta_g(\Phi))\}\,.
        \]
        Furthermore, the sequence is called \textit{pulled-tight} if every varifold in $\bC((\Phi_i)_i)$ is stationary.
    \end{defn}

    \begin{thm}[Restrictive homotopic min-max theorem for $\Pi^\delta_g(\Phi)$]\label{thm:homotopy_min_max}
        Given a closed Riemannian manifold $(M^{n + 1}, g)\,(2 \leq n \leq 6)$, $\delta > 0$, $D > 0$, and $\Phi: X \rightarrow \Zc_n(M;\bF_g;\Z_2)$, if
        \begin{equation}\label{eq:nonTrvialWidth}
            \mathbf{L}(\Pi^\delta_g(\Phi)) > 0\,,
        \end{equation}
        then there exists some sequence $(\Phi_{i})^\infty_{i = 1}$ in $\Pi^\delta_g(\Phi)$ such that:
        \begin{enumerate}[\normalfont(1)]
            \item $(\Phi_{i})^\infty_{i = 1}$ is a pulled-tight minimizing sequence for $\Pi^\delta_g(\Phi)$.
            \item The critical set $\bC((\Phi_{i})_i)$ contains a $(2k+1, r)_g$-almost minimizing varifold $V$ with $\lVert V\rVert_g(M) = \mathbf{L}(\Pi^\delta_g(\Phi))$ for some $r > 0$;
            \item $\spt(V)$ is a smooth, closed, embedded, minimal hypersurface.
            \item If $\mathcal{W}^L_g$ is the set of all the embedded minimal cycles of total measure $L\coloneqq\mathbf{L}(\Pi^\delta_g(\Phi))$, then
            \[
                \bC((\Phi_{i})_i) \subset \bB^{\Fb_g}_D(\mathcal{W}^L_g)\,.
            \]
            \item If we further assume that $Z \neq \emptyset$ and $\mathbf{L}(\Pi^\delta_g(\Phi)) > \sup_{x \in Z} \bM_g(\Phi(x)) + \delta$, then the corresponding homotopy maps $H_i$ for $\Phi_i$, defined in Definition~\ref{def:restrictive_homotopy_class}, satisfies
            \[
                \sup_i \sup_{t \in [0, 1], z \in Z} \bM_g(H_i(t,z)) < \bL(\Pi^\delta_g(\Phi))\,.
            \]
        \end{enumerate}
    \end{thm}

    \begin{proof}[Proof of Theorem \ref{thm:homotopy_min_max}]
        First, we pick a minimizing sequence $(\Psi_i: X \to \Zc_n(M; \Fb_g; \Z_2))^\infty_{i = 1}$ for $\Pi^\delta_g(\Phi)$. Without loss of generality, we can assume that 
        \[
            \sup_{x \in X} \mathbf{M}_g(\Psi_i(x)) \leq \sup_{x \in X} \mathbf{M}_g(\Phi(x))\,.
        \] Since $\mathbf{L}(\Pi^\delta_g(\Phi)) > 0$, applying the pull-tight process, Corollary \ref{cor:PT-seq}, to $(\Psi_i)_i$ with an upper bound $c = \sup_x \Mb_g(\Phi(x)) + \delta$, we obtain a new sequence 
        \[
            (\Psi'_i: X \to \Zc_n(M; \Fb_g; \Z_2))_i\,.
        \]
        This new sequence $(\Psi'_i)_i \subset \Pi^\delta_g(\Phi)$, because the homotopy map $H^\text{PT}_i$ in the pull-tight process doesn't increase the mass (Corollary \ref{cor:PT-seq} (1)).
        
        Since $(\Psi_i)_i$ is a minimizing sequence, 
        \[
            L = \limsup_i \sup_{x \in X} \Mb_g(\Psi_i(x)) = \bL(\Pi^\delta_g(\Phi))\,,
        \]
        and by Corollary~\ref{cor:PT-seq} (1) again,
        \[
            \limsup_i \sup_{x \in X} \Mb_g(\Psi'_i(x)) \leq \limsup_i \sup_{x \in X} \Mb_g(\Psi_i(x)) = \bL(\Pi^\delta_g(\Phi))\,.
        \]
        Hence, the definition of restrictive min-max width implies 
        \[
            L = \limsup_i \sup_{x \in X} \Mb_g(\Psi'_i(x))\,.
        \]
        It follows from  Corollary~\ref{cor:PT-seq} (2) and (3) that 
        \[
            \bC((\Psi_i)_i) = \Vc^L((\Psi_i)_i)
        \]
        and 
        \[
            \bC((\Psi'_i)_i) = \Vc^L((\Psi'_i)_i) \subset \cS\cV^{\leq c}_n\,,
        \]
        which implies that $(\Psi'_i)_i$ is a pulled-tight minimizing sequence.
        
        Next, since $H^\textsc{PT}_i$ is a pull-tight process, we have
        \[
            \sup_{x \in X} \mathbf{M}_g(\Psi'_i(x)) \leq \sup_{x \in X} \mathbf{M}_g(\Psi_i(x)) \leq \sup_{x \in X} \mathbf{M}_g(\Phi(x))\,.
        \]
        Hence, we choose a sequence $(\delta_i)^\infty_{i = 1}$ such that for each $i \in \N^+$,
        \[\begin{aligned}
            \delta_i < \delta - 
                \big(\sup_{z \in Z} \mathbf{M}_g(\Psi'_i(z)) - \sup_{z \in Z} \mathbf{M}_g(\Phi(z))\big)
        \end{aligned}\]
        and 
        \[
            \lim_{i \to \infty} \delta_i = 0\,.
        \]
        In addition, let $c = 2L$, $m = 2k + 1$, $\eta = \eta_{\ref{cor:(e, d)-deformation_seq}}(M, g, D, 2L)$ from Corollary \ref{cor:(e, d)-deformation_seq}, $\mathcal{W}^L_{g, m, \eta}$ be the set of all the $(m, \eta)_g$ almost-minimizing varifolds and 
        \[
            \mathcal{W} := \mathcal{W}^L_{g, m, \eta} \cap \bC((\Psi'_i)_i)\,.
        \]
        Hence, we can apply Corollary~\ref{cor:(e, d)-deformation_seq} to obtain a new sequence of sweepout $(\Psi^*_i)_i$, such that
        \[
            \bC((\Psi^*_i)_i) \subset \bB^{\Fb_g}_D(\mathcal{W}) \subset \bB^{\Fb_g}_D(\mathcal{W}^L_g)\,.
        \]
        Due to the choice of $\delta_i$, the homotopy map $H^\text{DEF}_i$ in Corollary~\ref{cor:(e, d)-deformation_seq} doesn't increase the mass too much, i.e.,
        \[
            \sup_t \bM_g(H^\text{DEF}_i(t,x)) < \bM_g(\Psi'_i(x)) + \delta_i < \bM_g(\Phi_i(x)) + \delta\,,
        \]
        and thus, we have $(\Psi^*_i)_i \subset \Pi^\delta_g(\Phi)$.
        In particular, this implies that $\mathcal{W}^L_{g, m, \eta} \cap \bC((\Psi'_i)_i) \neq \emptyset$. Analogously, by applying Corollary \ref{cor:(e, d)-deformation_seq} to $(\Psi^*_i)_i$ again, we can conclude that
        \[
            \mathcal{W}^L_{g, m, \eta} \cap \bC((\Psi^*_i)_i) \neq \emptyset\,,
        \]
        which implies (2) and thus, by Theorem \ref{thm:Schoen_Simon_reg}, (3).

        Finally, we apply Corollary \ref{cor:PT-seq} again to $(\Psi^*_i)_i$ to obtain a pulled-tight sequence $(\Phi_i)_i$ which satisfies (2), (3), and (4) as well. Note that $(\Phi_i)_i \subset \Pi^\delta_g(\Phi)$, as shown in the preceding argument. In addition, (5) follows from Definition~\ref{def:restrictive_homotopy_class}.
    \end{proof}

\subsection{Restrictive homological min-max theory}\label{thm_restrictive_min_max}

    Recall that in $(M^{n+1}, g)$, for each $p \in \mathbb{N}^+$, the min-max $p$-width is defined by
    \[
        \omega_p(M, g) = \inf_{\Phi \in \mathcal{P}_p} \sup_{x \in \mathrm{dmn}(\Phi)} \Mb_g \circ \Phi(x)\,,
    \]
    where 
    \[
        \mathcal{P}_p = \{\Phi: X \to \cZ_n(M; \mathbf{F}; \mathcal{Z}_2) \vert X \text{ is a finite simplicial complex and } \Phi^*(\bar{\lambda}^p) \neq 0\}\,.
    \]

    The effectiveness of the homotopic min-max theory in producing minimal hypersurfaces of $p$-width can be attributed to the following rationale. Given two continuous maps $\Phi: X \to \cZ_n(M; \mathbf{F}; \Z_2)$ and $\Phi': X \to \cZ_n(M; \mathbf{F}; \Z_2)$, by the homotopy theory, if there exists a homotopy map $H: [0, 1] \times X \to \cZ_n(M; \mathbf{F}; \mathcal{Z}_2)$ such that $H(0, \cdot) = \Phi(\cdot)$ and $H(1, \cdot) = \Phi'(\cdot)$, then
    \begin{equation}
        \Phi \in \mathcal{P}_p \iff \Phi' \in \mathcal{P}_p\,.
    \end{equation}

    However, from the definition of the admissible set $\mathcal{P}_p$, the condition $\Phi^*(\bar{\lambda}^p) \neq 0$ suggests that we should appeal to a homology/cohomology theory. In addition, following the approach of choosing a representative of $\sigma$ in \$cite[p.47]{MN21}, for each $\Phi \in \mathcal{P}_p$, we may restrict to a pure simplicial $p$-cycle $\tilde X \subset X$. The Almgren-Pitts min-max theory applies as before.
    
    We have the following two observations.

    \begin{lem}\label{lem:holg_p_width}
        Given $p \in \mathbb{N}^+$, let $W$ be a pure finite simplicial $(p+1)$-complex such that $\partial W = W^\alpha + W^\omega$ (with $\Z_2$ coefficients) where $W^\alpha$ and $W^\omega$ are both simplicial $p$-cycles. Given a continuous map $\Psi: W \to \mathcal{Z}_n(M;\mathbf{F}; \Z_2)$, we set $\Psi^\alpha \coloneqq \Psi\vert_{W^\alpha}$ and $\Psi^\omega \coloneqq \Psi\vert_{W^\omega}$. Then we have
        \[
            \Psi^\alpha \in \mathcal{P}_p \iff \Psi^\omega \in \mathcal{P}_p\,.
        \]
    \end{lem}

    \begin{proof}
        Let $A \coloneqq \Psi(W^\alpha)$, $B \coloneqq \Psi(W^\omega)$ and $C \coloneqq \Psi(W)$ be the corresponding singular chains in $C_*(\cZ_n(M; \mathbf{F}; \Z_2); \Z_2)$. Since both $W^\alpha$ and $W^\omega$ are cycles, $\partial A = \partial B = 0$ and we obtain
        \[
            a \coloneqq [A], b \coloneqq [B] \in H_p(\cZ_n(M; \mathbf{F}; \Z_2); \Z_2)\,.
        \]
        Moreover, $\partial W = W^\alpha + W^\omega$ implies that $\partial C = A + B\,,$
        and thus $a=b$.
        Hence, we have
        \[
            \Psi^\alpha \in \mathcal{P}_p \iff \langle\bar{\lambda}^p, a \rangle \neq 0 \iff \langle\bar{\lambda}^p, b \rangle \neq 0 \iff \Psi^\omega \in \mathcal{P}_p\,.
        \]        
    \end{proof}

    \begin{lem}\label{lem:hotp_to_holg}
        Given $k \in \mathbb{N}^+$, let $X$ be a pure finite simplicial $k$-complex with boundary $Z = \partial X$ (possibly empty). Then $W = [0, 1] \times X$ is a pure finite simplicial $(k + 1)$-complex such that $\partial W = W^\alpha + W^\omega$ (with $\Z_2$ coefficients) where 
        \[
            W^\alpha \coloneqq \{0\} \times X
        \] and 
        \[
            W^\omega \coloneqq \{1\} \times X + [0, 1] \times Z
        \] are both pure finite simplicial $k$-complex with boundary $\{0\} \times Z$.
    \end{lem}
    \begin{proof}
        This follows immediately from the definition.
    \end{proof}

    \begin{defn}
        Given $\delta > 0$ and an $\Fb$-continuous map $\Phi:X \to\Zc_n(M;\bF_g;\Z_2)$, where $X$ is a pure finite simplicial $k$-complex with boundary $Z = \partial X$ (possibly empty), we define $\tilde\HC^\delta_{g}(\Phi)$ to be the set of all $\bF_g$-continuous 
        maps $\Psi: W \to \mathcal{Z}_n(M; \bF_g; \mathbb{Z}_2)$ such that:
        \begin{itemize}
            \item  $W$ is a pure finite simplicial $(k+1)$-complex such that $\partial W=W^\alpha + W^\omega$ (with $\Z_2$ coefficients) where $W^\alpha=X$ and $W^\omega$ is another pure finite simplicial $k$-complex with $\partial W^\omega = Z$;
            \item $\Psi^\alpha \coloneqq \Psi|_{W^\alpha }=\Phi$, and $\Psi^\omega\coloneqq\Psi|_{W^\omega}$;
            \item 
            \[
                \sup_{x \in W} \mathbf{M}_g\circ \Psi(x) < \sup_{x \in X} \mathbf{M}_g\circ \Phi(w) + \delta\,.
            \]
        \end{itemize}

        We define
        \[
            \HC^\delta_g(\Phi)\coloneqq\{\Psi^\omega:\Psi\in\tilde\HC^\delta_g(\Phi)\}\,.
        \]
    \end{defn}
    
    We should think of  $\Psi$ as a cobordism between ``the beginning" $\Phi$ and ``the loose end" $\Psi|_{W^\omega}$. The set $\HC^\delta_g(\Phi)$ is the homology class represented by $\Phi$.

    \begin{defn}
        The {\it restrictive min-max width of $\HC^\delta(\Phi)$} is defined as
        \[
          \bL_g(\HC^\delta_g(\Phi)) \coloneqq \inf_{\Phi'\in \HC^\delta_g(\Phi)}\sup \bM_g\circ\Phi'\;\;\left(=\inf_{\Psi\in\tilde\HC^\delta_g(\Phi)} \sup \mathbf{M}_g\circ \Psi^\omega\right).
        \]
    \end{defn}

    Lemma \ref{lem:holg_p_width} implies that this width is nontrivial provided that $\Phi$ is some $p$-admissible sweepout.

    \begin{cor}\label{cor:relation_to_p_width}
        If $\Phi: X \to \cZ_n(M; \bF_g; \Z_2) \in \cP_p$ with $X$ a pure finite simplicial $p$-cycle, then every element of $ \HC^\delta_g(\Phi)$ is also inside $\cP_p$. In particular,
        \[
            \bL_g(\HC^\delta_g(\Phi))\geq \omega_p(M, g)\,.
        \]
    \end{cor}

    \begin{defn}
        A sequence of maps $(\Phi_i)^\infty_{i = 1}$ in $\HC^\delta_g(\Phi)$ is called a \textit{minimizing sequence} for $\HC^\delta_g(\Phi)$ if 
        \[
            \mathbf{L}(\HC^\delta_g(\Phi)) = \limsup_{i \to \infty} \sup \mathbf{M}_g\circ \Phi_i\,.
        \]
        For convenience, we  also say that a sequence $(\Psi_i)_i$ in $\tilde\HC^\delta_g(\Phi)$ is a \emph{minimizing sequence} for $\tilde\HC^\delta_g(\Phi)$ if $(\Psi^\omega_i)_i$ is a minimizing sequence for $\HC^\delta_g(\Phi)$.
        
        For a minimizing sequence $(\Phi_i)^\infty_{i = 1}$ for $\HC^\delta_g(\Phi)$, we define its \textit{critical set} by 
        \[
            \bC_g((\Phi_i)_i)\coloneqq\{V=\lim_j|\Phi_{i_j}(x_j)|:\{i_j\}_j\subset\mathbb N, x_j\in \dmn(\Phi_{i_j}),\| V\|_g(M)=\mathbf{L}(\HC^\delta_g(\Phi))\}\,.
        \]
        Furthermore, the sequence is called \textit{pulled-tight} if every varifold in $\bC((\Phi_i)_i)$ is stationary.
    \end{defn}

    \begin{nota}\label{nota:critical}
        For notational convenience, given {\it any} sequence of maps $(\Psi_i:W_i\to\cZ_n(M;\bF_g;\Z_2))_i$, we denote
        \begin{align*}
        \bC_g((\Psi_i)_i)\coloneqq\{V=\lim_j|\Psi_{i_j}(x_j)|:&\;\{i_j\}_j\subset\mathbb N, x_j\in W_{i_j},\\
        &\| V\|_g(M)=\limsup_{i\to\infty}\sup\bM_g\circ\Psi_i\}\,.
        \end{align*}
        \end{nota}

    Let us now state the homological min-max theorem we need. 

    \begin{thm}[Restrictive homological min-max theorem]\label{thm_restrictive_holo_min_max}
        Given $\delta > 0$, $D > 0$ and an $\Fb_g$-continuous map $\Phi:X \to\Zc_n(M;\Fb_g;\Z_2)$, where $X$ is a pure finite simplicial $k$-complex with boundary $Z = \partial X$, suppose that 
        \begin{equation}\label{eqn:mass_gap}
            L\coloneqq\bL_g( \HC^\delta_g(\Phi)) > \max(\sup_{x \in Z} \bM_g(\Phi(x)) + \delta, 0)\,.
        \end{equation} Then there exists a minimizing sequence 
        \[
            (\Psi_{i}:W_i\to\Zc_n(M;\Fb_g;\Z_2))^\infty_{i = 1}
        \]
        in $\tilde \HC^\delta_g(\Phi)$ such that:
        \begin{enumerate}
            \item $(\Psi_{i}^\omega)^\infty_{i = 1}$ is a pulled-tight minimizing sequence for $\HC^\delta_g(\Phi)$.
            \item The critical set $\bC((\Psi_{i}^\omega)_i)$ contains a $(2k+1, r)_g$-almost minimizing varifold $V$ with $\lVert V\rVert_g(M) = \mathbf{L}(\Pi^\delta_g(\Phi))$ for some $r > 0$;
            \item $\spt(V)$ is a smooth, closed, embedded, minimal hypersurface.
            \item If $\mathcal{W}^L_g$ is the set of all the embedded minimal cycles of area $L$, then
            \[
                \bC((\Psi_{i}^\omega)_i) \subset \bB^{\Fb_g}_D(\mathcal{W}^L_g)\,.
            \]
        \end{enumerate}
    \end{thm} 
    \begin{proof}
        First, we pick a minimizing sequence $(\Psi'_i: W'_i \to \cZ_n(M; \mathbf{F}; \Z_2))^\infty_{i=1}$ in $\tilde \HC^\delta_g(\Phi)$. Without loss of generality, we assume that for every $i$,
        \[
            \sup_x \Mb_g({\Psi'_i}^\omega(x)) \leq \sup_x \Mb_g({\Phi}(x))\,.
        \]
        Since 
        \[
            L\coloneqq\bL_g( \HC^\delta_g(\Phi)) > 0\,,
        \] 
        we can apply Corollary \ref{cor:PT-seq} to each ${\Psi'_i}^\omega$ (recall the notation $^\omega$ in \S \ref{sect:nota}) with $c = \sup_{x \in X} \Mb_g \circ \Phi(x) + \delta$, and obtain \[
            H^\textsc{PT}_i: [0, 1] \times {W'_i}^\omega \to \Zc_n(M; \Fb_g; \Z_2)
        \]
        such that for all $(t, x) \in [0, 1] \times {W'_i}^\omega$,
        \[
            \bM_g(H^\textsc{PT}_i(t, x)) \leq \bM_g({\Psi'_i}^\omega(x))\,.
        \]
        
        Then, we define a new parameter space
        \[
            W''_i \coloneqq W'_i \cup [0, 1] \times {W'_i}^\omega\,,
        \]
        by identifying ${W'_i}^\omega$ in $W'_i$ and $\{0\} \times {W'_i}^\omega$ (see Figure~\ref{fig:W''}), and a $\Fb_g$-continuous map in $\tilde{\mathcal{H}}^\delta_g(\Phi)$,
        \[
            \Psi''_{i}: W''_i\to\Zc_n(M;\Fb_g;\Z_2)
        \]
        by concatenating $\Psi'_i$ and $H^\textsc{PT}_i$.
        By Lemma \ref{lem:hotp_to_holg}, we denote ${W''_i}^\alpha = X$ and ${W''_i}^\omega = [0, 1] \times Z \cup \{1\} \times {W'_i}^\omega$. Since the homotopy map $H^{\text{PT}}_i$ in the pull-tight process doesn't increase the mass
        , and in $[0, 1] \times Z \subset {W''_i}^\omega$, by~\eqref{eqn:mass_gap}, 
        \[
        \begin{aligned}
            \sup_{x \in [0, 1] \times Z} \Mb_g({\Psi''_i}^\omega(x)) &= \sup_{x \in [0, 1] \times Z} \Mb_g(H^\textsc{PT}_i(x))\\
                &\leq \sup_{z \in Z} \Mb_g(\Phi(z)) < L - \delta\,,
        \end{aligned}
        \]
        we can conclude that
        $\mathcal{V}^L(({\Psi''_i}^\omega)_i) = \mathcal{V}^L((H^\textsc{PT}_i(1))_i)$.
        Therefore, $({\Psi''_i}^\omega)_i$ is a pulled-tight minimizing sequence in $\mathcal{H}^\delta_g(\Phi)$. Moreover, we have
        \[
            \sup_x \mathbf{M}_g({\Psi''_i}^\omega(x)) \leq \sup_x \mathbf{M}_g(\Psi'_i(x)) \leq \sup_x \Mb_g(\Phi_i(x))\,.
        \]

        \begin{figure}
            \centering
             
            \includegraphics[width=4in]{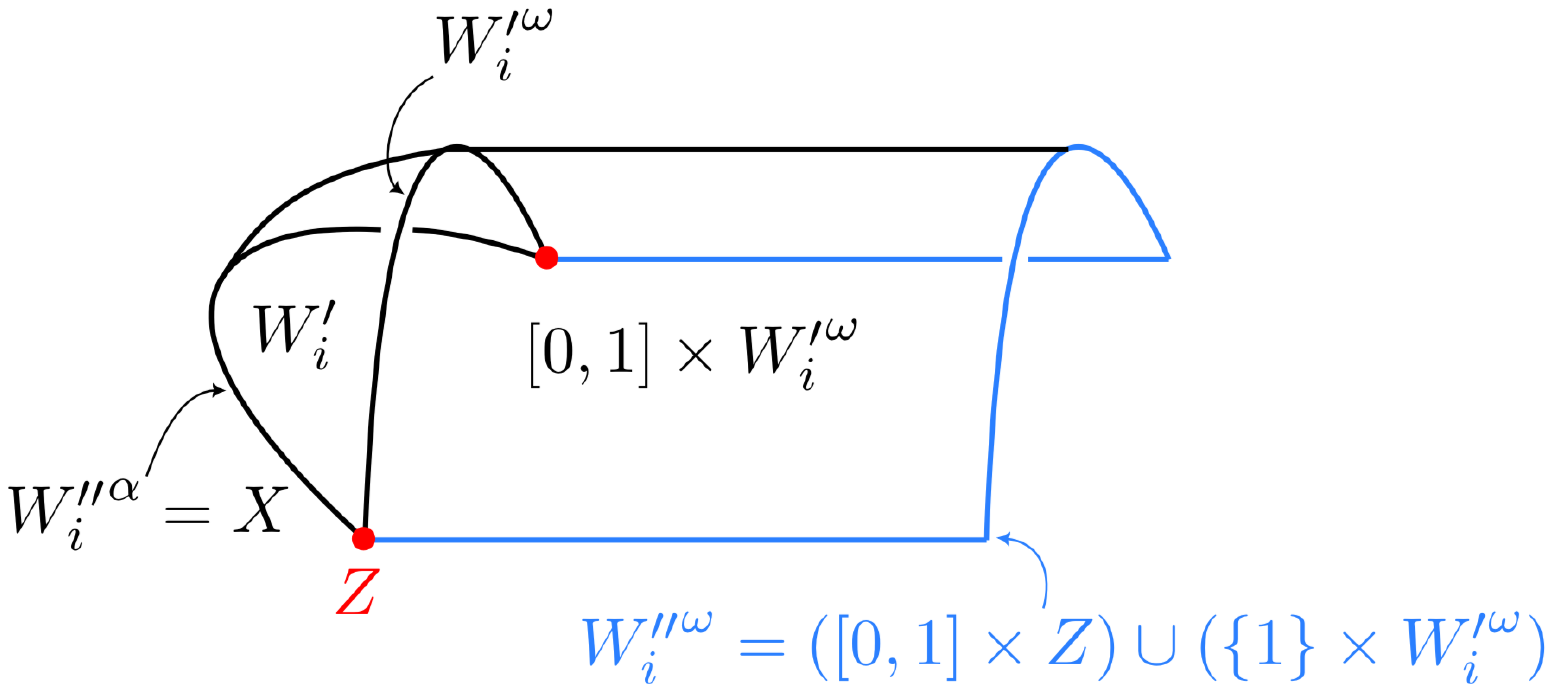}
            \caption{Construction of $W''_i=\dmn(\Psi''^\omega_i)$.}
            \label{fig:W''}
        \end{figure}
    
        Next, for each $i$, by the previous inequality, we can choose $\delta_{i} < \delta / 2$,
        such that $\lim_i \delta_i = 0$. As in the proof of Theorem \ref{thm:homotopy_min_max}, let $c = 2L$, $m = 2k + 1$, $\eta = \eta_{\ref{cor:(e, d)-deformation_seq}}(M, g, D, 2L)$ from Corollary \ref{cor:(e, d)-deformation_seq}, $\mathcal{W}^L_{g, m, \eta}$ be the set of all the $(m, \eta)_g$-almost minimizing varifolds and 
        \[
            \mathcal{W} = \mathcal{W}^L_{g, m, \eta} \cap \bC(({\Psi''_i}^\omega)_i)\,.
        \]
        Hence, we can apply Corollary \ref{cor:(e, d)-deformation_seq} to $({\Psi''_i}^\omega)_i$ to obtain a new sequence of sweepout $(\Phi^*_i)^\infty_{i = 1}$ and homotopy maps $H^\textsc{DEF}_i$ such that
        \[
            \mathcal{V}^L((\Phi^*_i)_i) \subset \bB^{\Fb_g}_D(\mathcal{W}) \subset \bB^{\Fb_g}_D(\mathcal{W}^L_g)\,
        \]
        and
        \begin{equation}\label{eqn:restr_min_max_unique}
            \Mb_g(H^\textsc{DEF}_i(t,x)) < \Mb_g({\Psi''_i}^\omega(x)) + \delta_i \leq \sup_{x \in X} \Mb_g(\Phi(x)) + \delta_i\,.
        \end{equation}

        We define a space
        \[
            W^{(3)}_i \coloneqq W''_i \cup [0, 1] \times {W''_i}^\omega
        \]
        by identifying ${W''_i}^\omega$ in $W''_i$ and $\{0\} \times {W''_i}^\omega$, and define a $\Fb_g$-continuous map in $\tilde{\mathcal{H}}^\delta_g(\Phi)$,
        \[
            \Psi^{(3)}_{i}: W^{(3)}_i \to \Zc_n(M;\Fb_g;\Z_2)
        \]
        by concatenating $\Psi''_i$ and $H^\textsc{DEF}_i$. By Lemma \ref{lem:hotp_to_holg}, we denote ${W^{(3)\alpha}_i} = X$ and ${W^{(3)\omega}_i} = [0, 1] \times Z \cup \{1\} \times {W''_i}^\omega$. The choice of $\delta_i$ guarantees that the increment of the mass in $H^\textsc{DEF}_i$ satisfies the mass restriction, i.e., $\Psi^{(3)}_i \in \tilde{ \mathcal{H}}^\delta(\Phi)$. Moreover, by ~\eqref{eqn:mass_gap} again, we have
        \[
            \sup_{[0, 1] \times Z} \bM_g\left({\Psi^{(3)}_i}^\omega\right) \leq \sup_Z \bM_g(\Phi) + \delta_i < L - \delta/2\,,
        \]
        and thus,
        \[
            \Vc^L(({\Psi^{(3)\omega}_{i}} )_i) = \mathcal{V}^L((\Phi^*_i)_i)\,.
        \]
        Together with \eqref{eqn:restr_min_max_unique}, we can conclude that $(\Psi^{(3)}_{i})_i$ is a minimizing sequence in $\tilde{\HC}^\delta_g(\Phi)$, and
        \[
            \bC(({\Psi^{(3)\omega}_{i}} )_i) = \Vc^L(({\Psi^{(3)\omega}_{i}} )_i) = \mathcal{V}^L((\Phi^*_i)_i) \subset \bB^{\Fb_g}_D(\mathcal{W}) \subset \bB^{\Fb_g}_D(\mathcal{W}^L_g)\,.
        \]
        In particular, this also implies that $\mathcal{W}^L_{g, m, \eta} \cap \bC(({\Psi''_{i}}^\omega)_i) \neq \emptyset$, and analogously, by applying Corollary \ref{cor:(e, d)-deformation_seq} to $({\Psi^{(3)\omega}_{i}} )_i$, we also have
        \[
            \mathcal{W}^L_{g, m, \eta} \cap \bC(({\Psi^{(3)\omega}_{i}} )_i) \neq \emptyset\,,
        \]
        which concludes (2) and thus, by Theorem \ref{thm:Schoen_Simon_reg}, (3).

        Finally, we apply Corollary \ref{cor:PT-seq} again to pull tight $({\Psi^{(3)\omega}_{i}} )_i$ as in the first step, so the concatenation induces a minimizing sequence $({\Psi}_i)_i$ where $(\Psi^\omega_i)_i$ is pulled tight, which (2), (3) and (4) as well. This concludes the theorem.
    \end{proof}

\begin{rmk} Let us make a remark regarding the role of mass restriction.
In all the statements of Theorem~\ref{thm_13_width} (implicit in the definition of $p$-width),  Theorem~\ref{thm:strong_multi_one_I} and Theorem~\ref{thm:strong_multi_one_II}, no mass restrictions are imposed on the homology classes, i.e., we take $\partial X = \emptyset$ and $\delta = \infty$. 
    
    In their proofs, however (see, for example, the sketches in \S\ref{subsubsect:Thm1} and \S\ref{subsubsect:Thm3}), we begin by connecting the given sweepout $\Phi_0$ with an improved sweepout $\Phi_1$ in a perturbed metric (Proposition~\ref{prop_baire_residual}), and then analyze the ``cobordism'' $\Psi$ bridging them. Using restrictive min-max, we can control the maximal area on  $\Psi$ and have good characterization of those  members with large area: See the second bullet point in the fourth paragraph of \S \ref{subsubsect:Thm1}.
\end{rmk}

\section{Proof of Theorem \ref{thm:strong_multi_one_I}}\label{sect:proof_main_thm_I}
    Let $(M^{n+1}, g)$ ($3 \leq n + 1 \leq 7$) be a closed Riemannian manifold equipped with a bumpy metric or a metric of positive Ricci curvature. Fix an $\bF$-continuous map $\Phi:X \to\Zc_n(M; \Fb_g; \Z_2)$, where  $X$ is a pure finite simplicial $k$-cycle with $\partial X = \emptyset$, such that $L \coloneqq \bL_g(\cH(\Phi))>0$. Note that by Lemma~\ref{lem:simp_to_cub}, $X$ can be viewed as a cubical subcomplex of ambient dimension $2k + 1$, and a pure finite simplicial $(k + 1)$-complex can be viewed as a cubical subcomplex of ambient dimension $2k + 3$.
    
    We define 
    \[
        r_0 \coloneqq \min(\eta_{\ref{lem:balls_not_cover_SV}}(M, g, 2k + 3), \eta_{\ref{lem:cycles_remove_balls}}(M, g, 2k + 3))
    \]
    from Lemmas \ref{lem:balls_not_cover_SV} and \ref{lem:cycles_remove_balls}, and define the following sets of varifolds on $(M, g)$:
    \begin{itemize}
        \item $\mathcal{SV}^{L}_{g} \subset \mathcal{V}_n(M)$ is the set of all stationary $n$-varifolds on $(M, g)$ with total measure $L$;
        \item $\mathcal{W}^{L}_{g} \subset \mathcal{SV}^{L}_{g}$ is the subset consisting of all stationary integral varifolds whose support is a smooth, embedded, closed minimal surface;
        \item $\cG_{g} \subset \mathcal{W}^{L}_{g, m+1, r_0}$ is the (good) subset comprising all $(2k + 3, r_0)_g$-almost minimizing varifolds $V$ for which there is $T \in \Zc_n(M;\Z_2)$ with $V = |T|$;
        \item $\cB_{g} \subset \cW^{L}_{g, m+1, r_0}$ is the (bad) subset comprising all $(2k + 3, r_0)_{g}$-almost minimizing varifolds $V$ for which no $T \in \Zc_2(M;\Z_2)$ with $V = |T|$.
    \end{itemize}
    With these definitions, it follows that $\cG_g \cup \cB_g = \cW^{L}_{g, m+1, r_0}$.

    Suppose, for the sake of contradiction, that there exists a minimizing sequence $(\Phi_i)_i$ of $\mathcal{H}(\Phi)$ with critical set $\bC((\Phi_i)_i)$ containing no varifold induced by a multiplicity one, smooth, embedded, separating minimal hypersurface. In other words,
    \begin{equation}\label{eqn:crit_contra}
        \bC((\Phi_i)_i) \cap \cW^{L}_{g, m+1, r_0} \subset \cB_{g}\,.
    \end{equation}
    In the following, for each $i$, we denote $X_i := \dmn(\Phi_i)$, which has dimension $k$.

    Since $g$ either is bumpy or has positive Ricci curvature, by Lemma \ref{lem:am_cpt_varying_metric}, both $\cG_g$ and $\cB_g$ are compact in the varifold topology. Therefore, we can define
    \begin{align*}
        d_0 &\coloneqq \eta_{\ref{lem:far_of_diff_geom}}(M, g, 2k + 3, L, \cG_g, \cB_{g}) / 10\,,\\
        \varepsilon_0 &\coloneqq \bar\varepsilon_{\ref{lem:not_am_varying_metric}}(M, g, 2k + 3, r_0, d_0, L)\,,\\
        s_0 &\coloneqq \bar{s}_{\ref{lem:not_am_varying_metric}}(M, g, 2k + 3, r_0, d_0, L)\,,\\
        \eta_0 &\coloneqq \min(d_0, \varepsilon_0, \eta_{\ref{lem:not_am_varying_metric}}(M, g, 2k + 3, r_0, d_0, L)) / 10\,,\\
        \tilde \cG_g &\coloneqq \bB^{\bF_{g}}_{\eta_0} (\cG_g)\,,\\
        \tilde \cB_{g} &\coloneqq \bB^{\bF_{g}}_{\eta_0} (\cB_{g})\,,\\
        \hat\cG_g &\coloneqq \cG'_{\ref{lem:far_of_diff_geom}}(M, g, 2k + 3, L, \cG_g, \cB_{g})\,,\\
        \hat\cB_{g} &\coloneqq \cB'_{\ref{lem:far_of_diff_geom}}(M, g, 2k + 3, L, \cG_g, \cB_{g})\,,
    \end{align*}
    from Lemma \ref{lem:far_of_diff_geom} and Lemma \ref{lem:not_am_varying_metric}. By \eqref{eqn:crit_contra}, we can apply Corollary~\ref{cor:(e, d)-deformation_seq} to $(\Phi_i)_i$ to obtain a new minimizing sequence, still denoted by $(\Phi_i)_i$, and there exists $\bar\varepsilon_0 > 0$ such that for sufficiently large $i$,
    \begin{equation}\label{eqn:g_bar_e_0}
        \Mb_{g} \circ \Phi_{i} (x) \geq L - \bar\varepsilon_0 \implies |\Phi_{i} (x)| \in \tilde{\cB}_{g}\,.
    \end{equation}

\medskip
\paragraph*{\textbf{Part 1. Metric perturbations.}}
    We begin with a proposition, which asserts the $\Z$-linear independency of area of minimal hypersurfaces for a generic metric.
    
    \begin{prop}\label{prop_baire_residual}
        Let $M^{n+1}$ be a smooth closed manifold with $3\leq n+1\leq 7$. Then there exists a Baire residual set $\Gamma^\infty_{\mathrm{uniq}}$ of $C^\infty$ bumpy Riemannian metrics on $M$ such that for any $g \in \Gamma^\infty_{\mathrm{uniq}}$ and any $L \in \mathbb{R}^+$, there exists at most one combination of minimal hypersurfaces (with multiplicities) whose total areas sum up to $L$.
    \end{prop}
    \begin{proof}
        For each metric $g$, let $\mathfrak M_g$ be the set of closed, embedded, smooth, minimal hypersurfaces in $(M,g)$.
        For each $\alpha>0$ and integer $p>0$, let us denote by $\cU_{p,\alpha}$ the set of smooth metrics on $M$ such that:
        \begin{enumerate}
            \item \label{prop_item_gammaC_nondeg} Every element of $\mathfrak M_g$ with index at most $p$ and area at most $\alpha$ is nondegenerate.
            \item \label{prop_item_gammaC_lin_indep} For any $p_1,\dots,p_N\in \Z$ and $\Sigma_1,\dots,\Sigma_N\in \mathfrak M_g$, where $|p_l|\leq p$, $\index(\Sigma_l)\leq p$, and $\area(\Sigma_l)\leq \alpha$ for each $l$, if
            \[
                p_1\area_g(\Sigma_1)+\dots+p_N\area_g(\Sigma_N)=0\,,
            \]
            then $p_1=\dots=p_N=0$.
        \end{enumerate}
        By \cite[Claim 8.6]{MN21}, for each $p$ and $\alpha$, the set $\cU_{p,\alpha}$ is open and dense in the space of all smooth metrics in $M$. Thus, we can let 
        \[
            \Gamma^\infty_{\mathrm{uniq}}\coloneqq\bigcap_{n\in \mathbb N^+}\cU_{n,n}\,,
        \]
        which is clearly a Baire residual set.
    \end{proof}

    Then, combined with X. Zhou's multiplicity one theorem \cite{Zho20}, this proposition implies that for each $g\in \Gamma^\infty_{\mathrm{uniq}}$, every (restrictive) min-max width can be realized by a unique combination of multiplicity one two-sided minimal hypersurfaces, which is also the boundary of a Caccioppoli set.

    We choose a sequence $(g_i)^\infty_{i = 1}$ in $\Gamma^\infty_{\text{uniq}}$ such that $\|g_i -  g\|_{C^\infty, g} < \eta_0$, and
    \[
        \lim_{i \to \infty} g_i = g
    \]
    in the $C^\infty$ topology. We define for each $i \in \N^+$,
    \[
        L_i \coloneqq \sup_x \Mb_{g_i} \circ \Phi_{i}(x)\,.
    \]
    Let $(\delta_i)^\infty_{i = 1}$ be a decreasing sequence in $\mathbb{R}^+$ such that 
    \[
        \lim_{i \to \infty}\delta_i = 0\,.
    \]

\medskip
\paragraph*{\textbf{Part 2. Restrictive homological min-max.}}
    We set $\eta_2 \coloneqq \eta_0$.

    For each $i \in \mathbb{N}^+$, we consider the restrictive homology class $\HC^{\delta_i}_{g_i}(\Phi_{i})$ of $\Phi_{i}$. For simplicity, we denote 
    \[
        \tilde\HC_i=\tilde \HC^{\delta_i}_{g_i}(\Phi_{i})\textrm{ and }\HC_i=\HC^{\delta_i}_{g_i}(\Phi_{i})\,.
    \]
    
    Note that, since $(\Phi_{i})_i$ is a minimizing sequence of $\mathcal{H}(\Phi)$, we have
    \[
        \lim_{i \to \infty} \bL_{g_i}(\HC_i) = \lim_{i \to \infty} (L_i + \delta_i) = L\,.
    \] 
    Thus, applying the restrictive min-max theorem, Theorem \ref{thm_restrictive_holo_min_max}, to each $\HC_i$,  we obtain for each $i$  a pulled-tight minimizing sequence 
    \[
        (\Psi^{j}_{i}:W^j_i\to\cZ_n(M;\bF_{g_i};\Z_2))^\infty_{j = 1}
    \] 
    in $\tilde \HC_i$ (note $\dim W^j_i=k+1$) such that:
    \begin{enumerate}[(i)]
        \item $\lim_{j \to \infty} \sup\mathbf{M}_{g_i}\circ (\Psi^{j}_{i})^\omega = \bL_{g_i}(\HC_i)\,.$
        \item There exists a constant $\varepsilon_{2,i}>0$ such that for all $j$ large enough,
            \begin{equation}\label{eqn:e0_of_e2i}
                \Mb_{g_i}((\Psi^j_i)^\alpha(x))\geq \bL_{g_i}(\HC_i)-\varepsilon_{2,i} \implies \Mb_{g}((\Psi^j_i)^\alpha(x))\geq L - \bar\varepsilon_0\,,
            \end{equation}
            and
            \begin{equation}\label{eqn:uniq_real}
                \Mb_{g_i}((\Psi^j_i)^\omega(x)) \geq \bL_{g_i}(\HC_i)-\varepsilon_{2,i} \implies |(\Psi^j_i)^\omega(x)|\in \bB^{\bF_{g_i}}_{1/i}(\Sigma_i),
            \end{equation}
        where $\Sigma_i$ is the unique multiplicity-one two-sided minimal surface with area $\bL_{g_i}(\HC_i)$, and it is the boundary of a Caccioppoli set. Indeed, the last statement follows from the fact that $g_i\in\Gamma^\infty_{\text{uniq}}$ and Proposition \ref{prop_baire_residual}.
        \item $|\Sigma_i|$ is $(2k + 1, r_0)_{g_i}$-almost minimizing and thus, $(2k + 3, r_0)_{g_i}$-almost minimizing.
    \end{enumerate}
    Note that, according to the definition of a minimizing sequence, 
    $$\partial W^j_i=(W^j_i)^\alpha\sqcup(W^j_i)^\omega$$
    with $(W^j_i)^\alpha=X_i$, and  $$(\Psi^j_i)^\alpha=\Psi^j_i|_{X_i}=\Phi_i.$$
    
    Then by Proposition \ref{lem:am_cpt_varying_metric}, after relabeling the $i$'s, $\Sigma_i$ converges in $C^\infty$ (with multiplicity one) to some $\Sigma\in \cG_{g}$. Hence, we can assume for any $i \in \mathbb{N}^+$, 
    \begin{equation}\label{eqn:Sigma_i_approx}
        \Sigma_i \in\bB^{\bF_{g}}_{\eta_2 / 2} (\cG_{g})\,.
    \end{equation}
    Furthermore, for each $i$ we can take an $j(i) \in \mathbb{N}^+$ such that, by discarding finitely many $g_i$:
    \begin{enumerate}[(i)]
        \item $(\Psi^{j(i)}_i)^\infty_{i = 1}$ is a minimizing sequence for $\tilde\cH(\Phi)$. 
        \item For all $j\geq j(i)$,
            \begin{equation}  \label{eqn:epsilon_2,i_a}
                \mathbf{M}_{g_i}((\Psi^{j}_i)^\alpha(x)) \geq  \mathbf{L}_{g_i}(\HC_i) - \varepsilon_{2,i} \implies |(\Psi^{j}_i)^\alpha(x)| \in \tilde{\cB}_{g}\,.
            \end{equation}
            and
            \begin{equation}  \label{eqn:epsilon_2,i_w}
                \mathbf{M}_{g_i}((\Psi^{j}_i)^\omega(x)) \geq  \mathbf{L}_{g_i}(\HC_i) - \varepsilon_{2,i} \implies |(\Psi^{j}_i)^\omega(x)| \in \tilde{\cG}_{g}\,.
            \end{equation}
    \end{enumerate}
    Note that the first item follows from $\delta_i\to 0$, while the second from \eqref{eqn:g_bar_e_0}, \eqref{eqn:e0_of_e2i}, \eqref{eqn:uniq_real} and \eqref{eqn:Sigma_i_approx}.
    For simplicity, we will denote 
    \[ 
        \Psi_i\coloneqq\Psi^{j(i)}_i, \quad W_i\coloneqq W^{j(i)}_i\,.
    \]
    
\medskip
\paragraph*{\textbf{Part 3. Pull-tight.}} 
    In this part, our goal is to pull tight, for each $i$, $\Psi_i$ to obtain a new sequence $\hat{\Psi}_i\in\tilde \HC_i$, such that for some $\eta_3 \in (0, \eta_0)$ the critical set
    \[
        \mathbf{C}_{g}((\hat{\Psi}_i)_{i }) \subset \bB^{\bF_{g}}_{\eta_3}(\mathcal{SV}^{L}_{g})\,, 
    \]
    while preserving conditions similar to \eqref{eqn:epsilon_2,i_a} and \eqref{eqn:epsilon_2,i_w}. 

    We set 
    \[
        \eta_3 \coloneqq 3\eta_2\,.
    \]
    Additionally, for any positive constant $A$, we denote the set
    \[
        \mathcal{V}^{\leq A}_g \coloneqq \{V \in \mathcal{V}_n (M) \mid \|V\|_{g}(M) \leq A\}\,.
    \]
    We also denote, by $\mathcal{S}^{\leq A}_{g}$, a subset of $\mathcal{V}^{\leq A}_g$ consisting of stationary varifolds with respect to the metric $g$, and additionally, denote
    \begin{equation}\label{eq:U}
        \mathcal{U}^{\leq A}_g \coloneqq \mathcal{V}^{\leq A}_g \setminus \bB^{\bF_{g}}_{\eta_3}(\mathcal{S}^{\leq A}_g)\,,
    \end{equation}
    which is a compact set. Similarly, for cycles, we define
    \[
        \cZ^{\leq A}_{n, g} \coloneqq \{T\in\cZ_n(M;\Z_2):\bM_{g}(T)\leq A\}\,.
    \]

    Following Pitt's pull-tight construction \cite{Pit81}, we have the following proposition which, intuitively, asserts that we pull tight $\Psi_i$ with respect to both $g$ and $g_i$ metrics. The proof is postponed to \S \ref{proof_prop_pull_tight}.
    
    \begin{prop} \label{prop_pull_tight}
        After possibly discarding finitely many elements in the sequence $(g_i)_i$ of metrics, there exists a continuous deformation map 
        \[
            H:[0,1]\x \cZ^{\leq A}_{n, g} \to \cZ^{\leq A}_{n, g}
        \]
        satisfying the following properties:
        \begin{enumerate}
            \item\label{prop_item_start_id} $H(0,\cdot)=\id$.
            \item\label{prop_item_fix_stationary} If $|T| \in \bB^{\Fb_{g}}_{\eta_3 / 2}(\mathcal{S}^{\leq A}_g)$ then $H(t,T)=T$ for each $t$.
            \item\label{prop_item_decrease_mass} For each $i$ and $(t,T)$, $\bM_{g}(H(t,T))\leq \bM_{g}(T)$ and $\bM_{g_i}(H(t,T)) \leq \bM_{g_i}(T)$. And either equality holds only if $H(t,T)=T$.
            \item\label{prop_item_mass_fixed_amount_below} There exists $\varepsilon_3 > 0$, such that for each $i$ and $T$, if $|H(1,T)|\notin \bB^{\bF_{g}}_{\eta_3}(\mathcal{S}^{\leq A}_g)$, then
            \[
                \bM_{g}(H(1,T)) \leq \bM_{g}(T) - \varepsilon_3\,,\quad \bM_{g_i}(H(1,T)) \leq \bM_{g_i}(T) - \varepsilon_3\,.
            \]
        \end{enumerate}
    \end{prop}
    
    We choose $A \coloneqq L + 1$ and, in accordance with the proposition above, discard finitely many values of $i$. Without loss of generality, we may assume for each $i$,
    \[
        \sup \bM\circ \Psi_{i} + \delta_i < A.
    \]

    Next, we define for each $i$ the space
    \[
        \hat W_i\coloneqq([0,1] \x X_i)\cup W_i
    \]
    where $\{1\}\x X_i \subset [0,1]\x X_i$ and $X_i \subset W_i$ are identified. Note that
    \[
        \partial \hat W_i=(\{0\}\x X_i) \cup [0, 1] \times Z \cup W_i^\omega.
    \]
    Naturally, we denote $\hat W_i^\alpha:=\{0\}\x X_i$, $\hat W_i^\omega:= [0, 1] \times Z \cup W^\omega_i$ (see Figure \ref{fig:hatWi}).
    We can apply Proposition \ref{prop_pull_tight} to obtain a deformation map $H$, and define
    \[
        \hat\Psi_i:\hat W_i\to \cZ_n(M;\bF_{g_i};\Z_2)
    \] by
    \begin{align*}
        \hat\Psi_i|_{[0,1]\x X_i}(t, x)&\coloneqq H(t, \Phi_{i}(x))\\
        \hat\Psi_i|_{W_i}(x)&\coloneqq H(1,\Psi_{i}(x)).
    \end{align*}

        \begin{figure}
            \centering
            \includegraphics[width=3.5in]{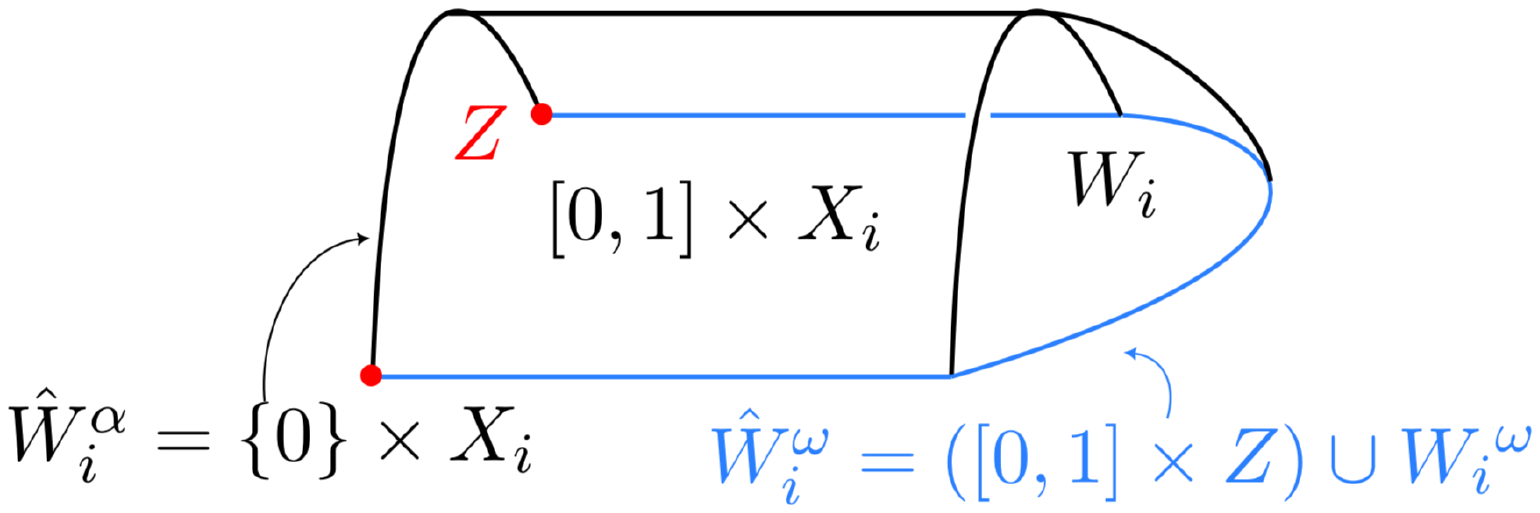}
            \caption{Construction of $\hat W_i=\dmn(\hat\Psi_i)$.}
            \label{fig:hatWi}
        \end{figure}

    Now, note that: 
    \begin{enumerate}[(i)]
        \item Since $(M, g_i)$ and $(M, g)$ are diffeomorphic for each $i$, $\bF_{g_i}$ and $\bF_{g}$ induces homeomorphic topologies. Thus, $\hat\Psi_i$ is still $\Fb_{g_i}$-continuous.
        \item By construction, we still have $\hat\Psi^\alpha_i=\Phi_{i}$.
        \item By Proposition \ref{prop_pull_tight} (\ref{prop_item_decrease_mass}), 
        \[
            \sup_{x \in \hat W_i} \bM_{g_i}\circ\hat \Psi_i (x) \leq \sup_{x \in W_i} \bM_{g_i}\circ \Psi_i(x) < \sup_{x \in X_i} \bM_{g_i}\circ\Phi_{i}(x) + \delta_i.
        \]
    \end{enumerate}
    Consequently, we can conclude that $\hat\Psi_i\in\tilde\HC_i$. Moreover, by (ii) and (iii), since $\delta_i\to 0$, $(\hat{\Psi}_i)_{i }$ is also a minimizing sequence.
    
    Additionally, we have the following lemma with proof postponed to \S \ref{sect_lem_critical_set_proof}.

    \begin{lem}\label{lem_critical_set}
        Under Notation \ref{nota:critical}, we have $\mathbf{C}_{g}((\hat{\Psi}_i)_{i }) \subset \bB^{\bF_{g}}_{\eta_3}(\mathcal{SV}^{L}_{g})\,.$
    \end{lem}
    
    Hence, there exists an $\bar\varepsilon_{3}>0$ such that for every sufficiently large $i$ and each $x\in \hat W_i$, if $|\hat{\Psi}_{i}(x)| \notin \bB^{\bF_{g}}_{\eta_3}(\mathcal{SV}^{L}_{g})$, then
    \begin{equation} \label{eqn:epsilon_bar_3}
        \mathbf{M}_{ g}(\hat{\Psi}_{i}(x)) < L - \bar\varepsilon_{3}\,, \quad \mathbf{M}_{g_i}(\hat{\Psi}_{i}(x)) < \mathbf{L}_{g_i}(\HC_i) - \bar \varepsilon_{3}\,.
    \end{equation}
    Additionally, let $\varepsilon_{3,i} \coloneqq \varepsilon_{2, i}$ and we observe that
    \begin{equation}  \label{eqn:epsilon_3,i_a}
        \mathbf{M}_{g_i}(\hat \Psi^\alpha_i(x)) \geq  \mathbf{L}_{g_i}(\HC_i) - \varepsilon_{3,i} \implies |\hat \Psi^\alpha_i(x)| \in \tilde{\cB}_{g}\,,
    \end{equation}
    and 
    \begin{equation}  \label{eqn:epsilon_3,i_w}
        \mathbf{M}_{g_i}(\hat \Psi^\omega_i(x)) \geq  \mathbf{L}_{g_i}(\HC_i) - \varepsilon_{3,i} \implies |\hat \Psi^\omega_i(x)| \in \tilde{\cG}_{g}\,.
    \end{equation}
    Indeed, \eqref{eqn:epsilon_3,i_a} trivially follows from \eqref{eqn:epsilon_2,i_a} and the fact that $\hat\Psi^\alpha_i=\Phi_{i}$. Regarding \eqref{eqn:epsilon_3,i_w}, we can justify it step by step:
    \begin{align*}
        & \mathbf{M}_{g_i}(\hat \Psi^\omega_i(x)) \geq \mathbf{L}_{g_i}(\HC_i) - \varepsilon_{3,i} \\
        \implies & \mathbf{M}_{g_i}(\Psi^\omega_i(x)) \geq  \mathbf{L}_{g_i}(\HC_i) - \varepsilon_{3,i}\quad &&(\text{by Proposition \ref{prop_pull_tight}} (\ref{prop_item_decrease_mass}))\\
        \implies & |\Psi^\omega_i(x)| \in \tilde{\cG}_{g} \subset \bB^{\bF_{g}}_{\eta_3 / 2}(\mathcal{SV}^{L}_{g})\quad &&(\text{by \eqref{eqn:epsilon_2,i_w}})\\
        \implies & \Psi^\omega_i(x) = \hat \Psi^\omega_i(x) \quad &&(\text{by Proposition \ref{prop_pull_tight} ((\ref{prop_item_fix_stationary}))})\\
        \implies & |\hat \Psi^\omega_i(x)| \in \tilde{\cG}_{g}\,.
    \end{align*}

    For simplicity, we discard finitely many $i$ such that \eqref{eqn:epsilon_bar_3}, \eqref{eqn:epsilon_3,i_a} and \eqref{eqn:epsilon_3,i_w} hold for every $i$.

\medskip
\paragraph*{\textbf{Part 4. $(\varepsilon, \delta)$-deformation.}}
    In this part, we aim to construct a sweepout $\hat{\hat{\Psi}}:\hat{\hat{W}}\to\Zc_n(M;\bF_{g_i};\Z_2)$ in $\tilde \HC_i$ such that for each $x\in \hat{\hat{W}}$,
    \begin{equation}\label{eq_hat_hat_Psi_prop}
        \mathbf{M}_{g_i}(\hat{\hat{\Psi}}(x)) \geq \mathbf{L}_{g_i}(\HC_i) - \varepsilon_{4} \implies |\hat{\hat{\Psi}}(x)| \in \hat\cG_{g} \sqcup \hat\cB_{g}\,,
    \end{equation}
    where $i$ is some positive integer, and $\varepsilon_4 > 0$.
    Furthermore, similar to \eqref{eqn:epsilon_3,i_a} and \eqref{eqn:epsilon_3,i_w}, it also satisfies
    \begin{equation}\label{eqn:epsilon_4_a}
        \mathbf{M}_{g_i}(\hat{\hat\Psi}^\alpha(x)) \geq  \mathbf{L}_{g_i}(\HC_i) - \varepsilon_{4}\implies|\hat {\hat \Psi}^\alpha(x)| \in \hat\cB_{g}\,,
    \end{equation}
    and
    \begin{equation}\label{eqn:epsilon_4_w}
        \mathbf{M}_{g_i}(\hat{\hat\Psi}^\omega(x)) \geq  \mathbf{L}_{g_i}(\HC_i) - \varepsilon_{4}\implies|\hat {\hat \Psi}^\omega(x)| \in \hat\cG_{g}\,.
    \end{equation}

    For each $i \in \N^+$, we choose an arbitrary $\varepsilon > 0$, which we will specify later, and apply Proposition \ref{prop:M_cts_approx_F_cts} with this $\varepsilon$ to $\hat \Psi_i$ to obtain a $\Mb_{g_i}$-continuous map
    \[
        \hat \Psi'_i: \hat W_i \to \Zc_n(M; \bM_{g_i}; \Z_2)\,.
    \]
    $\hat \Psi'_i(x)$ satisfies the conditions:
    \begin{enumerate}[(i)]
        \item There is an $\Fb_{g_i}$-continuous homotopy 
        \[
            \hat H_i: [0, 1] \times \hat W_i \to \Zc_n(M; \bF_{g_i}; \Z_2)
        \]
        with $\hat H_i(0, \cdot) = \hat \Psi_i$ and $\hat H_i(1, \cdot) = \hat \Psi'_i$;
        \item $\sup_{(t, x)} \Fb_{g_i}(\hat H_i(t, x), \hat \Psi_i(x)) < \varepsilon$;
        \item $\sup_{(t, x)} \Mb_{g_i}(\hat H_i(t, x)) < \sup_x \Mb_{g_i}(\hat \Psi_i(x)) + \varepsilon$.
    \end{enumerate}
    We can take $\varepsilon$ very small such that
    \begin{equation}\label{eqn:hat_H_i_upper_bound}
        \sup_{(t, x)} \Mb_{g_i}(\hat H_i(t, x)) < L_i + \delta_i\,,
    \end{equation}
    and for $\eta_4 \coloneqq 2\cdot\eta_3 \in (0, \min(\eta_0, \varepsilon_0))$, $\bar \varepsilon_4 \coloneqq \bar \varepsilon_3 / 2$, and $\varepsilon_{4, i} \coloneqq \varepsilon_{3, i}/2$,
    \begin{align}
        \label{eqn:epsilon_bar_4} \forall x \in \hat W_i,\ \mathbf{M}_{g_i}(\hat H_i(t, x)) \geq  \mathbf{L}_{g_i}(\HC_i) - \bar\varepsilon_4 &\implies|\hat H_i(t, x)| \in \bB^{\bF_{g}}_{\eta_4} (\mathcal{SV}^{L}_{g})\,,\\
        \label{eqn:epsilon_4_i_a} \forall x \in \hat W^\alpha_i,\ \mathbf{M}_{g_i}(\hat H_i(t, x)) \geq  \mathbf{L}_{g_i}(\HC_i) - \varepsilon_{4,i}&\implies|\hat H_i(t, x)| \in \tilde{\cB}_{g}\,,\\
        \label{eqn:epsilon_4_i_w} \forall x \in \hat W^\omega_i,\ \mathbf{M}_{g_i}(\hat H_i(t, x)) \geq  \mathbf{L}_{g_i}(\HC_i) - \varepsilon_{4,i}&\implies|\hat H_i(t, x)| \in \tilde{\cG}_{g}\,,
    \end{align}
    for any $t \in [0, 1]$. These conditions are a direct result of (ii) and the previously established \eqref{eqn:epsilon_bar_3}, \eqref{eqn:epsilon_3,i_a}, and \eqref{eqn:epsilon_3,i_w}.

    Given that $\lim_i (L_i + \delta_i) = \lim_i \mathbf{L}_{g_i}(\HC_i)$, for sufficiently large $i$, we have
    \[
        \mathbf{L}_{g_i}(\HC_i) - \min(\varepsilon_0, d_0, \bar \varepsilon_4) / 100 > L_i + \delta_i - \min(\varepsilon_0, d_0, \bar \varepsilon_4)/10\,.
    \] We can fix such an $i$ for our subsequent construction.

    Now, in $(M, g_i)$, we apply Lemma \ref{lem:(e, d)-deformation} to $\hat \Psi'_i$ with
    \begin{align*}
        R &= d_0\,,\\
        \bar\varepsilon &= \min(\varepsilon_0, d_0, \bar \varepsilon_4)\,,\\
        \eta &= r_0\,,\\
        s &= s_0\,,\\
        \mathcal{W} &= \cG_{g} \cup \cB_{g}\,.
    \end{align*}
    To verify that $\hat \Psi'_i$ satisfies the assumptions of Lemma \ref{lem:(e, d)-deformation}, let $x \in \hat W_i$ satisfy 
    \[
        \bM_{g_i}(\hat \Psi'_i(x)) \geq L - \bar\varepsilon\,,\quad \Fb_{g_i}(|\hat \Psi'_i(x)|, \mathcal{W}) \geq R\,.
    \]
    Since $L = \sup_{x \in \hat X_i} \bM_g(\hat \Psi'_i(x)) \geq \mathbf{L}_{g_i}(\HC_i)$ and $\bar \varepsilon \leq \bar \varepsilon_4$, we have
    \[
        \bM_{g_i}(\hat \Psi'_i(x)) \geq \mathbf{L}_{g_i}(\HC_i) - \bar\varepsilon_4\,,
    \]
    and thus, by \eqref{eqn:epsilon_bar_4},
    \[
        |\hat \Psi'_i(x)| \in \bB^{\bF_{g}}_{\eta_4} (\mathcal{SV}^{L}_{g}) \subset \bB^{\bF_{g}}_{\eta_0} (\mathcal{SV}^{L}_{g})\,.
    \]
    By Lemma \ref{lem:not_am_varying_metric}, given that 
    \[
        R = d_0 < \eta_{\ref{lem:far_of_diff_geom}}(M, g, 2k + 3, L, \cG_{g}, \cB_{g})
    \] and 
    \[
        \|g_i - g\|_{C^\infty, g} < \eta_0 \leq \eta_{\ref{lem:not_am_varying_metric}}(M, g, 2k + 3, r_0, d_0, L)\,,
    \] $|\hat \Psi'_i(x)|$ satisfies the annular $(\bar \varepsilon, \delta)$-deformation conditions as required by Lemma \ref{lem:(e, d)-deformation}.
    
    Consequently, for an arbitrary $\bar \delta > 0$, which will be specified later, we obtain a continuous map
    \[
        \hat \Psi^*_i: \hat{W}_i \to \Zc_n(M; \bM_g; \Z_2)\,,
    \]
    and a homotopy map
    \[
        \hat H^\textsc{DEF}_i: [0, 1] \times \hat{W}_i \to \Zc_n(M; \bM_g; \Z_2)
    \]
    with $\hat H^\textsc{DEF}_i(0, \cdot) = \hat\Psi'_i$ and $\hat H^\textsc{DEF}_i(1, \cdot) = \hat\Psi^*_i$. 
    
    We can indeed choose $\bar \delta$ to satisfy the following conditions:
    \begin{enumerate}[(i)]
        \item $\bar \delta < \min(d_0, \varepsilon_{4, i}/10)$;
        \item \begin{equation}\label{eqn:H_DEF_upper_bound}
            \sup_{t, x} \Mb_{g_i} \circ \hat H^\textsc{DEF}_i(t, x) < L_i + \delta_i\,,
        \end{equation}
        \item For any $T, S \in \Zc_n(M; \Z_2)$ with $\Mb_{g_i}(T), \Mb_{g_i}(S) \leq L_i + \delta_i$,
        \[
            \Mb_{g_i}(T, S) < \bar\delta  \implies \Fb_{g_i}(T, S) < d_0\,.
        \]
    \end{enumerate}

    Now, we need the following three claims, whose proof will be postponed to the \S \ref{sect:technical}.
    
    \begin{claim}\label{claim1}
        For any $x \in \hat W_i$, if $\Mb_{g_i}(\hat \Psi^*_i(x)) \geq \mathbf{L}_{g_i}(\HC_i) - \bar \varepsilon / 100$, then $|\hat \Psi^*_i(x)| \in \hat \cG_{g} \sqcup \hat \cB_{g}$
    \end{claim}

    \begin{claim}\label{claim2}
        For any $x \in \hat W^\alpha_i$ and any $t \in [0, 1]$, if $\Mb_{g_i}(\hat H^\textsc{DEF}_i(t, x)) \geq \mathbf{L}_{g_i}(\HC_i) - \varepsilon_{4, i} / 4$, then $|\hat H^\textsc{DEF}_i(t, x)| \in \hat \cB_{g}$.
    \end{claim}

    \begin{claim}\label{claim3}
        For any $x \in \hat W^\omega_i$ and any $t \in [0, 1]$, if $\Mb_{g_i}(\hat \Psi^*_i(x)) \geq \mathbf{L}_{g_i}(\HC_i) - \varepsilon_{4, i} / 4$, then $|\hat \Psi^*_i(x))| \in \hat \cG_{g}$.
    \end{claim}

    We define the space 
    \[
        \hat{\hat W} \coloneqq ([0, 2] \times X_i) \cup \hat W_i
    \]
    where $\{2\} \times X_i \subset [0, 2] \times X_i$ and $\hat W^\alpha_i \subset \hat W_i$ are identified. On $\hat{\hat W}$, we define the map
    \[
        \hat{\hat \Psi}: \hat{\hat W} \to \Zc_n(M; \Fb_{g_i}; \Z_2)
    \]
    as follows:
    \begin{align*}
        \hat{\hat \Psi}\vert_{[0, 1] \times X_i}(t, x) &\coloneqq \hat H_i\vert_{[0, 1] \times X_i}(t, x)\\
        \hat{\hat \Psi}\vert_{[1, 2]\x X_i}(t, x) &\coloneqq \hat H^\textsc{DEF}_i\vert_{[0, 1] \times  X_i}(t - 1, x)\\
        \hat{\hat \Psi}\vert_{\hat W_i} &\coloneqq \Psi^*_i\,.
    \end{align*}

    Since $\hat {\hat \Psi}^\alpha = \hat H_i\vert_{\{0\} \times X_i} = \hat \Psi^\alpha_i = \Phi_{i}$ and $\sup_{x} \Mb_{g_i}\circ \hat{\hat{\Psi}}(x) < L_i + \delta_i$ by \eqref{eqn:hat_H_i_upper_bound} and \eqref{eqn:H_DEF_upper_bound}, we can conclude that $\hat {\hat \Psi} \in \tilde H^i$.
    
    Moreover, for $\varepsilon_4 \coloneqq \min(\varepsilon_{4, i}/4, \bar\varepsilon / 100)$, the previous three claims imply \eqref{eq_hat_hat_Psi_prop}, \eqref{eqn:epsilon_4_a} and \eqref{eqn:epsilon_4_w}.

\medskip
\paragraph*{\textbf{Part 5. Constructing a map $\Xi^\omega$ homologous to $(\hat{\hat{\Psi}})^\omega$}}
    Finally, for the $i$ we fixed in the last part, we will now construct some map $\Xi\in\tilde\HC_i$ such that
    \[
        \sup \bM_{g_i}\circ\Xi^\omega<\bL(\HC_i),
    \]
    thereby arriving at a contradiction.

    First, we define two subsets
    \begin{align*}
        \hat{\hat{\cG}}_{g} &\coloneqq \{V \in \hat{\cG}_{g} : \|V\|_{g_i}(M) \geq \mathbf{L}_{g_i}(\HC_i) - \varepsilon_4\}\,,\\
        \quad \hat{\hat{\cB}}_{g} &\coloneqq \{V \in \hat{\cB}_{g} : \|V\|_{g_i}(M) \geq \mathbf{L}_{g_i}(\HC_i) - \varepsilon_4\}\,.
    \end{align*}
    Due to the fact that $\Fb_{g}(\hat{\cG}_{g}, \hat{\cB}_{g}) > \eta_0$,   \eqref{eq_hat_hat_Psi_prop}, \eqref{eqn:epsilon_4_a}, and $\eqref{eqn:epsilon_4_w}$, we can proceed by subdividing the simplicial complex structure of $\hat{\hat W}$ to ensure that the subcomplex
    \[
        A\coloneqq\bigcup\{(2k + 3)\textrm{-cells }\alpha\subset \hat{\hat W}:|\hat{\hat \Psi}|_\alpha|\textrm{ intersects } \hat{\hat{\cG}}_{g}\}
    \]
    of $\hat{\hat W}$ satisfies (see Figure \ref{fig:part5}): 
    \begin{enumerate}[(i)]
        \item $A$ is disjoint from $\hat{\hat W}^\alpha=\{0\}\x X_i$.
        \item $|\hat{\hat \Psi}|_A|$ is disjoint from  $\hat{\hat \cB}_{g}$.
        \item For every $x \in \partial A\backslash \hat{\hat W}^\omega$, 
        \begin{equation}\label{eqn:partial_A_bound}
            \Mb_{g_i} \circ \hat{\hat \Psi}(x) < \mathbf{L}_{g_i}(\HC_i) - \varepsilon_4\,,
        \end{equation}
        and thus,
        \[
            |\hat{\hat \Psi}|_{\partial A\backslash \hat{\hat W}^\omega}|\;\textrm{ is disjoint from }\; \hat{\hat\cG}_{g} \cup \hat{\hat\cB}_{g}.
        \]
    \end{enumerate}
    \begin{figure}[h]
        \centering
        \makebox[\textwidth][c]{\includegraphics[width=5.5in]{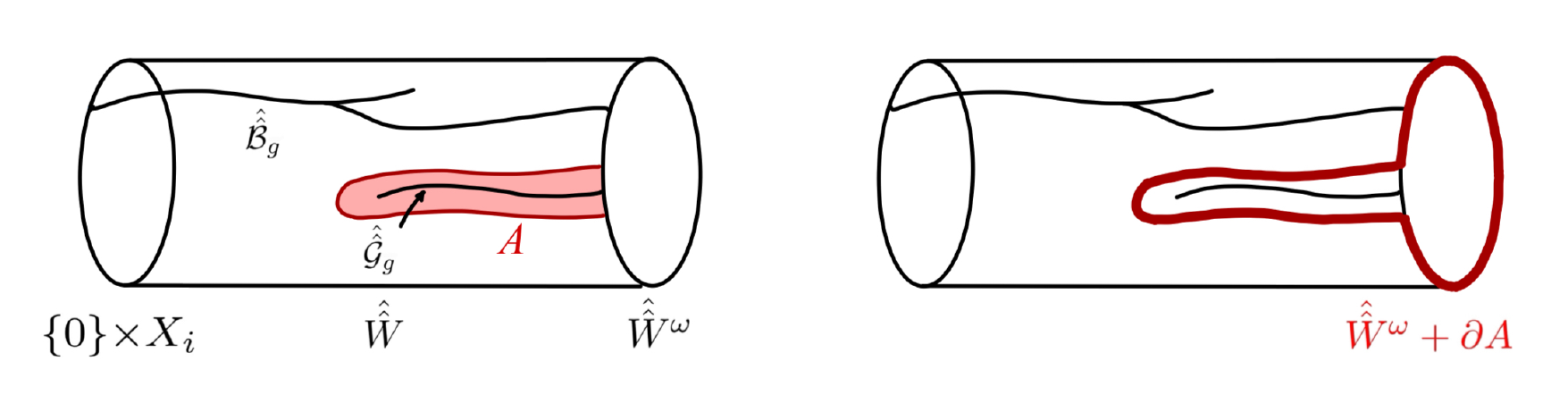}}
        \caption{The final parameter space $\hat{\hat W}^\omega+\partial A.$}
        \label{fig:part5}
    \end{figure}

    Next, we define the map
    \[
        \Xi\coloneqq\hat{\hat \Psi}|_{\overline{\hat{\hat W}\backslash A}}\;.
    \]
    
    We will now show that $\Xi$ leads to a contradiction. As $A$ is disjoint from $\hat{\hat W}^\alpha \subset \hat{\hat W}$, we know that the boundary of $\overline{\hat{\hat W}\backslash A}$ is a {\it disjoint} union of $\hat{\hat W}^\alpha$ and $\hat{\hat{W}}^\omega+\partial A$: Note that we used ``$+$" in the sense of adding simplicial subcomplexes with $\Z_2$-coefficients, such that the part $\hat{\hat W}^\omega\cap \partial A$ cancels out. Consequently, we can conclude that $\Xi\in \tilde\HC_i$.
    
    Now, let us consider $\Xi^\omega\in\HC_i$. Its domain $\hat{\hat W}^\omega+ \partial A$ is a union of $\partial A\backslash \hat{\hat W}^\omega$ and $\hat{\hat W}^\omega\backslash A$ (see Figure (\ref{fig:part5})). By \eqref{eqn:partial_A_bound},
    \[
        \sup_{\partial A\backslash \hat{\hat W}^\omega} \bM_{g_i} \circ \Xi < \mathbf{L}_{g_i}(\HC_i) - \varepsilon_4\,,
    \]
    and by \eqref{eqn:epsilon_4_w},
    \[
        \sup_{\hat{\hat W}^\omega \backslash  A} \bM_{g_i} \circ \Xi < \mathbf{L}_{g_i}(\HC_i) - \varepsilon_4\,.
    \]
    In summary, we arrive at the conclusion
    \begin{equation*}
        \bM_{g_i}\circ \Xi^\omega< \bL_{g_i}(\HC_i)-\varepsilon_4,
    \end{equation*}
    leading to a contradiction. This completes the proof of Theorem \ref{thm:strong_multi_one_I}.
    
\section{Proof of Theorem \ref{thm_13_width}}\label{sect:proof_main_thm}

    Theorem \ref{thm_13_width} is a consequence of Theorem \ref{thm:strong_multi_one_I}. 
    
    Let $(S^3, \bar g)$ be the unit $3$-sphere. Recall that C. Nurser \cite{Nur16} showed the $13$-width 
    \[
        \omega_{13}(S^3, \bar g) \leq 8\pi\,.
    \]
    Suppose, for the sake of contradiction, that the $13$-width is precisely $8\pi$ on $(S^3, \bar g)$. In fact, C. Nurser \cite[\S 3.6]{Nur16} has constructed an $\cF$-continuous map $\Phi_0: \mathbb{RP}^{13} \to \mathcal{Z}_2(S^3; \mathbb{Z}_2)$ with no concentration of mass, such that
    \[
        \sup_{ x \in \mathbb{RP}^{13}} \mathbf{M}_{\bar g}(\Phi_0(x)) = 8\pi\,.
    \]
    Namely, let us consider the following 14 polynomials for $(x_1,x_2,x_3,x_4)\in\R^4$:
    \begin{align}\label{eq_13poly}
        &1,x_1,x_2,x_3,x_4,x_1x_2,x_1x_3,x_1x_4,x_2x_3,x_2x_4,x_3x_4,
        x_1^2-x_2^2,x_1^2-x_3^2,x_1^2-x_4^2
    \end{align}
    and denote them by $p_0,p_1, \cdots ,p_{13}$ respectively. 
    Then $\Phi_{0}:\RP^{13}\to\Zc_2(S^3;\Z_2)$ is defined by
    \[
        \Phi_{0}([a_0:a_1:\cdots:a_{13}])=\partial^* \left\{x\in S^3:\sum^{13}_{i=0}a_ip_i(x) < 0\right\}\,,
    \]
    where $\partial^*$ is the reduced boundary operator.

    For each $i \in \mathbb{N}^+$, we set $\Phi_i \equiv \Phi_0$. By Corollary \ref{cor:relation_to_p_width}, $(\Phi_i)_i$ is a minimizing sequence in $\mathcal{H}(\Phi_0)$. By Theorem \ref{thm:strong_multi_one_I}, to derive a contradiction, it suffices to verify that every embedded minimal cycle in $\bC_{\bar g}((\Phi_i)_i)$ is a multiplicity two equator. This is equivalent to the following lemma.

    \begin{lem}\label{lem:noMultiplicityOne}
        No embedded minimal cycle in $\bC_{\bar g}((\Phi_i)_i)$ is of multiplicity one.
    \end{lem}
    Indeed, if Lemma~\ref{lem:noMultiplicityOne} holds, since the equator is the unique minimal surface of the least area $4\pi$ in $(S^3, \bar g)$, by the Frankel property, every embedded minimal cycle in $\bC_{\bar g}((\Phi_i)_i)$ is of multiplicity two and thus, a multiplicity two equator. So below let us prove Lemma~\ref{lem:noMultiplicityOne}.
    
        Suppose, for the sake of contradiction, that there exists $V \in \bC_{\bar g}((\Phi_i)_i)$ and an embedded minimal surface $\Sigma \subset (S^3, \bar g)$ such that
        \[
            V = |\Sigma|\,.
        \]
        In particular, $\area_{\bar g}(\Sigma) = 8\pi$, and there exists a sequence $(\underline{a}^{(j)} \equiv [a^{(j)}_0 : a^{(j)}_1 : \cdots : a^{(j)}_{13}])_j$ in $\mathbb{RP}^{13}$ such that in the varifold sense, as $j \to \infty$,
$$|\Phi_0(\underline{a}^{(j)})| \to V\,.$$
        Without loss of generality, by taking a subsequence, we can assume that $\lim_{i \to \infty} \underline{a}^{(j)} = \underline{a}^{(\infty)} = [a^{(\infty)}_0 : a^{(\infty)}_1 : \cdots : a^{(\infty)}_{13}] \in \mathbb{RP}^{13}$.

Let us state a lemma, which will be proven shortly.
\begin{lem}\label{lem:zeroSetNotSmooth8pi}
        For any choice of $a_0,a_1,...,a_{13}\in \R$, the zero set of the polynomial $\sum^{13}_{i=0}a_ip_i(x)$ restricted onto $S^3$ cannot be the union of a closed smooth surface and an $\mathcal{H}^2$-measure zero set, which has area $8\pi$.
    \end{lem}
    Note that here we do not treat the zero set as varifold, and so it cannot have multiplicity.

        If we assume Lemma~\ref{lem:zeroSetNotSmooth8pi}, to obtain a contradiction, it suffices to show that 
        \begin{equation}\label{eqn:V_is_zero_set}
            V = \left|\left\{x\in S^3:\sum^{13}_{i=0}a^{(\infty)}_ip_i(x) = 0\right\}\right| \,.
        \end{equation}

        Note that for every $[a_0: a_1: \cdots: a_{13}] \in \mathbb{RP}^{13}$, as sets,
        \[\begin{aligned}
            \spt(\Phi_0(\underline{a})) \subset \partial \left\{x\in S^3:\sum^{13}_{i=0}a_ip_i(x) < 0\right\} \subset \left\{x\in S^3:\sum^{13}_{i=0}a_ip_i(x) = 0\right\}\,,
        \end{aligned}\]
        where $\partial$ is the topological boundary operator. Furthermore, for every $\varepsilon > 0$, let $N_\varepsilon(\cdot)$ be the tubular $\varepsilon$-neighborhood map, and we have
        \[
            \left\{x\in S^3:\sum^{13}_{i=0}a^{(j)}_ip_i(x) = 0\right\} \subset N_\varepsilon\left(\left\{x\in S^3:\sum^{13}_{i=0}a^{(\infty)}_ip_i(x) = 0\right\}\right)\,.
        \]
        Therefore, we have
        \[
            \spt V \subset \left\{x\in S^3:\sum^{13}_{i=0}a^{(\infty)}_ip_i(x) = 0\right\}\,.
        \]
        Then \eqref{eqn:V_is_zero_set} follows from $\|V\|(S^3, \bar g) = 8\pi$ and
        \[
            \mathcal{H}^2\left(\left\{x\in S^3:\sum^{13}_{i=0}a^{(\infty)}_ip_i(x) = 0\right\}\right) \leq 8\pi\,.
        \]
        proved in \S~3.6 in~\cite{Nur16}.

        Now, it remains to prove Lemma~\ref{lem:zeroSetNotSmooth8pi}.

    \begin{proof}[Proof of Lemma~\ref{lem:zeroSetNotSmooth8pi}]
        First, let us recall a Crofton-type formula. Let  $d\omega_n$ be the volume form on the unit $n$-sphere $S^n$. Let $C_1$ and $C_2$ be submanifolds in $S^n$ with finite volume that have  dimension $p$ and $q$ respectively. Noting that each orthonormal $(n+1)$-frame in $\R^{n+1}$ can defined by first choosing an element in $S^n$, then another in $S^{n-1}$, and so on, and finally in $S^1$, we can equip on $SO(n+1)$ the volume form $dg$ given by $d\omega_n d\omega_{n-1}...d\omega_1$. Then a formula of Satanl\'o \cite[\S 3]{San50} states
        \begin{align*}
            &\int_{SO(n+1)}\cH^{p+q-n}(C_1\cap g\cdot C_2)dg\\
            &=\frac{\vol(S^{p+q-n})}{\vol(S^p)\vol(S^q)}\vol(S^1)\vol(S^2)...\vol(S^n)\cH^p(C_1)\cH^q(C_2).
        \end{align*}
        Note that $\cH^n$ stands for the $n$-dimensional Hausdorff measure.
        
        In particular,  when $n=2$, $p=1$, and $C_2$ is a great circle,
        \begin{equation}\label{eq_crofton_nis2}
        \int_{SO(3)}\cH^{0}(C_1\cap g\cdot C_2)d\omega_2d\omega_1=\frac{1}{\pi}\vol(S^1)\vol(S^2)\cH^1(C_1),    
        \end{equation}
        and  when $n=3,p=2,$ and $C_2$ is a great 2-sphere,
        \begin{equation}\label{eq_crofton_nis3}
        \int_{SO(4)}\cH^{1}(C_1\cap g\cdot C_2)d\omega_3d\omega_2d\omega_1=\frac{1}{2}\vol(S^1)\vol(S^2)\vol(S^3)\cH^2(C_1). 
        \end{equation}
        
        \begin{lem}\label{lem_crofton_nis2}
            Let $\Gamma$ be the zero set of a quadratic polynomial, for $(x_1,x_2,x_3)\in\R^3$, restricted onto the unit $2$-sphere $S^2$. Assume $\Gamma$ is not the entire $S^2$. Then $\mathcal H^1(\Gamma)\leq 4\pi$, with equality holds if and only if $\Gamma$ is the union of two distinct great circle.
        \end{lem}
        \begin{proof}
            To show $\mathcal H^1(\Gamma)\leq 4\pi$, we apply (\ref{eq_crofton_nis2}) with $C_1=\Gamma$. Since $\Gamma$ intersects a generic great circle at at most 4 points by B\'ezout's theorem, the left hand side of (\ref{eq_crofton_nis2}) is at most $4\vol(S^1)\vol(S^2)$, and so $\mathcal H^1(\Gamma)\leq 4\pi$.
        
            As for the equality case, by (\ref{eq_crofton_nis2}) we know $\Gamma$ intersects a generic great circle at  exactly 4 points. Suppose by contradiction that $\Gamma$ does not consist of two distinct great circles. Since $\cH^1(\Gamma)=4\pi$, $\Gamma$ cannot just be one great circle, and so there exists $x\in \Gamma$ at which $\Gamma$ has non-zero geodesic curvature. Let $\gamma$ be some  great circle that touches $\Gamma$ at $x$ (on one side, in fact). Then, by considering the degree of intersection, there are three cases: (1) $\gamma\cap \Gamma=\{x\}$; (2) $\gamma\cap \Gamma=\{x,y\}$ with $y$ being a degree two touching point; and (3) $\gamma\cap \Gamma=\{x,y,z\}$ with $y,z$ being transversely intersecting points. But in each  case, since $\Gamma$ has non-zero curvature at $x$, we can perturb $\gamma$ to a generic great circle that does not intersect  $\Gamma$ at 4 points. Contradiction arises.
        \end{proof}
        
        Now, to prove Lemma \ref{lem:zeroSetNotSmooth8pi}, we fix a polynomial $\sum^{13}_{i=0}a_ip_i(x)$ and denote by $\Sigma$ the zero set of it restricted onto $S^3$. It suffices to assume that $\Sigma$ has area $8\pi$  and prove that it cannot be a smooth surface except for a $\mathcal{H}^2$ measure zero set.
        
        Let us recall how to prove $\mathcal H^2(\Sigma)\leq 8\pi$ (proven by C. Nurser in \cite[\S 3.6]{Nur16}). Let $C_2$ be a fixed great 2-sphere in $S^3$. Applying Lemma \ref{lem_crofton_nis2} we have   $\cH^{1}(\Sigma\cap g\cdot C_2)\leq 4\pi$: Note that since $\Sigma$ is not the whole $S^3$ itself, we know that for each generic $g\cdot C_2$ the set $\Sigma\cap g\cdot C_2$ is not the great 2-sphere $g\cdot C_2$. Then by (\ref{eq_crofton_nis3}), $\cH^2(\Sigma)\leq 8\pi$.
        
            Suppose the equality $\cH^2(\Sigma)= 8\pi$ holds. Then by (\ref{eq_crofton_nis3}), we know for each generic great 2-sphere $g\cdot C_2$, $\cH^{1}(\Sigma\cap g\cdot C_2)= 4\pi$, which  by Lemma \ref{lem_crofton_nis2} means $\Sigma\cap g\cdot C_2$ is a union of two distinct great circles. On the other hand, if $\Sigma$ were the union of a closed smooth surface with an $\mathcal{H}^2$-measure zero set, then a generic great 2-sphere should intersect the smooth part of $\Sigma$ transversely and the intersection would be smooth closed curves. Hence, it can be checked that $\Sigma$ cannot be the union of a closed smooth surface with an $\mathcal{H}^2$-measure zero set.
        This finishes the proof of Lemma \ref{lem:zeroSetNotSmooth8pi}.
    \end{proof}
This finishes the proof of Theorem \ref{thm_13_width}.
\section{Proof of Theorem \ref{thm:strong_multi_one_II}} \label{sect:proof_main_thm_II}

    The proof is similar to that of Theorem \ref{thm:strong_multi_one_I}. The only difference is that we need to construct a new minimizing sequence instead of arguing by contradiction.
    
    As in \S \ref{sect:proof_main_thm_I}, let $(M^{n+1}, g)$ ($3 \leq n + 1 \leq 7$) be a closed Riemannian manifold equipped with a bumpy metric or a metric of positive Ricci curvature. Fix an $\bF_g$-continuous map $\Phi:X \to\Zc_n(M; \Fb_g; \Z_2)$, where $X$ is a pure finite simplicial $k$-cycle with empty boundary,   such that  $L \coloneqq \bL_g(\cH(\Phi))>0$.

    We also define $r_0, \mathcal{SV}^{L}_{g}$, $\cM^{L}_{g}$,  $\cG_{g}$ and $\cB_{g}$ in the same manner.

    Since $(M, g)$ has positive Ricci curvature or a bumpy metric, by Lemma \ref{lem:am_cpt_varying_metric}, both $\cG_{g}$ and $\cB_{g}$ are compact in the varifold topology. We define
    \begin{align*}
        d_0 &\coloneqq \min(\eta_{\ref{lem:cycle_of_small_rep}}(M, g, 2k + 3, \cG_{g}), \eta_{\ref{lem:far_of_diff_geom}}(M, g, 2k + 3, L, \cG_{g}, \cB_{g})) / 10\,,\\
        \varepsilon_0 &\coloneqq \bar\varepsilon_{\ref{lem:not_am_varying_metric}}(M, g, 2k + 3, r_0, d_0, L)\,,\\
        s_0 &\coloneqq \bar{s}_{\ref{lem:not_am_varying_metric}}(M, g, 2k + 3, r_0, d_0, L)\,,\\
        \eta_0 &\coloneqq \min(d_0, \varepsilon_0, \eta_{\ref{lem:not_am_varying_metric}}(M, g, 2k + 3, r_0, d_0, L)) / 10\,,\\
        \tilde{\cG}_{g} &\coloneqq \cG'_{\ref{lem:far_of_diff_geom}}(M, g, 2k + 3, L, \cG_{g}, \cB_{g})\,,\\
        \tilde{\cB}_{g} &\coloneqq \cB'_{\ref{lem:far_of_diff_geom}}(M, g, 2k + 3, L, \cG_{g}, \cB_{g})\,,\\
        \hat \cG_{g} &\coloneqq \cG''_{\ref{lem:far_of_diff_geom}}(M, g, 2k + 3, L, \cG_{g}, \cB_{g})\,,\\
        \hat \cB_{g} &\coloneqq \cB''_{\ref{lem:far_of_diff_geom}}(M, g, 2k + 3, L, \cG_{g}, \cB_{g})\,,
    \end{align*}
    from Lemma \ref{lem:cycle_of_small_rep}, Lemma \ref{lem:far_of_diff_geom} and Lemma \ref{lem:not_am_varying_metric}. Note that here $d_0, \tilde{\cG}_{g}$, $\tilde{\cB}_{g}$, $\hat \cG_{g}$ and $\hat \cB_{g}$ are different from those defined in \S \ref{sect:proof_main_thm_I}.

    In fact, we can refine the critical set as follows.

    \begin{prop}
        There exists a pulled-tight minimizing sequence of $\mathbf{F}_{g}$-continuous maps $(\Phi_{i}: X_i \to \mathcal{Z}_n(M; \mathbf{F}_{g}; \mathbb{Z}_2))^\infty_{i = 1}$, such that
        \[
            \mathbf{C}_{g}((\Phi_i)^\infty_{i = 1}) \subset (\tilde{\cG}_{g} \cup \tilde{\cB}_{g}) \cap \mathcal{SV}^{L}_{g}\,.
        \]
        
        In particular, there exists an positive constant $\bar{\varepsilon}_0 > 0$, such that
        \[
            \Mb_{g}\circ \Phi_i(x) \geq L - \bar{\varepsilon}_0 \implies |\Phi_i(x)| \in (\tilde{\cG}_{g} \cup \tilde{\cB}_{g}) \cap \bB^{\bF_{g}}_{\eta_0} (\mathcal{SV}^{L}_{g})\,.
        \]
    \end{prop}
    \begin{proof}
        By Marques-Neves \cite{MN21}, there exists a pulled-tight minimizing sequence of $\mathbf{F}_{g}$-continuous  maps $(\Phi_{i}: X_i \to \mathcal{Z}_n(M; \mathbf{F}_{g}; \mathbb{Z}_2))^\infty_{i = 1}$ such that
        \[
            \mathbf{C}_{g}((\Phi_i)^\infty_{i = 1}) \subset \mathcal{SV}^{L}_{g}\,.
        \]
        Therefore, by discarding finitely many $i$'s and after relabelling, there exists $\varepsilon'_0 \in (0, \varepsilon_0/2)$ such that for any $i \in \N^+$ and any $x \in X_i$,
        \[
            \Mb_{g}\circ \Phi_i(x) \geq L - \varepsilon'_0 \implies |\Phi_i(x)| \in \bB^{\Fb_{g}}_{\eta_0}(\mathcal{SV}^{L}_{g})\,.
        \]
        In addition, by discarding extra finitely many $i$'s, we may assume for every $i$,
        \[
            L_i \coloneqq \sup_{x \in X_i} \Mb_{g}\circ \Phi_i(x) < L + \varepsilon'_0/100\,.
        \]

        To apply Lemma \ref{lem:(e, d)-deformation} to each $\Phi_i$ with
        \begin{align*}
            R &= d_0\,,\\
            \bar\varepsilon &= \varepsilon'_0\,,\\
            \eta &= r_0\,,\\
            s &= s_0\,,\\
            \mathcal{W} &= \cG_{g} \cup \cB_{g}\,,
        \end{align*}
        it suffices to verify that $\Phi_i$ satisfies the assumption of Lemma \ref{lem:(e, d)-deformation}. Let $x \in X_i$ satisfy 
        \[
            \bM_{g}(\Phi_i(x)) \geq L_i - \bar\varepsilon\,,\quad \Fb_{g}(|\Phi_i(x)|, \mathcal{W}) \geq R\,.
        \] 
        and then we have
        \[
            \Phi_i(x) \in \bB^{\Fb_{g}}_{\eta_0}(\mathcal{SV}^{L}_{g}) \setminus \bB^{\Fb_{g}}_{d_0}(\mathcal{W})\,.
        \]
        By Lemma \ref{lem:not_am_varying_metric}, $|\Phi_i(x)|$ satisfies the annular $(\bar \varepsilon, \delta)$-deformation conditions as required by Lemma \ref{lem:(e, d)-deformation}.

        Consequently, for any $\bar \delta_i > 0$, we obtain a new sweepout $\Phi^*_i$ homotopic to $\Phi_i$ in the $\bF_{g}$ topology such that
        \begin{enumerate}[(i)]
            \item $\Mb_{g}(\Phi^*_i(x)) < \Mb_{g}(\Phi_i(x)) + \bar \delta_i$;
            \item There exists $T_{1, x} \in \Zc_n(M; \Z_2)$ and $\hat x \in X_i$ such that
                \[
                    \Mb_{g}(\Phi^*_i(x), T_{1, x}) < \bar \delta_i\,,
                \]
                and $T_{1, x} = \Phi_i(\hat x)$ on $M \setminus \bigcup^m_{i = 1} \overline{B}_{g}(p_i, \eta)$ for some collection $\{p_i\}^m_{i = 1} \subset M$, $m \leq 3^{2k+1}$;
            \item If $\Mb_{g}\circ \Phi^*_i(x) \geq L - \varepsilon'_0 / 100 > L_i -  \varepsilon'_0 / 10$, 
            \[
                \Fb_{g}(|\Phi_i(\hat x)|, \mathcal{W}) \leq 2R\,.
            \]
        \end{enumerate}
        
        Then, let $\bar \delta_i < \min(\eta_0, \frac{1}{i})$ and $\bar \varepsilon_0 = \varepsilon'_0 / 100$. Then if $\Mb_{g}\circ \Phi^*_i(x) \geq L - \bar{\varepsilon}_0$, by (ii), (iii) and Lemma \ref{lem:far_of_diff_geom}, we have
        \[
            |\Phi^*_i(x)| \in \tilde{\cG}_{g} \cup \tilde{\cB}_{g}\,.
        \]
        In addition, $(\Phi^*_i)^\infty_{i = 1}$ is a minimizing sequence.

        Finally, applying Corollary \ref{cor:PT-seq}, we obtain a pulled-tight minimzing sequence $(\Psi_i)^\infty_{i = 1}$ such that 
        \[
            \mathbf{C}_{g}((\Psi_i)^\infty_{i = 1}) \subset (\tilde{\cG}_{g} \cup \tilde{\cB}_{g}) \cap \mathcal{SV}^{L}_{g}\,.
        \]

        The second part follows from the compactness of the critical set and the openness of $\tilde{\cG}_{g} \cup \tilde{\cB}_{g}$.
    \end{proof}

    Let $(\Phi_i)$ be a pulled-tight minimizing sequence from the previous Proposition. Since $\Fb_{g}(\tilde{\cG}_{g}, \tilde{\cB}_{g}) > 0$, for each $X_i$, we can refine $X_i$ to obtain two pure cubical subcomplex $X^B_i$ and $X^G_i$ with the following properties.
    \begin{enumerate}[(i)]
        \item $X_i = X^B_i + X^G_i$;
        \item On $X^B_i$, $\Mb_{g}\circ \Phi_i(x) \geq L - \bar{\varepsilon}_0 \implies |\Phi_i(x)| \in \tilde{\cB}_{g} \cap \bB^{\bF_{g}}_{\eta_0} (\mathcal{SV}^{L}_{g})$;
        \item On $X^G_i$, $\Mb_{g}\circ \Phi_i(x) \geq L - \bar{\varepsilon}_0 \implies |\Phi_i(x)| \in \tilde{\cG}_{g} \cap \bB^{\bF_{g}}_{\eta_0} (\mathcal{SV}^{L}_{g})$;
        \item For all $x \in \partial X^B_i$,
        \[
            \Mb_{g}\circ \Phi_i(x) \leq L - \bar{\varepsilon}_0\,.
        \]
    \end{enumerate}
    Define $\Phi^B_i \coloneqq \Phi_i \vert_{X^B_i}$ and $\delta = \bar\varepsilon_0 / 10$. 
    
    Now, there are two cases. The first case is that there exists $\delta_0 > 0$ such that
    \[
        \limsup_i \bL_g(\HC^\delta_{g}(\Phi^B_{i})) < L - \delta_0\,.
    \]
    In this case, for sufficiently large $i$, we obtain a $\Psi_i \in \tilde{\HC}^\delta_{g}(\Phi^B_{i})$ with domain $W_i$ such that
    \begin{equation}\label{eqn:less_than_L}
        \sup_{x \in W^\omega_i} \Mb_g \circ \Psi^\omega_i(x) < L - \delta_0\,.
    \end{equation}
    Therefore, we can define $X'_i \coloneqq X^G_i \cup W^\omega_i$ and $\Phi'_i: X'_i \to \cZ_n(M; \bF_g; \Z_2)$ by
    \[
        \Phi'_i(x) \coloneqq \begin{cases}
            \Phi_i(x) & x \in X^G_i\\
            \Psi^\omega_i(x) & x \in W_i^\omega\,.
        \end{cases}
    \]
    Clearly, $\Phi'_i \in \HC(\Phi)$. Moreover, by \eqref{eqn:less_than_L} and Property (iii) for $\Phi_i$, for every $x \in X'_i$,
    \[
        \Mb_{g}\circ \Phi'_i(x) \geq L - \min(\delta_0, \bar \varepsilon_0) \implies |\Phi'_i(x)| \in \tilde{\cG}_{g}\,,
    \]
    and thus, $\bC(\Phi'_i) \subset \tilde{\cG}_{g}$. By Lemma \ref{lem:cycle_of_small_rep}, it is a pulled-tight minimizing sequence such that every embedded cycle in its critical set is associated with a flat cycle. This concludes Theorem \ref{thm:strong_multi_one_II}.

    The second case would be
    \begin{equation}\label{eqn:eq_L}
        \limsup_i \bL_{g}(\HC^\delta_{g}(\Phi^B_{i})) = L\,,
    \end{equation}
    and we shall deduce a contradiction from this by proceeding with the proof of Theorem \ref{thm_13_width}. It suffices to consider the boundary part 
    \[
        Z_i \coloneqq \partial X^B_i
    \] when attaching a homotopy map. Fortunately, due to the small mass of cycles on the boundary, they have no significant impact on the critical set. This will be elaborated upon in the following.

\medskip
\paragraph*{\textbf{Part 1. Metric perturbations.}}

    We make use of the same Proposition \ref{prop_baire_residual} to choose a sequence $(g_i)^\infty_{i = 1}$ in $\Gamma^\infty_{\text{uniq}}$ such that $\|g_i - g\|_{C^\infty, g} < \eta_0$, and
    \[
        \lim_{i \to \infty} g_i = g
    \]
    in the $C^\infty$ topology. We define for each $i \in \N^+$,
    \[
        L_i \coloneqq \sup_{x \in X^B_i} \Mb_{g_i} \circ \Phi^B_{i}(x)\,.
    \]
    Let $(\delta_i)^\infty_{i = 1}$ be a decreasing sequence in $(0, \bar\varepsilon_0 / 2)$ such that 
    \[
        \lim_{i \to \infty}\delta_i = 0\,.
    \]

\medskip
\paragraph*{\textbf{Part 2. Restrictive homological min-max.}}
    We define
    \[
        \eta_2 \coloneqq \min(\eta_0, \varepsilon_0) / 10\,.
    \]

    For each $i \in \mathbb{N}^+$, we consider the restrictive homology class 
    \[
        \tilde\HC_i=\tilde \HC^{\delta_i}_{g_i}(\Phi^B_i)\,, \quad \HC_i=\HC^{\delta_i}_{g_i}(\Phi^B_i)\,.
    \]
    
    By \eqref{eqn:eq_L},
    \[
        \lim_{i \to \infty} \Lb(\HC_i) = \lim_{i \to \infty} L_i + \delta_i = L\,.
    \] 
    
    Similarly,  applying the restrictive min-max theorem, Theorem \ref{thm_restrictive_holo_min_max}, to each $\HC_i$, for each $i$, we obtain a sweepout $\Psi_i : W^B_i \to \Zc_n(M; \Fb_{g_i};\Z_2)$ in $\tilde\HC_i$ and $\varepsilon_{2, i} > 0$, such that
    \[
        \mathbf{M}_{g_i}((\Psi^B_i)^\alpha(x)) \geq  \mathbf{L}_{g_i}(\HC_i) - \varepsilon_{2,i}\implies|(\Psi^B_i)^\alpha(x)| \in \tilde{\cB}_{g} \cap \bB^{\bF_{g}}_{\eta_2} (\mathcal{SV}^{L}_{g})\,,
    \]
    and
    \[
        \mathbf{M}_{g_i}((\Psi^B_i)^\omega(x)) \geq  \mathbf{L}_{g_i}(\HC_i) - \varepsilon_{2,i}\implies|(\Psi^B_i)^\omega(x)| \in \tilde{\cG}_{g} \cap \bB^{\bF_{g}}_{\eta_2} (\mathcal{SV}^{L}_{g})\,.
    \]

\medskip
\paragraph*{\textbf{Part 3. Pull-tight.}}
    We define 
    \[
        \eta_3 \coloneqq 3\eta_2\,.
    \]
    
    For each $i$, after possibly discarding finitely many $i$, we define the space
    \[
        \hat W^B_i \coloneqq ([0,1] \x X^B_i) \cup W^B_i\,,
    \]
    where $\{1\}\x X^B_i\subset [0,1]\x W^B_i$ and $X^B_i\subset W^B_i$ are identified. Note that in this case,
    \[
        (\hat W^B_i)^\alpha = \{0\} \times X^B_i \cong X^B_i, \quad (\hat W^B_i)^\omega = [0, 1] \times Z_i \cup (W^B_i)^\omega\,.
    \]
    We can apply the same Proposition \ref{prop_pull_tight} with $A = L + 1$ to obtain a deformation map $H$, and define
    \[
        \hat\Psi^B_i:\hat W^B_i\to \cZ_n(M;\bF_{g_i};\Z_2)
    \] by
    \begin{align*}
        \hat\Psi^B_i|_{[0,1]\x X^B_i}(t, x)&\coloneqq H(t, \Phi^B_{i}(x))\\
        \hat\Psi^B_i|_{W^B_i}(x)&\coloneqq H(1,\Psi^B_{i}(x)).
    \end{align*}

    Then, there exists an $\bar\varepsilon_{3} \in \R^+$ and $(\varepsilon_{3, i})^\infty_{i = 1} \subset \R^+$ such that
    \begin{align*}
        \mathbf{M}_{g_i}((\hat \Psi^B_i)^\alpha(x)) \geq  \mathbf{L}_{g_i}(\HC_i) - \bar\varepsilon_{3}&\implies|(\hat \Psi^B_i)^\alpha(x)| \in \bB^{\bF_{g}}_{\eta_3} (\mathcal{SV}^{L}_{g})\,,\\
        \mathbf{M}_{g_i}((\hat \Psi^B_i)^\alpha(x)) \geq  \mathbf{L}_{g_i}(\HC_i) - \varepsilon_{3,i}&\implies|(\hat \Psi^B_i)^\alpha(x)| \in \tilde{\cB}_{g} \cap \bB^{\bF_{g}}_{\eta_2} (\mathcal{SV}^{L}_{g})\,,\\
        \mathbf{M}_{g_i}((\hat \Psi^B_i)^\omega(x)) \geq  \mathbf{L}_{g_i}(\HC_i) - \varepsilon_{3,i}&\implies|(\hat \Psi^B_i)^\omega(x)| \in \tilde{\cG}_{g} \cap \bB^{\bF_{g}}_{\eta_2} (\mathcal{SV}^{L}_{g})\,.
    \end{align*}
    Here, we use the inequality on $(\hat W^B_i)^\omega$,
    \[
        \sup_{[0, 1] \times Z_i} \mathbf{M}_{g_i} \circ (\hat \Psi^B_i)^\omega \leq \mathbf{L}_{g_i}(\HC_i) - \bar{\varepsilon}_0 / 2\,.
    \]

\medskip
\paragraph*{\textbf{Part 4. $(\varepsilon, \delta)$-deformation.}}
    Fix such a sufficiently large $i$. In $(M, g_i)$, we apply Lemma \ref{prop:M_cts_approx_F_cts} to ${\hat{\Psi}}^B_i: \hat{W}^B_i \to \Zc_n(M; \Fb_{g_i}; \Z_2)$ 
    and obtain an $\Mb_{g_i}$-continuous map ${{\hat{\Psi}}^{B**}_i}$ and an
    $\bF_{g_i}$-continuous homotopy map ${\hat{H}}_i$.
    
    Similar arguments imply that we can apply Lemma \ref{lem:(e, d)-deformation} to ${\hat{\Psi}^{B**}_i}$ to obtain 
    \[
        {\hat \Psi}^{B*}_i: {\hat{W}}^B_i \to \Zc_n(M; \bM_g; \Z_2)\,,
    \]
    and a homotopy map
    \[
        \hat H^\textsc{DEF}_i: [0, 1] \times \hat{W}^B_i \to \Zc_n(M; \bM_g; \Z_2)\,.
    \]

    Analogously, we define the space 
    \[
        \hat{\hat W}^B \coloneqq ([0, 2] \times X^B_i) \cup \hat W^B_i
    \]
    where $\{2\} \times X^B_i \subset [0, 2] \times X^B_i$ and $(\hat W^B_i)^\alpha \subset \hat W^B_i$ are identified. On $\hat{\hat W}^B$, we define the map
    \[
        \hat{\hat \Psi}: \hat{\hat W}^B \to \Zc_n(M; \Fb_{g_i}; \Z_2)
    \]
    in $\tilde{\mathcal{H}}^i$, as follows
    \begin{align*}
        \hat{\hat \Psi}\vert_{[0, 1] \times X^B_i}(t, x) &\coloneqq \hat H_i\vert_{[0, 1] \times  X^B_i}(t, x),\\
        \hat{\hat \Psi}\vert_{[1, 2] \times X^B_i}(t, x) &\coloneqq \hat H^\textsc{DEF}_i\vert_{[0, 1] \times  X^B_i}(t - 1, x),\\
        \hat{\hat \Psi}\vert_{\hat W^B_i} &\coloneqq \Psi^*_i\,.
    \end{align*}
    Note that $(\hat{\hat W}^B)^\alpha = X^B_i$ and $(\hat{\hat W}^B)^\omega = [0, 2] \times Z_i \cup (\hat W^B_i)^\omega$.
    
    Moreover, there exists $\varepsilon_4 > 0$ such that 
    \begin{align*}
        \forall x \in \hat {\hat W}^B,\ \mathbf{M}_{g_i}(\hat{\hat \Psi}(x)) \geq  \mathbf{L}_{g_i}(\HC_i) - \varepsilon_4 &\implies|\hat{\hat \Psi}(x)| \in \hat{\cG}_{g} \cup \hat{\cB}_{g} \,,\\
        \forall x \in (\hat {\hat{W}}^B)^\alpha,\ \mathbf{M}_{g_i}(\hat{\hat \Psi}(x)) \geq  \mathbf{L}_{g_i}(\HC_i) - \varepsilon_{4}&\implies|\hat{\hat \Psi}(x)| \in \hat{\cB}_{g}\,,\\
        \forall x \in (\hat {\hat{W}}^B)^\omega,\ \mathbf{M}_{g_i}(\hat{\hat \Psi}(x)) \geq  \mathbf{L}_{g_i}(\HC_i) - \varepsilon_{4}&\implies|\hat{\hat \Psi}(x)| \in  \hat{\cG}_{g}\,,
    \end{align*}
    Here, we use the inequality on $(\hat {\hat{W}}^B)^\omega$,
    \[
        \sup_{[0, 2] \times Z_i} \mathbf{M}_{g_i} \circ \hat {\hat \Psi}^\omega \leq \mathbf{L}_{g_i}(\HC_i) - \bar{\varepsilon}_0 / 4\,.
    \]

\paragraph*{\textbf{Part 5. Constructing a map $\Xi^\omega$ homologous to $(\hat{\hat{\Psi}})^\omega$}}
    
    Similar to Part 5 in \S \ref{sect:proof_main_thm_I}, using the fact that $\Fb_{g_i}(\hat{\cG}_{g}, \hat{\cB}_{g}) > 0$, we can construct a map $\Xi^B\in\tilde\HC_i$ from $\hat{\hat{\Psi}}^B$ such that
    \[
        \bM_{g_i}\circ(\Xi^B)^\omega<\bL_{g_i}(\HC_i),
    \]
    thereby arriving at a contradiction. This completes the proof of Theorem \ref{thm:strong_multi_one_II}.
      
\section{Technical ingredients}\label{sect:technical}

\subsection{Proof of Proposition \ref{prop_pull_tight}}\label{proof_prop_pull_tight}

    Since $\mathcal{U}^{\leq A}= \mathcal{V}^{\leq A} \setminus \bB^{\bF_{g}}_{\eta_3}(\mathcal{S}^{\leq A})$ (defined in (\ref{eq:U})) is compact, there exists $\bar\varepsilon_{3, 1} > 0$, such that
    \[
        \inf_{V \in \mathcal{U}^{\leq A}} \|\delta_{g} V\|_{g} > 4\bar\varepsilon_{3, 1}\,,
    \]
    where 
    \[
        \|\delta_{g} V\|_{g} \coloneqq \sup\{\delta_{g}V(Y): Y \in \Gamma TM, \|Y\|_{C^1, g} \leq 1\}\,,
    \]
    with $\Gamma TM$ denoting the set of smooth vector fields on $M$.
      
    Furthermore, we can choose
    \begin{itemize}
        \item a positive integer $q$,
        \item varifolds $\{V_j\}^q_{j = 1} \subset \mathcal{U}^{\leq A}$,
        \item radii $\{r_j\}^q_{j = 1} \subset (0, \eta_3)$,
        \item smooth vector fields $\{X_j\}^q_{j = 1} \subset \Gamma TM$ with $\|X_j\|_{C^1, g} \leq 1$,
        \item open balls $\{\mathcal{B}_j \coloneqq \bB^{\mathbf{F}_{g}}_{r_j}(V_j)\}^q_{j = 1}$ with $r_j < \eta_3 / 2$,
    \end{itemize}
    such that
    \begin{align}\label{eq_epsilon31}
        \nonumber \mathcal{U}^{\leq A} &\subset\tilde{\mathcal{U}}^{\leq A}\coloneqq \bigcup^q_{j = 1} \mathcal{B}_j\,, \\
        \nonumber \emptyset &= \bB^{\Fb_{g}}_{\eta_3 / 2}(\mathcal{S}^{\leq A}) \cap \tilde{\mathcal{U}}^{\leq A}\,,\\
        \delta_{g} V(X_j) & \leq -\frac{1}{2}\|\delta_{g} V_j\| < -2\bar\varepsilon_{3, 1} < 0,\ \forall V \in \mathcal{B}_j, \text{ and } j = 1, 2, \cdots, q\,.
    \end{align}
    Hence, without loss of generality, by possibly discarding finitely many $i$'s, for each $i \in \mathbb{N}^+$ and each $j = 1, 2, \cdots, q$, we have 
    \begin{equation}\label{eq_epsilon31i}
        \delta_{g_i} V(X_j) < -\bar\varepsilon_{3, 1}  < 0\,,
    \end{equation}
    holds for all $V \in \mathcal{B}_j$.
      
    For each $j = 1, 2, \cdots, q$, we define a continuous function
    \[
        \psi_j: \tilde{\mathcal{U}}^{\leq A} \to [0, \infty)\,,
    \]
    by defining $\psi_j(V)\coloneqq\Fb_{g}(V,\mathcal{V}^{\leq A} \setminus \mathcal{B}_j)$, the $\mathbf{F}_{g}$-distance of $V$  from $ \mathcal{V}^{\leq A} \setminus \mathcal{B}_j$.
    
    Now, we can define vector fields associated with each varifold in $ \tilde{\mathcal{U}}^{\leq A}$, i.e.
    \begin{align*}
        X: \tilde{\mathcal{U}}^{\leq A} &\to \Gamma TM\,,\\
        V &\mapsto \sum^q_{j = 1} \psi_j(V) X_j\,.
    \end{align*}
    Note that $X$ is continuous, and by compactness, there exists $\bar\varepsilon_{3, 2} > 0$  such that for all $V \in \mathcal{U}^{\leq A}$,
    \begin{equation}\label{eq_epsilon32}
        \delta_{g} V(X(V)) < -2\bar\varepsilon_{3,2} < 0\,. 
    \end{equation}
    Then, by discarding finitely many $g_i$, for each $i \in \mathbb{N}^+$,     \begin{equation}\label{eq_epsilon32i}
        \delta_{g_i} V(X(V)) < -\bar\varepsilon_{3,2} < 0\,, 
    \end{equation}
    We can extend $X$ such that $X$ becomes a continuous map $\cV^{\leq A}\to\Gamma TM$ by putting $X(V)$ to be the zero vector field for each $V$ outside $\tilde \cU^{\leq A}$.  

    Now, we are ready to define the desired map       
    \[
        H:[0,1]\x \{T\in\cZ_n(M;\bF_{g};\Z_2):\bM_{g}(T)\leq A\}\to \{T\in\cZ_n(M;\bF_{g};\Z_2):\bM_{g}(T)\leq A\}.
    \]
    First, we define a map
    \[
        f:[0,\infty)\x \mathcal{V}^{\leq A}\to\mathrm{Diff}(M)
    \]
    by letting $\{f(t,V)\}_t$ be the one-parameter family of diffeomorphisms on $M$ generated by the vector field $X(V)$. By (\ref{eq_epsilon32}) and (\ref{eq_epsilon32i}),  there exists a continuous function $h:\cV^{\leq A}\to [0,1]$ such that:
    \begin{itemize}
        \item $h>0$ on $\tilde\cU^{\leq A}$ and $h=0$ elsewhere.
        \item If $0\leq t<s\leq h(V)$,
        \begin{equation}\label{eq_mass_decrease_under_diffeo_gbar}
            \|f(s,V)_\#(V)\|_{g}(M)<\|f(t,V)_\#(V)\|_{g}(M).
        \end{equation}
        \item For each $i$, if $0\leq t<s\leq h(V)$, \begin{equation}\label{eq_mass_decrease_under_diffeo_gi}
            \|f(s,V)_\#(V)\|_{g_i}(M)<\|f(t,V)_\#(V)\|_{g_i}(M).
        \end{equation}
    \end{itemize}
    Now, define the desired map $H$ by
    $$H(t,T) = \begin{cases}
  f(t,|T|)_\#(T)  & \textrm{ if } 0\leq t\leq h(|T|), \\
  f(h(|T|),|T|)_\#(T) &  \textrm{ if } h(|T|)\leq t\leq 1.
    \end{cases}$$

    Clearly, $H(0,\cdot)=\id$, so  $H$ satisfies Proposition \ref{prop_pull_tight} (\ref{prop_item_start_id}). And because $h=0$ on $\bB^{\Fb_{g}}_{\eta_3 / 2}(\mathcal{S}^{\leq A})$, $H(t,\cdot)$ fixes $T$ if $|T| \in \bB^{\Fb_{g}}_{\eta_3 / 2}(\mathcal{S}^{\leq A})$, so Proposition \ref{prop_pull_tight} (\ref{prop_item_fix_stationary}) holds. Moreover, by the three bullet points  in the definition of $h$ above,   we know for each $(t,T)$,
    \[
        \bM_{g}(H(t,T))\leq \bM_{g}(T),\;\;\bM_{g_i}(H(t,T)) \leq \bM_{g_i}(T)\,.
    \]
    And any of the equalities hold only if $t=0$ or $h(|T|)=0$.
    So $H$ satisfies  Proposition \ref{prop_pull_tight}  (\ref{prop_item_decrease_mass}) too.
    Furthermore, we claim that there exists $\varepsilon_3 > 0$, such that for each $T$, if $|H(1,T)|\in\cU^{\leq A}$, then
    \[
        \bM_{g}(H(1,T)) \leq \bM_{g}(T) - 2\varepsilon_3\,.
    \]
    Indeed, if not, then there exists a sequence $(T_j)_j$ such that $|H(1,T_j)|\in\cU^{\leq A}$ and 
    \begin{equation}\label{eq_MH}
        \bM_{g}(H(1,T)) \geq  \bM_{g}(T) - 1/j.
    \end{equation}
    Then, by compactness and relabeling the $j$'s,  $|T_j|$  converges to some $V'\in\cV^{\leq A}$. However, note that:
    \begin{itemize}
        \item Putting $T=T_j$ into (\ref{eq_MH}) and taking $j\to\infty$, we have 
        \[
            \|f(h(V'),V')(V')\|(M)\geq \|V'\|(M)\,,
        \]
        which by (\ref{eq_mass_decrease_under_diffeo_gbar}) implies $h(V')=0$. Thus, by the definition of $h$, $V'\notin\tilde\cU^{\leq A}$, so $V'\notin\cU^{\leq A}$.
        \item On the other hand, taking $j\to\infty$  to $|H(1,T_j)|\in\cU^{\leq A}$, we know 
        \[
            f(h(V'),V')_\#(V')\in\cU^{\leq A}\,.
        \] But we have shown $h(V')=0,$ so $V'\in \cU^{\leq A} $. 
    \end{itemize}
    Hence, a contradiction arises. This proves our claim that if $|H(1,T)|\in\cU^{\leq A}$, then
    \[
        \bM_{g}(H(1,T)) \leq \bM_{g}(T) - 2\varepsilon_3\,.
    \]
    Then, we also have  
    \[
        \bM_{g_i}(H(1,T)) \leq \bM_{ g_i}(T) - \varepsilon_3
    \]
    by discarding finitely many $g_i$. Therefore, $H$ satisfies Proposition \ref{prop_pull_tight} (\ref{prop_item_mass_fixed_amount_below}).  
    
    This finishes the proof of Proposition \ref{prop_pull_tight}.
\subsection{Proof of Lemma \ref{lem_critical_set}}\label{sect_lem_critical_set_proof}
    Note that:
    \begin{itemize}
        \item $\lim_i \sup \Mb_{g}\circ \Phi_{i}= L$, which by Proposition \ref{prop_pull_tight} (\ref{prop_item_decrease_mass}) implies 
        \[
            \lim_i \sup \Mb_{g}\circ \hat{\Psi}_i|_{[0,1]\x X_i} = L\,.
        \]
        \item By $\Lb_{g_i}(\mathcal{H}_i)\to L$, $\lim_i (\sup \Mb_{g}\circ \Phi_{i}+\delta_i) = L $, and Proposition \ref{prop_pull_tight} (\ref{prop_item_decrease_mass}), we know 
        \[
            \lim_i \sup\Mb_{g}\circ \hat{\Psi}_i|_{W_i}= L\,.
        \]
    \end{itemize}
    Hence, 
    \[
        \mathbf{C}_{g}((\hat{\Psi}_i)_{i }) =\bC_{g}((\hat{\Psi}_i|_{[0,1]\x X_i})_{i})\cup \bC_{g}((\hat{\Psi}_i|_{W_i})_i)\,,
    \]
    where the right hand side uses Notation \ref{nota:critical} (as we do not really view $(\hat{\Psi}_i|_{[0,1]\x X_i})_{i }$ and $(\hat{\Psi}_i|_{W_i})_i$ as minimizing sequences).

    To see that $\bC_{g}((\hat{\Psi}_i|_{[0,1]\x X_i})_{i })\subset \bB^{\bF_{g}}_{\eta_3}(\mathcal{SV}^{L}_{g})$, take $(t_i,x_i)\in [0,1]\x X_i$ such that, after passing to a subsequence, \[
        |\hat{\Psi}_i|_{[0,1]\x X_i}(t_i,x_i)|
    \] 
    tends to some varifold $V$ with mass $L$. Then by Proposition \ref{prop_pull_tight} (\ref{prop_item_decrease_mass}) we know the mass of 
    \[
        \hat{\Psi}_i|_{[0,1]\x X_i}(0,x_i)=\Phi_{i}(x_i)
    \]
    tends to $L$. Since $(\Phi_{i})_i$ is a pulled-tight minimizing sequence, we know $\Phi_{i}(x_i)$ subsequentially converge to some stationary varifold of mass $L$. By Proposition \ref{prop_pull_tight} (\ref{prop_item_fix_stationary}), this stationary varifold is $V$. So $\bC_{g}((\hat{\Psi}_i|_{[0,1]\x X_i})_{i })\subset \bB^{\bF_{g}}_{\eta_3}(\mathcal{SV}^{L}_{g})$.

    To see that $\bC((\hat{\Psi}_i|_{W_i})_i)\subset \bB^{\bF_{g}}_{\eta_3}(\mathcal{SV}^{L}_{g})$, suppose by contradiction that, after passing to a subsequence, $|\hat{\Psi}_i|_{W_i}(x_i)|$ converges to some $V$ with mass $L$ but $V\notin \bB^{\bF_{g}}_{\eta_3}(\mathcal{SV}^{L}_{g})$. Then for every large $i$, by Proposition \ref{prop_pull_tight} (\ref{prop_item_mass_fixed_amount_below}),  $\bM_{g}(\hat\Psi_i(x_i)) \leq \bM_{g}(\Psi(x_i)) - \varepsilon_3$. However, noting 
    \[
        \lim_i\sup \Mb_{g}\circ\Psi_i=L\,,
    \] contradiction arises. This finishes the proof.
\subsection{Proof of Claim \ref{claim1}}
    We need to prove that,
        for any $x \in \hat W_i$, if $\Mb_{g_i}(\hat \Psi^*_i(x)) \geq \mathbf{L}_{g_i}(\HC_i) - \bar \varepsilon / 100$, then $|\hat \Psi^*_i(x)| \in \hat \cG_{g} \sqcup \hat \cB_{g}$

  First,
        $\Mb_{g_i}(\hat \Psi^*_i(x)) \geq \mathbf{L}_{g_i}(\HC_i) - \bar \varepsilon / 100$ implies $\Mb_{g_i}(\hat \Psi^*_i(x)) \geq L - \bar\varepsilon /10$. 
        
        By (2) and (3) of Lemma \ref{lem:(e, d)-deformation}, there exists $T_{1, x} \in \Zc_n(M; \Z_2)$ and $\hat x \in \hat W_i$ satisfying
        \begin{enumerate}[(i)]
            \item $\Mb_{g_i}(\hat \Psi^*_i(x), T_{1, x}) < \bar \delta$, and thus, by Property (iii) of $\bar \delta$, $\Fb_{g_i}(\hat \Psi^*_i(x), T_{1, x}) < d_0$;
            \item \[T_{1, x}\llcorner(M \setminus(\overline{B}_{g_i}(p_1, \eta) \cup \cdots \cup (\overline{B}_{g_i}(p_m, r_0))) = \hat \Psi'_i(\hat x)\llcorner(M \setminus(\overline{B}_{g_i}(p_1, r_0) \cup \cdots \cup (\overline{B}_{g_i}(p_m, r_0)))\]
            for some collection $\{p_1, \cdots, p_m\} \subset M$, $m \leq 3^{2k + 3}$,
            \item $\Fb_{g_i}(|\hat \Psi'_i(\hat x)|, \cB_{g} \cup \cG_{g}) \leq 2d_0$.
        \end{enumerate}

        Since $d_0 = \eta_{\ref{lem:far_of_diff_geom}}(M, g, 2k + 3, L, \cG_{g}, \cB_{g}) / 10$, by Lemma \ref{lem:far_of_diff_geom} with $V_j = |\hat \Psi'_i(\hat x)|$, $V'_j = |T_{1, x}|$ and $V''_j = |\hat \Psi^*_i(x)|$ for $j = 1$ or $2$, we have
        \[
            |\hat \Psi^*_i(x)| \in \hat \cG_{g} \sqcup \hat \cB_{g}\,.
        \]

\subsection{Proof of Claim \ref{claim2}}

    We need to prove that,
        for any $x \in \hat W^\alpha_i$ and any $t \in [0, 1]$, if $\Mb_{g_i}(\hat H^\textsc{DEF}_i(t, x)) \geq \mathbf{L}_{g_i}(\HC_i) - \varepsilon_{4, i} / 4$, then $|\hat H^\textsc{DEF}_i(t, x)| \in \hat \cB_{g}$.
 
        By (2) of Lemma \ref{lem:(e, d)-deformation}, there exists $T_{t, x} \in \Zc_n(M; \Z_2)$ such that
        \begin{enumerate}[(i)]
            \item $\Mb_{g_i}(\hat \Psi'_i(x), \hat \Psi'_i(\hat x)) < \bar\delta$ 
            \item $\Mb_{g_i}(\hat \Psi'_i(\hat x)) > \Mb_{g_i}(\hat H^\textsc{DEF}_i(t, x)) - \bar \delta$;
            \item $\Mb_{g_i}(\hat H^\textsc{DEF}_i(t, x), T_{t, x}) < \bar \delta$;
            \item 
            \[
                T_{t, x}\llcorner(M \setminus(\overline{B}_{g_i}(p_1, \eta) \cup \cdots \cup (\overline{B}_{g_i}(p_m, r_0))) = \hat \Psi'_i(\hat x)\llcorner(M \setminus(\overline{B}_{g_i}(p_1, r_0) \cup \cdots \cup (\overline{B}_{g_i}(p_m, r_0)))
            \]
            for some collection $\{p_1, \cdots, p_m\} \subset M$, $m \leq 3^{2k + 3}$.
        \end{enumerate}

        By (i) and (ii), we have 
        \[
            \Mb_{g_i}(\hat \Psi'_i(x)) > \Mb_{g_i}(\hat \Psi'_i(\hat x)) - \bar\delta > \Mb_{g_i}(\hat H^\textsc{DEF}_i(t, x)) - 2\bar \delta > \mathbf{L}_{g_i}(\HC_i) - \varepsilon_{4, i}\,,
        \]
        and it follows from \eqref{eqn:epsilon_4_i_a} that
        \[
            |\hat \Psi'_i(x)| \in \tilde{\cB}_{g}\,.
        \]
        
        By Property (iii) of $\bar \delta$, (i) and (iii) above imply $\Fb_{g_i}(\hat \Psi'_i(x), \hat \Psi'_i(\hat x)) < d_0$ and $\Fb_{g_i}(\hat H^\textsc{DEF}_i(t, x), T_{t, x}) < d_0$.
        Consequently,
        \[
             |\hat \Psi'_i(\hat x)| \in \bB^{\Fb_{g_i}}_{d_0}(\tilde{\cB}_{g})\,.
        \]

        Since $d_0 = \eta_{\ref{lem:far_of_diff_geom}}(M, g, 2k + 3, L, \cG_{g}, \cB_{g}) / 10$, by Lemma \ref{lem:far_of_diff_geom} with $V_2 = |\hat \Psi'_i(\hat x)|$, $V'_2 = |T_{t, x}|$ and $V''_2 = |\hat H^\textsc{DEF}_i(t, x)|$, we conclude that
        \[
            |\hat H^\textsc{DEF}_i(t, x)| \in \hat \cB_{g}\,.
        \]

    \subsection{Proof of Claim \ref{claim3}}

    We need to prove that,
        for any $x \in \hat W^\omega_i$ and any $t \in [0, 1]$, if $\Mb_{g_i}(\hat \Psi^*_i(x)) \geq \mathbf{L}_{g_i}(\HC_i) - \varepsilon_{4, i} / 4$, then $|\hat \Psi^*_i(x))| \in \hat \cG_{g}$.
 
        By (2) of Lemma \ref{lem:(e, d)-deformation}, there exists $T_{1, x} \in \Zc_n(M; \Z_2)$ such that
        \begin{enumerate}[(i)]
            \item $\Mb_{g_i}(\hat \Psi'_i(x), \hat \Psi'_i(\hat x)) < \bar\delta$ 
            \item $\Mb_{g_i}(\hat \Psi'_i(\hat x)) > \Mb_{g_i}(\hat H^\textsc{DEF}_i(t, x)) - \bar \delta$;
            \item $\Mb_{g_i}(\hat \Psi^*(x), T_{1, x}) < \bar \delta$;
            \item 
            \[
                T_{1, x}\llcorner(M \setminus(\overline{B}_{g_i}(p_1, \eta) \cup \cdots \cup (\overline{B}_{g_i}(p_m, r_0))) = \hat \Psi'_i(\hat x)\llcorner(M \setminus(\overline{B}_{g_i}(p_1, r_0) \cup \cdots \cup (\overline{B}_{g_i}(p_m, r_0)))
            \]
            for some collection $\{p_1, \cdots, p_m\} \subset M$, $m \leq 3^{2k + 3}$.
        \end{enumerate}

        By (i) and (ii), we have 
        \[
            \Mb_{g_i}(\hat \Psi'_i(x)) > \Mb_{g_i}(\hat \Psi'_i(\hat x)) - \bar\delta > \Mb_{g_i}(\hat H^\textsc{DEF}_i(t, x)) - 2\bar \delta > \mathbf{L}_{g_i}(\HC_i) - \varepsilon_{4, i}\,,
        \]
        and it follows from \eqref{eqn:epsilon_4_i_w} that
        \[
            |\hat \Psi'_i(x)| \in \tilde{\cG}_{g}\,.
        \]
        
        By Property (iii) of $\bar \delta$, (i) and (iii) above imply $\Fb_{g_i}(\hat \Psi'_i(x), \hat \Psi'_i(\hat x)) < d_0$ and $\Fb_{g_i}(\hat \Psi^*(x), T_{1, x}) < d_0$.
        Consequently,
        \[
             |\hat \Psi'_i(\hat x)| \in \bB^{\Fb_{g_i}}_{d_0}(\tilde{\cG}_{g})\,.
        \]

        Since $d_0 = \eta_{\ref{lem:far_of_diff_geom}}(M, g, 2k + 3, L, \cG_{g}, \cB_{g}) / 10$, by Lemma \ref{lem:far_of_diff_geom} with $V_1 = |\hat \Psi'_i(\hat x)|$, $V'_1 = |T_{t, x}|$ and $V''_1 = |\hat \Psi^*(x)|$, we conclude that
        \[
            |\hat \Psi^*(x)| \in \hat \cG_{g}\,.
        \]


\printbibliography

\end{document}